\documentclass[11pt]{article}

\usepackage[margin=1.1in,includefoot]{geometry}

\usepackage{comment}
\usepackage{amsfonts}
\usepackage{amssymb}
\usepackage{amsthm}
\usepackage{mathtools}   

\usepackage{hyperref} 
\usepackage[alphabetic]{amsrefs}

\usepackage{csquotes}
\usepackage{footmisc}

\usepackage{xpatch}
\usepackage{longtable}
\usepackage[all,cmtip]{xy}
\usepackage{fancyhdr}
\usepackage[normalem]{ulem}

\usepackage{bm}

\usepackage{graphicx}

\usepackage{tikz-cd}
\usepackage{tikz}
\usepackage{pifont}

\usepackage{aurical}
\usepackage{stmaryrd}
\usepackage[T1]{fontenc}
\usepackage[outline]{contour}
\usepackage{xcolor}
\usepackage{etoolbox}
\usepackage{array}
\usepackage{lipsum}

\usepackage{enumitem}

\usepackage{parskip}

\setlist[enumerate,1]{label={(\alph*)}}

\makeatletter
\DeclareFontEncoding{LS2}{}{\@noaccents}
\makeatother
\DeclareFontSubstitution{LS2}{stix}{m}{n}
\DeclareSymbolFont{largesymbolsstix}{LS2}{stixex}{m}{n}
\DeclareMathDelimiter{\lbrbrak}{\mathopen}{largesymbolsstix}{"EE}{largesymbolsstix}{"14}
\DeclareMathDelimiter{\rbrbrak}{\mathclose}{largesymbolsstix}{"EF}{largesymbolsstix}{"15}

\newcommand{\git}{\mathbin{
  \mathchoice{\mkern-3mu/\mkern-6mu/\mkern-3mu}
    {\mkern-3mu/\mkern-6mu/\mkern-3mu}
    {/\mkern-5mu/}
    {/\mkern-5mu/}}}

\newcommand{\genlegendre}[4]{%
  \genfrac{(}{)}{}{#1}{#3}{#4}%
  \if\relax\detokenize{#2}\relax\else_{\!#2}\fi
}

\let\integral\int\let\int\relax\DeclareMathOperator*{\int}{int}

\newcommand{\what}[1]{{\widehat{#1}}}

\newcommand{\bs}{\backslash}

\newcommand{\lra}{\longrightarrow}
\newcommand{\ra}{\rightarrow}

\newcommand{\bA}{\mathbb{A}}

\newcommand{\bC}{\mathbb{C}}

\newcommand{\bE}{\mathbb{E}} 
\newcommand{\bF}{\mathbb{F}}

\newcommand{\bL}{\mathbb{L}}

\newcommand{\bP}{\mathbb{P}}
\newcommand{\bQ}{\mathbb{Q}}
\newcommand{\bR}{\mathbb{R}}
\newcommand{\bS}{\mathbb{S}}

\newcommand{\bX}{\mathbb{X}}

\newcommand{\bZ}{\mathbb{Z}}

\newcommand{\Qbar}{{\overline{\mathbb{Q}}}}

\newcommand{\cA}{\mathcal{A}}
\newcommand{\cB}{\mathcal{B}}
\newcommand{\cC}{\mathcal{C}}

\newcommand{\cE}{\mathcal{E}}

\newcommand{\cG}{\mathcal{G}}
\newcommand{\cH}{\mathcal{H}}
\newcommand{\cI}{\mathcal{I}}

\newcommand{\cM}{\mathcal{M}}
\newcommand{\cN}{\mathcal{N}}
\newcommand{\cO}{\mathcal{O}}

\newcommand{\cT}{\mathcal{T}}

\newcommand{\cX}{\mathcal{X}}
\newcommand{\cY}{\mathcal{Y}}

\newcommand{\Zhat}{{\widehat{\mathbb{Z}}}}

\newcommand{\la}[1]{\,^{#1}\!} 

\newcommand{\mf}[1]{\mathfrak{#1}} 

\newcommand{\ff}{\mathfrak{f}}

\newcommand{\fA}{\mathfrak{A}}

\newcommand{\fS}{\mathfrak{S}}

\newcommand{\fs}{\!\!\fatslash}

\newcommand{\et}{{\acute{e}t}}

\newcommand{\im}{\operatorname{im }}

\newcommand{\Aut}{\operatorname{Aut}}

\newcommand{\Inn}{\operatorname{Inn}}
\newcommand{\Out}{\operatorname{Out}}
\newcommand{\IOut}{\operatorname{IOut}}

\newcommand{\Ker}{\operatorname{Ker}}

\newcommand{\gen}{{\operatorname{gen}}}

\newcommand{\spmatrix}[4]{\left[\begin{smallmatrix}#1&#2\\#3&#4\end{smallmatrix}\right]}
\newcommand{\ttmatrix}[4]{\left[\begin{array}{rr}#1&#2\\#3&#4\end{array}\right]}

\newcommand{\cAdm}{\mathcal{A}dm}
\newcommand{\Adm}{Adm}

\newcommand{\Frac}{\operatorname{Frac}}

\newcommand{\Gal}{{\operatorname{Gal}}}
\newcommand{\id}{\operatorname{id}}
\newcommand{\Hom}{\operatorname{Hom}}

\newcommand{\Mon}{\operatorname{Mon}}

\newcommand{\Epi}{\operatorname{Epi}}

\newcommand{\GL}{\operatorname{GL}}
\newcommand{\SL}{\operatorname{SL}}
\newcommand{\PSL}{\operatorname{PSL}}

\newcommand{\PSU}{\operatorname{PSU}}

\newcommand{\SU}{\operatorname{SU}}
\newcommand{\U}{\operatorname{U}}

\newcommand{\D}{\operatorname{D}}
\newcommand{\Or}{\operatorname{O}}
\newcommand{\M}{\operatorname{M}}
\newcommand{\J}{\operatorname{J}}
\newcommand{\SO}{\operatorname{SO}}
\newcommand{\Sz}{\operatorname{Sz}}

\newcommand{\Stab}{\operatorname{Stab}}

\newcommand{\Tr}{\operatorname{Tr}}

\newcommand{\ab}{{\operatorname{ab}}}
\newcommand{\meta}{{\textrm{meta}}}
\newcommand{\cyc}{{\operatorname{cyc}}}

\newcommand{\tr}{\operatorname{tr}}

\newcommand{\ord}{\operatorname{ord}}

\newcommand{\rightiso}{\stackrel{\sim}{\longrightarrow}}

\newcommand{\abs}{{\operatorname{abs}}}

\newcommand{\ol}[1]{{\overline{#1}}}
\newcommand{\ul}[1]{{\underline{#1}}}

\newcommand{\an}{\text{an}}

\newcommand{\genus}{\operatorname{genus}}
\newcommand{\Spec}{\operatorname{Spec}}

\newcommand{\pr}{\operatorname{pr}}

\newcommand{\ext}{{\mathrm{ext}}}

\newcommand{\pre}{{\text{pre}}}
\newcommand{\bgs}{{\text{bgs}}}

\newcommand{\ls}[1]{(\!(#1)\!)}
\newcommand{\ps}[1]{[\![#1]\!]}

\newcommand{\Homeo}{\operatorname{Homeo}}

\newcommand{\tp}{{\text{top}}}

\newcommand{\MT}{\operatorname{MT}}

\newcommand{\Burau}{{\text{Burau}}}

\newcommand{\Tate}{\operatorname{Tate}}

\newcommand{\Sch}{\underline{\textbf{Sch}}}

\newcommand{\Sets}{\underline{\textbf{Sets}}}

\newcommand{\FiniteSets}{\underline{\textbf{FiniteSets}}}

\newcommand{\Isom}{\operatorname{Isom}}

\newcommand{\Sym}{\operatorname{Sym}}

\newcommand{\gap}{\vspace{0.4cm}}

\theoremstyle{definition}\newtheorem{defn}{Definition}[section]
\theoremstyle{remark}
\theoremstyle{remark}
\theoremstyle{remark}\newtheorem{remark}[defn]{Remark}
\theoremstyle{plain}\newtheorem{question}[defn]{Question}
\theoremstyle{remark}\newtheorem*{remark*}{Remark}
\theoremstyle{remark}
\theoremstyle{remark}
\theoremstyle{definition}\newtheorem{example}[defn]{Example}
\theoremstyle{definition}
\theoremstyle{remark}
\theoremstyle{definition}

\theoremstyle{plain}
\theoremstyle{plain}
\theoremstyle{plain}\newtheorem{prop}[defn]{Proposition}
\theoremstyle{plain}\newtheorem{thm}[defn]{Theorem}
\theoremstyle{plain}
\theoremstyle{plain}\newtheorem{lemma}[defn]{Lemma}
\theoremstyle{plain}\newtheorem{cor}[defn]{Corollary}
\theoremstyle{plain}\newtheorem{conj}[defn]{Conjecture}
\theoremstyle{plain}\newtheorem*{thm*}{Theorem}
\theoremstyle{plain}\newtheorem*{conj*}{Conjecture}
\theoremstyle{plain}\newtheorem*{prop*}{Proposition}

\theoremstyle{plain}

\newcommand{\note}[1]{{\color{blue} \sf Will: [#1]}}

\title{Noncongruence modular curves as Hurwitz spaces}

\renewcommand{\int}{\operatorname{int}}
\author{William Y. Chen}

\begin{document}
\maketitle

\begin{abstract} In this survey article we give an overview of how noncongruence modular curves can be viewed as Hurwitz moduli spaces of covers of elliptic curves at most branched above the origin. We describe some natural questions that arise, and applications of these ideas to the Inverse Galois Problem, Markoff triples and the arithmetic of Fourier coefficients for noncongruence modular forms.	
\end{abstract}

\tableofcontents

\section{Introduction}
Let $\SL_2(\bZ)$ denote the group of $2\times 2$ matrices with integer coefficients and determinant 1. The group $\SL_2(\bZ)$ acts on the complex upper half plane $\cH$ by m\"{o}bius transformations:
\begin{equation}\label{eq_left_action}
\ttmatrix{a}{b}{c}{d}\tau := \frac{a\tau +b}{c\tau+d}
\end{equation}

To any finite index subgroup $\Gamma\le\SL_2(\bZ)$, the quotient $\Gamma\bs\cH$ is an algebraic curve over $\bC$, called a \emph{modular curve}. Recall that a subgroup $\Gamma\le\SL_2(\bZ)$ is \emph{congruence} if, for some $n\ge 1$, it contains the kernel of the reduction map
$$\Gamma(n) := \Ker(\SL_2(\bZ)\ra\SL_2(\bZ/n))$$
and otherwise it is \emph{noncongruence}. For congruence subgroups $\Gamma\le\SL_2(\bZ)$, the curves $\Gamma\bs\cH$ are defined over number fields, and the study of their arithmetic properties has led to profound breakthroughs in number theory, perhaps most famously the modularity theorem and Fermat's last theorem \cite{CSFLT, DS06, Wiles95}. The theory hints at deep connections between arithmetic geometry, analysis and representation theory, which has blossomed into the Langlands program. By contrast, the study of noncongruence subgroups has received far less attention. The spectacular success of the congruence theory begs the question:
\begin{equation}\label{eq_question}
\parbox{\linewidth-5em}{Could a similarly rich theory exist in the case of \emph{noncongruence subgroups}? More broadly, what is the meaning of these noncongruence subgroups, and how do they fit into the big picture?}
\end{equation}

\subsection{Noncongruence modular forms and the work of Atkin and Swinnerton-Dyer}
The first systematic investigation into such questions was conducted by Atkin and Swinnerton-Dyer (ASD) in \cite{ASD71}. Aided by the fledgling computers of their time, they computed the Fourier expansions for bases of several spaces of modular forms for various noncongruence subgroups $\Gamma\le\SL_2(\bZ)$, and noticed remarkable arithmetic properties, expressed in terms of congruences between Fourier coefficients, which are strikingly similar to those enjoyed by their congruence cousins. At around the same time, it became apparent that while congruence modular forms can always be normalized to have integral coefficients, all computed examples of noncongruence modular forms could not -- they all had ``unbounded denominators''. The discoveries of Atkin and Swinnerton-Dyer triggered a surge of activity \cite{Sch85, Sch88, LLY, HLV, LL11, Kib14, LL14}. Today, both observations of Atkin and Swinnerton-Dyer are now much better understood. In \cite{Katz81} and \cite{Sch85}, Katz and Scholl have shown that these ASD congruences are induced by the action of a crystalline Frobenius on de Rham cohomology, and by a recent breakthrough of Calegari, Dimitrov, and Tang \cite{CDT21}, we now know that the unbounded denominator property completely characterizes noncongruence modular forms.

\subsection{Noncongruence modular curves and nonabelian level structures}\label{ss_level_structures_intro}
In this expository article, we will describe progress towards \eqref{eq_question} which has developed in a somewhat parallel manner to the work described above. While the previous work treated noncongruence subgroups (resp. modular curves) as essentially \emph{arbitrary} subgroups (resp. algebraic curves), in this article we will breathe meaning into noncongruence objects by means of a moduli interpretation \cite{Chen18, Chen21}. To be more precise, we will describe a family of level structures for elliptic curves which simultaneously generalize the classical level structures given in terms of torsion data, and endows every noncongruence modular curve with a moduli interpretation.

The basic idea is as follows. Given an elliptic curve $E$ over $\bC$ with origin $O$, a generalized level structure on $E$ is the isomorphism class of a finite branched covering $\pi : C\ra E$, with branching only above $O$, and unramified elsewhere. If $E^\circ := E - O$ is the punctured elliptic curve, then for any base point $p_0\in E^\circ(\bC)$, $\pi_1(E^\circ(\bC),p_0)$ is a free group of rank 2, and the inclusion $i : E^\circ\hookrightarrow E$ induces a surjection
$$i_* : \pi_1(E^\circ(\bC),p_0)\lra\pi_1(E(\bC),p_0)$$
which identifies $\pi_1(E(\bC),p_0)\cong\bZ^2$ with the abelianization of $\pi_1(E^\circ(\bC),p_0)$. If $x,y$ are generators of $\pi_1(E^\circ(\bC),p_0)$, then their commutator $[x,y]$ is represented by a small loop around the puncture. Monodromy around this loop is then a permutation of the fiber $C_{p_0}$, whose cycles correspond to the ramified points above $O$, the cycle length being equal to the ramification index. It follows that our branched covering $\pi : C\ra E$ is unramified if and only if the monodromy representation
$$\varphi_\pi : \pi_1(E^\circ(\bC),p_0)\lra \Sym(C_{p_0})$$
associated to $\pi$ sends $[x,y]$ to the identity, which is exactly to say that its image, the monodromy group of $\pi$, is abelian. In this case, we say that $\pi$ defines an \emph{abelian level structure}, and the Hurwitz formula implies that $C$ is elliptic and $\pi$ is an isogeny. This shows that abelian level structures recover classical congruence level structures. The ``Hurwitz'' moduli space obtained by deforming the complex structure on $\pi$ (see \S\ref{sss_Teichmuller_uniformization}) is then naturally a congruence modular curve.

When $\pi$ has nonabelian monodromy, it must be ramified, and the associated moduli spaces can be calculated explicitly in terms of the monodromy representation $\varphi_\pi$ (see \S\ref{sss_combinatorial_decomposition}). In this case we say $\pi$ defines a \emph{nonabelian level structure}. By results of Asada \cite{Asa01} and Ellenberg-McReynolds \cite{EM12}, it follows that

\begin{thm}[Asada, Ellenberg-McReynolds, see \S\ref{ss_universality}]\label{thm_ACEM} Every noncongruence modular curve appears as a quotient of the moduli space of an appropriately chosen $\pi$.\footnote{The appropriate cover for a given modular curve can be found explicitly by the results of \cite[\S5]{BER11}. On the other hand, it can be remarkably difficult to recognize when a cover $\pi$ will result in a noncongruence modular curve, see \S\ref{ss_noncongruence_nonabelian}.}
\end{thm}

In this paper, we will focus on the case where $\pi$ is Galois with Galois group $G$, in which case the isomorphism class of $\pi$ is a \emph{$G$-structure} on $E$. The moduli space of elliptic curves with $G$-structures (see \S\ref{section_nonabelian}) will be denoted $M(G)$, and is a smooth affine curve over $\bZ[1/|G|]$. Over $\bC$, this space is typically disconnected: its components are modular curves, and every modular curve is a quotient of a component of $M(G)_\bC$ for an appropriate finite group $G$. In particular, these moduli interpretations give rise to integral models for noncongruence modular curves. Via the theory of admissible covers \cite{ACV03}, in \S\ref{ss_compactification} we will describe the extension of this moduli interpretation to the compactifications of modular curves.

These moduli interpretations indicate that while congruence modular curves reflect the properties of elliptic curves as abelian varieties, noncongruence modular curves capture their essence as affine hyperbolic curves (after removing the origin)\footnote{Hyperbolic curves are also examples of \emph{anabelian} varieties. This leads to the slogan that congruence subgroups are to elliptic curves as abelian varieties as noncongruence subgroups are to elliptic curves as anabelian curves.}. This dual nature of elliptic curves is also reflected in the group $\SL_2(\bZ)$, which is simultaneously an arithmetic group and a mapping class group. As an arithmetic group, one is naturally led to consider its action on the homology $H_1(E,\bZ)\cong\pi_1(E)\cong\bZ^2$ and the extended action of $\SL_2(\bR)$ on lattices in $\bR^2$, a context where Hecke operators naturally arise. As the mapping class group of a punctured torus $E^\circ$, one is instead led to consider its outer action on the nonabelian fundamental group $\pi_1(E^\circ)$, whence nonabelian covers, whence Hurwitz spaces. In this context, the $\SL_2(\bR)$-action on lattices should be translated into its action on the space of complex structures on $E$ via \emph{affine deformations} (see \S\ref{sss_EM}). Here, Hecke operators are less natural, and are in some sense ineffective: By a result of Berger \cite{Berg94}, the action of the Hecke operators $T_p$ on the space of modular forms for a noncongruence subgroup $\Gamma$ factors through its action on the subspace of forms for the congruence closure $\Gamma^c$ of $\Gamma$.

While the abelian theory has been wildly successful at understanding the arithmetic and geometry of elliptic curves, largely due to the Hecke theory and the connection with modular forms, the nonabelian theory carries more information. For example, by the Grothendieck anabelian conjectures for hyperbolic curves\footnote{Proved by Mochizuki and Tamagawa \cite{Tam97, Moc96}.}, if $E$ is an elliptic curve over a number field $K$, then the (outer) $\Gal(\Qbar/K)$ action on $\pi_1(E^\circ)$ is faithful and completely determines $E$ up to $K$-isomorphism, whereas the action on the abelianization $\pi_1(E)$ is not faithful and only determines $E$ up to $K$-isogeny. Moreover, the universality of the moduli interpretations may lead to novel approaches to solving generalized Fermat equations via Darmon's dictionary \cite[Principle 6]{Dar04}. On the other hand, the meat of such an approach would lie in the determination of rational points on noncongruence modular curves, or equivalently rational nonabelian level structures on elliptic curves; for this, significant additional arithmetic input would be required.

\subsection{Some fundamental questions}
By Belyi's theorem, every smooth proper algebraic curve defined over a number field is isomorphic to the compactification of some noncongruence modular curve. It follows that one should not expect to be able to say much about noncongruence modular curves in general. A main virtue of the moduli interpretations is that it provides a way to systematically split up the family of all modular curves into subfamilies, each of which may individually be amenable to a more productive investigation. For example, the family classifying abelian covers gives rise to the congruence modular curves.

In \S\ref{ss_metabelian} we will explain how the modular curves classifying metabelian covers are also congruence. This shows that the naive guess that nonabelian monodromy should lead to noncongruence moduli spaces fails, and leads to the fundamental and remarkably subtle problem of determining for which covers $\pi : C\ra E$ the associated moduli space is congruence or noncongruence (\S\ref{ss_noncongruence_nonabelian}). Understanding this would also shed light on just how much the congruence theory tells us about the nonabelian \'{e}tale fundamental group of punctured elliptic curves. More generally, one can ask about how the structure of $E$ as an abelian variety interacts with the nonabelian aspects of $E^\circ$. For example, how does $E$ having complex multiplication interact with the arithmetic of nonabelian level structures?

A powerful technique for studying arithmetic is to exploit its interaction with geometry. However, the geometry of $M(G)_\bC$ is locked behind the basic question of determining its connected components. While these components are easy to compute algorithmically for individual examples, it can be quite difficult to prove anything general about them (see \S\ref{sss_geometric_monodromy}). In \S\ref{ss_markoff}, we show how the components of $M(\SL_2(\bF_p))_\bC$, for any prime $p$, can be profitably studied via the arithmetic of character varieties. The problem of understanding connected components is related to the problem of Nielsen equivalence in combinatorial group theory, and can have interesting arithmetic consequences. For example, our results on $M(\SL_2(\bF_p))_\bC$ imply a Diophantine finiteness result on $\SL_2(\bF_p)$-covers for elliptic curves and strong approximation for the Markoff equation.

\subsection{Structure of the paper}

In \S\ref{section_moduli}, we introduce some basic topological properties of elliptic curves, as well as some basic principles for working with stacks, and why one should care about stacks. We will also give constructions of the moduli stack of elliptic curves in two ways; first via framings on homology, which emphasizes the abelian viewpoint, and then via Teichm\"{u}ller theory, which is more compatible with our nonabelian level structures.

In \S\ref{section_nonabelian} we describe nonabelian level structures in two ways: via representations of fundamental groups, and via admissible $G$-covers. In particular, we define our main player $\cM(G)$ : the moduli stack of elliptic curves with $G$-structures, and describe some of its finer features, such as coarse schemes, functoriality, and modular compactifications.

In \S\ref{section_noncongruence}, we describe how the components of $\cM(G)$ give rise to noncongruence modular curves. We explain the results of Asada and Ellenberg-McReynolds. We also describe when the stacks $\cM(G)$ are schemes, and give a group-theoretic framework for understanding the geometry and arithmetic of $\cM(G)$ (and hence noncongruence modular curves).

In \S\ref{section_examples}, we give a tour of the subject via a series of explicit examples of the structure of $\cM(G)$ for various finite groups $G$. We will also describe techniques for computing fields of definition and recognizing when a component of $\cM(G)$ is noncongruence, as well as relations with the Burau representation of braid groups, the Markoff equation, and unbounded denominators for noncongruence modular forms.

\subsection{Conventions and notation}

If $\Pi$ is the fundamental group of a topological space and $\gamma,\delta\in\Pi$, then the product $\gamma\delta$ is to be understood as the homotopy class obtained by first tracing a loop representing $\delta$, followed by tracing a loop representing $\gamma$. In the Galois correspondence, we will typically view monodromy actions as left actions, and Galois actions as right actions.

To keep notation simple, we will denote both the topological fundamental group and \'{e}tale fundamental group by the symbol $\pi_1$. If $X$ is a topological space, then $\pi_1(X)$ will denote the topological fundamental group; if $X$ is a scheme, then $\pi_1(X)$ will be the \'{e}tale fundamental group. An elliptic curve $E$ over $\bC$ will always be considered, by default, as an algebraic variety. Thus $\pi_1(E)$ denotes the \'{e}tale fundamental group, whereas $\pi_1(E(\bC))$ denotes the topological fundamental group of the space of $\bC$-points of $E$.

Let $\fA_n,\fS_n$ respectively denote the alternating and symmetric group on $n$ letters. Let $D_{2n}$ denote the dihedral group of order $2n$. The group $\SU_n(\bF_q)$ consists of the determinant 1 isometries of an $n$-dimensional Hermitian space over $\bF_{q^2}$ relative to the involution $a\mapsto a^q$. $\PSL_n(\bF_q)$ and $\PSU_n(\bF_q)$ are the quotients of $\SL_n(\bF_q),\SU_n(\bF_q)$ by their center.

\section{The moduli of elliptic curves}\label{section_moduli}
In this section we recall the moduli theory of elliptic curves. Much of this section works equally well for analytic spaces, algebraic varieties, and schemes, so we will refer to these collectively as a ``space''.\footnote{We will never need to consider algebraic spaces. In particular, all algebraic stacks under consideration will admit a coarse moduli scheme.} Spaces are typically insufficient for a good theory of moduli, and hence we are motivated to consider the larger category of stacks (or orbifolds), which can be thought of as spaces augmented with extra group structure, and which tautologically solve many moduli problems.

In this article we will treat the theory of stacks as a black box and will limit ourselves to describing the relevant properties of stacks as needed. We refer to \cite{LMB00,Ols16,stacks} for more details regarding stacks.

\subsection{Topological properties of elliptic curves and punctured elliptic curves}\label{ss_topological_properties}
Let $S$ be a torus, by which we mean an oriented closed surface of genus 1. Let $S^\circ$ be the complement of a point inside $S$. Let $s\in S^\circ$ be a base point. The fundamental group $\pi_1(S^\circ,s)$ is a free group of rank 2, and $\pi_1(S,s)$ is a free abelian group of rank 2. If $a,b\in\pi_1(S^\circ,s)$ is a basis with positive intersection number $a\cap b = +1$, then the commutator $[b,a] := bab^{-1}a^{-1}\in\pi_1(S^\circ,s)$ is represented by a small positively oriented loop around the puncture. The map
$$\pi_1(S^\circ,s)\ra\pi_1(S,s)$$
induced by the inclusion $S^\circ\hookrightarrow S$ thus sends $[b,a]$ to the identity, and hence identifies $\pi_1(S,s)\cong H_1(S,\bZ)$ with the abelianization of $\pi_1(S^\circ,s)$.

The \emph{mapping class group} $\Gamma(S)$ of $S$ (resp. $\Gamma(S^\circ)$ of $S^\circ$) is the group of homotopy classes of orientation preserving homeomorphisms (equivalently diffeomorphisms) of $S$ (resp. of $S^\circ$).\footnote{See \cite[\S1.4]{FM11} for a discussion on homotopy vs isotopy, homeomorphisms vs diffeomorphisms.} Since a general homeomorphism may not fix $s$, the mapping class group does not act on $\pi_1(S^\circ,s)$. Nonetheless, since fundamental groups with different base points are canonically isomorphic up to conjugation, we obtain a canonical outer representation
$$\Gamma(S^\circ)\lra\Out(\pi_1(S^\circ,s))$$
Since mapping classes are orientation preserving, the image is in fact contained in the index 2 subgroup $\Out^+(\pi_1(S^\circ,s))$ of outer automorphisms which act with determinant +1 on the abelianization $\pi_1(S^\circ,s)^\ab\cong\bZ^2$. Since $\pi_1(S,s)$ is abelian, in this we case we get an honest action of $\Gamma(S)$ on $\pi_1(S,s)\cong H_1(S,\bZ)$. The following commutative diagram relates these actions.
\begin{equation}\label{eq_MCG_actions}
\begin{tikzcd}
	\Gamma(S^\circ)\ar[r]\ar[d] & \Gamma(S)\ar[d] \\
	\Out^+(\pi_1(S^\circ,s))\ar[r] & \SL(H_1(S,\bZ))
\end{tikzcd}
\end{equation}

A key result that we will use repeatedly is the following:
\begin{thm}\label{thm_MCG_actions} Every map in the diagram \eqref{eq_MCG_actions} is an isomorphism.	
\end{thm}
\begin{proof} The map $\Out^+(\pi_1(S^\circ,s))\ra \SL(H_1(S,\bZ))$ is an isomorphism for purely group theoretic reasons \cite[Corollary N4]{MKS04}. The other maps are isomorphisms by the discussion in \cite[\S2.2.4]{FM11}.	
\end{proof}

\begin{remark} This result is very special to our situation, and in some sense embodies the idea that an elliptic curve simultaneously determines an abelian variety and a hyperbolic (punctured) curve. One can consider two natural generalizations of this situation: to surfaces of higher genus, and to free groups of higher rank. For surfaces of genus $g\ge 2$, the mapping class group does not act faithfully on homology: its kernel is the Torelli group, which is infinite for $g\ge 2$ \cite[\S6.5]{FM11}. For a free group $F_n$ of rank $n$, the abelianization map $F_n\ra\bZ^n$ always induces a surjection $\Out(F_n)\rightiso\GL_n(\bZ)$, but it is only injective for $n = 1, 2$.\footnote{If $x_1,\ldots,x_n$ are free generators, the map sending $x_1\mapsto x_1[x_2,x_3]$ and $x_i\mapsto x_i$ for $i\ge 2$ is a nontrivial element of the kernel.}
	
\end{remark}

\subsection{Fine moduli, coarse moduli, and the necessity of stacks}\label{ss_necessity}
From the classical theory of elliptic curves, we know that to any elliptic curve $E$ over $\bC$, we may associate its $j$-invariant $j(E)$, which is a complex number, and two elliptic curves are isomorphic if and only if their $j$-invariants agree. The map $j : \cH\rightarrow\bC$ sending $\tau$ to the $j$-invariant of the elliptic curve $E_\tau := \bC/\langle 1,\tau\rangle$ is an example of a modular form of weight 0. It is surjective, holomorphic and invariant under $\SL_2(\bZ)$. Since every elliptic curve over $\bC$ is isomorphic to $E_\tau$ for some $\tau\in\cH$, it follows that there is a bijection between the set of isomorphism classes of elliptic curves over $\bC$ and the modular curve $\cH/\SL_2(\bZ)\cong\bA^1_\bC$. This essentially means that $\cH/\SL_2(\bZ)$ is a \emph{coarse moduli space} for elliptic curves. While sufficient for classifying individual elliptic curves, it is not sufficient for classifying families. By this, we mean to desire a space $M$ and a bijection
\begin{equation}\label{eq_fine_moduli}
\Hom(S,M)\rightiso\{\text{isomorphism classes of families of elliptic curves over $S$}\}	
\end{equation}

which is \emph{natural} in the sense that for any maps $T\ra S$ and $S\ra M$, the family corresponding to the composition $T\ra M$ is isomorphic to the pullback, to $T$, of the family corresponding to $S\ra M$. If it exists, such a space $M$ would be called a \emph{fine moduli space} for elliptic curves. Unfortunately, no such space $M$ can exist; the typical example goes as follows. Let $E$ be an elliptic curve over $\bC$, and $\alpha\in\Aut(E)$ a nontrivial automorphism. Let $\MT(E,\alpha)$ denote the \emph{mapping torus} associated to $(E,\alpha)$; it is the family of elliptic curves over the circle $S^1$ constructed by gluing the two ends of $E\times[0,1]$ along the automorphism $\alpha$. If $M$ is any hypothetical moduli space, then by naturality every $t\in S^1$ must be mapped to the point of $M$ corresponding to the isomorphism class of $\cE_t$. Since the fibers are all isomorphic, both $\cE$ and the trivial family $E\times S^1$ must correspond to the same (constant) map $S^1\ra M$. More generally, this shows that fine moduli spaces cannot exist for objects with nontrivial automorphisms.

\begin{remark} While this example is topological, the same phenomenon occurs in arithmetic geometry. For example, for $d\in\bQ^\times$, let $E_d$ be the elliptic curve over $\bQ$ given by the Weierstrass equation $y^2 = x^3+d$. Then $j(E_d) = 0$, and hence for any $d,d'\in\bQ^\times$, $E_d$ is always isomorphic to $E_{d'}$ over $\Qbar$. However they are isomorphic over $\bQ$ if and only if $d/d'$ is a 6th power in $\bQ^\times$ \cite[Corollary 5.4.1]{Sil09}. This shows that if $M$ is a hypothetical fine moduli scheme for elliptic curves, then the base change map $\Hom(\Spec\bQ,M)\ra\Hom(\Spec\Qbar,M)$ must fail to be injective. This is a contradiction, because $\Spec\Qbar\ra\Spec\bQ$ is an epimorphism of schemes, and hence $\Hom(\Spec\bQ,S)\ra\Hom(\Spec\Qbar,S)$ is injective for any scheme $S$.
\end{remark}

Nonetheless, there is a geometric object that satifies the fine moduli property \eqref{eq_fine_moduli}, called the \emph{moduli stack} (or orbifold) of elliptic curves.

\subsection{The moduli stack of elliptic curves}
As described above, for the purposes of moduli, one must not only understand individual objects, but also families of such objects over varying base spaces $S$. We refer to such families as simply being an ``object over $S$''. In algebraic geometry, the definition is as follows:
\begin{defn} Let $S$ be a scheme. An elliptic curve over $S$ is a smooth proper morphism $f : E\ra S$, equipped with a section $O : S\ra E$, such that the geometric fibers of $f$ are connected schemes of pure dimension 1 and arithmetic genus 1.	
\end{defn}

In this article, the moduli stack of elliptic curves will be denoted $\cM(1)$, and its coarse moduli space will be denoted $M(1)$. Here, the ``1'' stands for ``trivial level structure''. For the most part, it suffices to treat $\cM(1)$ as a geometric object such that for every space $S$, there is a bijection
\begin{equation}\label{eq_elliptic_fmp}
\Hom(S,\cM(1))\rightiso \{\text{isomorphism classes of elliptic curves over $S$}\}	
\end{equation}
which is functorial in the sense described above. Thus $\cM(1)$ satisfies the ``fine moduli'' property, at the cost of not being a space, but a stack. Moreover, $\cM(1)$ has an underlying topological space, and there is a \emph{coarse map} $c : \cM(1)\ra M(1)$ which is a homeomorphism and is initial amongst all maps from $\cM(1)$ to spaces.

A precise definition in the context of algebraic geometry is as follows:

\begin{defn}[See {\cite[\S13.1]{Ols16}}] The moduli stack of elliptic curves $\cM(1)$ is the category fibered over the category of schemes $\Sch$ whose objects are triples $(S,E,O)$, where $(E,O)$ is an elliptic curve over the scheme $S$ with zero section $O$, and a morphism $(S,E,O)\ra (S',E',O')$ is a pair of morphisms $(f,g)$ fitting into a cartesian diagram
$$\begin{tikzcd}
	E\ar[d]\ar[r,"f"] & E'\ar[d] \\
	S\ar[r,"g"] & S'
\end{tikzcd}$$
such that $O'\circ g = f\circ O$.
\end{defn}

\begin{remark} In the context of this definition, the fine moduli property \eqref{eq_elliptic_fmp} can be recovered as follows. First, to any scheme $S$ one associates the category $\Sch/S$ of schemes over $S$. By the 2-categorical Yoneda lemma \cite[0GWH]{stacks}, any morphism $f : \Sch/S\lra \cM(1)$ of categories fibered over $\Sch$ is determined by where it sends $S$ (viewed as an $S$-scheme via the identity $\id_S$). Because $f$ is a morphism of fibered categories, the image of $S$ must be an elliptic curve over $S$. By taking isomorphism classes, one obtains the fine moduli property \eqref{eq_elliptic_fmp}. 	
\end{remark}


To construct $\cM(1)$, the general idea is to first construct a \emph{fine moduli space} for elliptic curves equipped with some rigidifying structure, and then to quotient out by a group action to forget the extra structure. Typically this group action will have fixed points; taking the naive quotient will lead to a coarse moduli space, and taking a stack (or orbifold) quotient will lead to a fine moduli space. To emphasize the dichotomy between the congruence and noncongruence settings, we will sketch two constructions of $\cM(1)$: via framings, and via Teichm\"{u}ller theory.

\subsubsection{Moduli of framed elliptic curves}\label{sss_framings}
Let $E$ be a complex analytic elliptic curve. It is topologically a torus, and its first homology is isomorphic to $\bZ^2$. Let $e_1,e_2\in\bZ^2$ be the canonical basis. A \emph{framing} on $E$ is an isomorphism
$$f : \bZ^2\rightiso H_1(E(\bC),\bZ)$$
such that the intersection number $f(e_1)\cap f(e_2) = +1$. If $E/S$ is a family of elliptic curves over a complex manifold $S$, then the topological local triviality of the family implies that the first homology of the fibers is a local system on $S$. A framing on $E/S$ is a locally constant family of framings $f = \{f_s\}_{s\in S}$ on the elliptic curves $E_s$ for $s\in S$.

 Two framed elliptic curves $(E,f),(E',f')$ are isomorphic if there is an isomorphism $\alpha : E\rightiso E'$ such that $\alpha_*\circ f = f'$. We will show that the category of framed elliptic curves admits a \emph{fine moduli space}, which is made possible by the fact that framed elliptic curves have no nontrivial automorphisms. We follow \cite[\S2]{Hain08}; also see \cite[\S2]{BBCL20}.

The action of $\bZ^2$ on $\cH\times\bC$ given by
$$(\tau,z)\cdot (n,m) = (\tau,z+n\tau+m)\qquad (\tau,z)\in\cH\times\bC,\; (n,m)\in\bZ^2$$
is free, and $\bE := (\cH\times\bC)/\bZ^2$, together with the zero section $\cH\times\{0\}$ is a family of elliptic curves over the upper half plane $\cH$. For each $\tau\in\cH$, let $\delta_1,\delta_\tau$ denote the straight-line paths $0\leadsto 1$ and $0\leadsto \tau$ respectively in $\bC$. Then the family
$$f_{\bE,\tau} : (e_1,e_2)\mapsto (\delta_1,\delta_\tau)\qquad \tau\in\cH$$
is a framing of $\bE_\tau$. Given a family of elliptic curves $E\rightarrow T$ with framing $f$, define the period map
$$\Phi : T\rightarrow\cH\quad\text{sending}\quad t\mapsto \frac{\integral_{f_{t}(e_2)}\omega_t}{\integral_{f_{t}(e_1)}\omega_t}$$
where $\omega_t$ is any nonzero holomorphic differential on the fiber $E_t$. Since any two such differentials differ by a scalar multiple, $\Phi$ is independent of the choice of $\omega_t$. Moreover, the framing $f$ is exactly the pullback, via the period map $\Phi$, of the universal framing $f_\bE$. This shows that $\cH$ is a \emph{fine moduli space of framed elliptic curves}, with universal family $(\bE,f_\bE)$.

Note that for any elliptic curve $E$, $\SL_2(\bZ)$ acts freely and transitively on the set of framings on $E$. This gives a \emph{right} action of $\SL_2(\bZ)$ on $\cH$, which can be made explicit as follows. Consider the elliptic curve $\bE_\tau$ for $\tau\in\cH$, with marking $f_{\bE,\tau}$ sending $e_1,e_2\in\bZ^2$ to the paths $\delta_1 : 0\leadsto 1, \delta_\tau : 0\leadsto\tau$ respectively. If $\gamma = \spmatrix{a}{b}{c}{d}\in\SL_2(\bZ)$, then $f_{\bE,\tau}\circ\gamma$ is given by
$$f_{\bE,\tau}\circ\gamma : \begin{array}{rcccl} e_1 & \mapsto & ae_1 + ce_2 & \mapsto & a\delta_1 + c\delta_\tau \\
 e_2 & \mapsto & be_1 + de_2 & \mapsto & b\delta_1 + d\delta_\tau	
 \end{array}$$
Multiplication by $a+c\tau$ defines an isomorphism of framed elliptic curves
$$\Bigg(\underbrace{\bC/\langle 1,\frac{b+d\tau}{a+c\tau}\rangle}_{\bE_{\frac{b+d\tau}{a+c\tau}}}, f_{\bE, \frac{b+d\tau}{a+c\tau}}\Bigg)\rightiso \Bigg(\underbrace{\bC/\langle a+c\tau, b+d\tau\rangle}_{\bE_\tau}, f_{\bE,\tau}\circ\gamma\Bigg)$$
which shows that the action of $\SL_2(\bZ)$ on framings induces the following \emph{right} action on $\cH$:
\begin{equation}\label{eq_right_action}
\tau\gamma := \frac{b+d\tau}{a+c\tau}\qquad\text{for }\gamma = \spmatrix{a}{b}{c}{d}\in\SL_2(\bZ)	
\end{equation}
\subsubsection{Remarks on mirror-image modular curves}\label{sss_mirror}
We note that the action \eqref{eq_right_action} differs from the more classical left action given by $\spmatrix{a}{b}{c}{d}\tau = \frac{a\tau+b}{c\tau+d}$ described in the introduction. The two actions are related by the formula
$$\tau\gamma = \mf{s}(\gamma)\tau$$
where $\mf{s}$ is given by $\spmatrix{a}{b}{c}{d}\mapsto\spmatrix{d}{b}{c}{a}$. Let $R := \spmatrix{-1}{0}{0}{1}$, then $\mf{s}$ can also be expressed by $\mf{s}(\gamma) = R\gamma^{-1}R$ for $\gamma\in\SL_2(\bZ)$. While the presence of the inverse $\gamma^{-1}$ is a formal necessity when translating between left and right actions, the presence of the antiholomorphic involution $R$ is more substantial. When considering quotients, we will distinguish the two actions by the following conventions:
\begin{itemize}
	\item $\Gamma\bs\cH$ denotes the quotient of $\cH$ by the classical left action of $\Gamma$ described in the introduction \eqref{eq_left_action}.
	\item $\cH/\Gamma$ denotes the quotient of $\cH$ by the right action \eqref{eq_right_action}.
\end{itemize}
Note that $\Gamma\bs\cH\cong\cH/\mf{s}(\Gamma)$, and $\cH/\Gamma\cong\mf{s}(\Gamma)\bs\cH$. We say that the two modular curves $\Gamma\bs\cH$ and $\cH/\Gamma$ are \emph{mirror images}. Note that for any of the common congruence subgroups $\Gamma = \Gamma(n),\Gamma_1(n),\Gamma_0(n)$, we always have $\mf{s}(\Gamma) = \Gamma$. In particular, $\mf{s}(\SL_2(\bZ)) = \SL_2(\bZ)$. For general $\Gamma\le\SL_2(\bZ)$, $\mf{s}(\Gamma)$ and $\Gamma$ are conjugate in $\GL_2(\bZ)$, but are in general not conjugate in $\SL_2(\bZ)$. This implies that the modular curves $\Gamma\bs\cH, \mf{s}(\Gamma)\bs\cH\cong\cH/\Gamma$ are generally not isomorphic as branched covers of $\cH/\SL_2(\bZ)$. They can be related via the \emph{antiholomorphic} involution $r : \cH\ra\cH$ sending $\tau\mapsto -\ol{\tau} = R\ol{\tau}$. This $r$ induces an antiholomorphic homeomorphism $\Gamma\bs\cH \lra\cH/\Gamma$, which fits into the commutative diagram
\[\begin{tikzcd}
	\cH\ar[d]\ar[r,"r"] & \cH\ar[d] \\
	\Gamma\bs\cH\ar[r,"r"]\ar[d] & \cH/\Gamma\ar[d] \\
	\cH/\SL_2(\bZ)\ar[r,"r"] & \cH/\SL_2(\bZ)
\end{tikzcd}\]
Since $r$ fixes the cusp and the two elliptic points of $\cH/\SL_2(\bZ)$, the diagram also implies that the ramification behavior of $\cH/\Gamma$ and $\Gamma\bs\cH$ over $\cH/\SL_2(\bZ)$ are identical. Arithmetically, they are complex conjugates:

\begin{prop}\label{prop_mirror} Let $\Gamma\le\SL_2(\bZ)$ be a finite index subgroup. The mirror images $\Gamma\bs\cH$ and $\cH/\Gamma$ are complex conjugates of each other relative to the $\bQ$-structure on $\cH/\SL_2(\bZ)\cong\bP^1_\bC$ given by the $j$-function.	
\end{prop}
\begin{proof} Let $f$ be a meromorphic function on $\cH/\Gamma$. If $f = \sum_{n} a_nq^n$ is the Fourier expansion in $q = e^{2\pi i z/m}$, then $f\circ r = \sum_n a_n \ol{q^n}$, where $\ol{*}$ denotes complex conjugation. The map $f\mapsto \ol{f\circ r}$ is an isomorphism of meromorphic function fields $\bC(\Gamma\bs\cH)\rightiso\bC(\cH/\Gamma)$ which fixes the modular $j$-function $j$, and restricts to complex conjugation on the subfield $\bC$ of constants. Thus, if $f$ satisfies a polynomial $h(j,F)\in\bC[j][F]$, then $\ol{f\circ r}$ satisfies $\ol{h}(j,F)$.
\end{proof}

\begin{remark} The same issue arises in the theory of Teichm\"{u}ller curves, c.f. \cite{HS07} and \cite[Prop 3.2]{Mcm03}.
\end{remark}

\subsubsection{Moduli of elliptic curves with Teichm\"{u}ller markings}\label{sss_Teichmuller}
\begin{defn}[Teichm\"{u}ller space for the torus] Let $S$ be a torus. Let $(E,O)$ be an elliptic curve. A (Teichm\"{u}ller) marking on $E$ is a diffeomorphism $m : S\ra E$. The Teichm\"{u}ller space $\cT(S)$ of $S$ is the set of equivalence classes of marked elliptic curves $(E,m)$, where two pairs $(E,m), (E',m')$ are considered equivalent if $m'\circ m^{-1}$ is homotopic to an isomorphism $E\rightiso E'$. For a pair $(E,m)$, let $[E,m]\in\cT(S)$ denote its equivalence class.\footnote{$\cT(S)$ is homeomorphic to the Teichm\"{u}ller space of a punctured torus. See the discussion in \cite[\S10-11]{FM11} for more details, especially \S10.2 and \S11.4.3. The Teichm\"{u}ller space $\cT(S)$ can also be understood as a \emph{fine moduli space} of marked elliptic curves \cite{AJP16}.}
\end{defn}




Let $f$ be a framing on $S$. Then to any marked elliptic curve $(E,m)$, we obtain a framing $m(f) := m_*\circ f$ on $E$
$$m_*\circ f : \bZ^2\stackrel{f}{\lra} H_1(S,\bZ)\stackrel{m_*}{\lra} H_1(E,\bZ)$$

\begin{prop} For any framing $f$ on $S$, the association $(E,m)\mapsto (E,m(f))$ defines a bijection
\begin{eqnarray*}
\Psi_f : \cT(S) & \lra & \cH \\
{[E,m]} & \lra & [E,m(f)]
\end{eqnarray*}
where we view $\cH$ as a fine moduli space of framed elliptic curves, with universal family $(\bE,f_\bE)$, and $[E,m(f)]$ is the isomorphism class of the framed elliptic curve $(E,m(f))$.
\end{prop}

We give $\cT(S)$ the structure of a (contractible) complex manifold via $\Psi_f$. This complex structure is in fact independent of $f$, and agrees with the more classical complex structure defined by means of the Bers embedding \cite[Notes on p179]{IT92}.

\begin{proof} If $(E,m)\sim (E',m')$, then $\alpha := m'\circ m^{-1} :E\ra E'$ is homotopic to an isomorphism $\alpha'$. Thus $\alpha'$ sends $m(f)$ to $(m_*' \circ m_*^{-1})\circ(m_* \circ f) = m_*'\circ f$, so $\alpha'$ induces an isomorphism $(E,m(f))\sim(E',m'(f))$. This shows that $\Psi_f$ is well-defined.

Now suppose that we are given $(E,m),(E',m')$ mapping to the same point in $\cH$. This means there is an isomorphism $\alpha : E\ra E'$ such that $\alpha_*\circ m_*\circ f = m'_*\circ f$. This would imply that $m'^{-1}_*\circ\alpha_*\circ m_*\circ f = f$, so $m'^{-1}_*\circ\alpha_*\circ m_* = \id_{H_1(S)}$. Since the mapping class group $\Gamma(S)$ acts faithfully on homology (see Theorem \ref{thm_MCG_actions}), this implies that $m'^{-1}\circ\alpha\circ m$ is homotopic to $\id_{S}$, so $m'\circ m^{-1}$ is homotopic to the isomorphism $\alpha$, which is to say $[E,m] = [E',m']$. This shows that $\Psi_f$ is injective.

Finally, since $\Homeo^+(S)$ acts transitively on the set of all framings, every framing on $E$ can be realized as the pushforward of a framing on $S$. Thus $\Psi_f$ is surjective.
\end{proof}

The mapping class group $\Gamma(S^\circ)\cong\Gamma(S)$ acts freely and transitively on the set of markings of any given elliptic curve $E$. For a framing $f$ of $S$, we have a transport of structure isomorphism
\begin{eqnarray*}
i_f : \Gamma(S^\circ)\cong\Gamma(S) & \rightiso & \SL_2(\bZ) \\
\alpha & \mapsto & f^{-1}\circ\alpha_*\circ f
\end{eqnarray*}
where $\alpha$ denotes the induced action on $H_1(S)$. The right actions of $\Gamma(S)$ on $\cT(S)$ and $\SL_2(\bZ)$ on $\cH$ are related as follows.
\begin{prop} Let $f$ be a framing on a torus $S$. The isomorphism
$$\Psi_f : \cT(S)\lra\cH$$
is equivariant with respect to the $\Gamma(S)$ action on $\cT(S)$, the $\SL_2(\bZ)$-action on $\cH$ (described in \eqref{eq_right_action}), and the isomorphism $i_f : \Gamma(S)\rightiso\SL_2(\bZ)$. In a formula, we have
$$\Psi_f(x\cdot \alpha) = \Psi_f(x)\cdot i_f(\alpha)\qquad\text{for all $x\in\cT(S)$, $\alpha\in\Gamma(S)$}$$
\end{prop}

\subsubsection{Moduli stack of elliptic curves as a quotient stack}

In the previous sections we have constructed fine moduli spaces of elliptic curves equipped with additional structure: either a framing (with moduli space $\cH$), or a marking (with moduli space $\cT(S)$). To obtain a moduli theory for elliptic curves, we wish to ``forget'' the extra structure, which amounts to taking an appropriate quotient of $\cH$ by $\SL_2(\bZ)$, or $\cT(S)$ by $\Gamma(S)\cong\Gamma(S^\circ)$. However, the mapping torus construction of \S\ref{ss_necessity} shows that the naive topological quotient cannot be a fine moduli space.\footnote{The obstruction to being a fine moduli space can also be understood as coming from the existence of points $\tau\in\cH$ with nontrivial $\SL_2(\bZ)$ stabilizers: If $\gamma\in\SL_2(\bZ)$ is a nonidentity element which fixes $\tau\in\cH$, this is to say that there is a framing $f$ on the elliptic curve $\bE_\tau$ and an automorphism $\alpha\in\Aut(\bE_\tau)$ such that $\alpha_*\circ f = f\circ\gamma\ne f$. This inequality gives another way to see that the mapping torus $\MT(E,\alpha)$ is not isomorphic to the trivial family $E\times S^1$: the latter admits a framing, but the former does not.}

The correct way to forget the framing while maintaining the fine moduli property is to instead consider the stack quotient $[\cH/\SL_2(\bZ)]$, or $[\cT(S)/\Gamma(S)]$ in the Teichm\"{u}ller setting.\footnote{For a precise definition in the setting of algebraic geometry, see \cite[Example 8.1.12]{Ols16}.} This stack quotient fits into the diagram
$$\cH\lra[\cH/\SL_2(\bZ)]\lra\cH/\SL_2(\bZ)$$
where the composition is the topological quotient, and the first map is the canonical quotient map, which is an \emph{(unramified) covering map} of stacks.\footnote{In general, the stack quotient by any discrete group is always a covering map. In particular, the fundamental group of $[\cH/\SL_2(\bZ)]$ is $\SL_2(\bZ)$.} The second map is called the ``coarse map''; it is a homeomorphism which forgets the stacky structure\footnote{One can imagine $[\cH/\SL_2(\bZ)]$ as being obtained by augmenting $\cH/\SL_2(\bZ)$ with additional data, which includes the stabilizer groups of every point.}, and can be viewed as having fractional degree above every point of $\cH/\SL_2(\bZ)$. The map $u_\bE : \cH\ra\cM(1)$ corresponding to the family $\bE$ induces isomorphisms
$$\ol{u_\bE} : [\cH/\SL_2(\bZ)]\rightiso\cM(1),\;\;\text{and}\;\;|\ol{u_\bE}| : \cH/\SL_2(\bZ)\rightiso M(1)$$
\begin{remark} To make sense of this construction in algebraic geometry, instead of framings of elliptic curves $E$, one can instead consider bases of \'{e}tale cohomology (or homology). If one takes cohomology with coefficients in $\bZ/n$, such a basis amounts to a classical ``full level $n$ structure'', corresponding to a basis for the $n$-torsion subgroup $E[n]$. The (analytification of the) associated moduli space is $\cH/\Gamma(n)$. For $n\ge 3$, $\cH/\Gamma(n)$ is a fine moduli space, and one can recover the moduli stack of elliptic curves by taking the stack quotient by $\SL_2(\bZ/n)$.

\end{remark}

\subsection{Subgroups as coverings: the Galois correspondence and the Riemann existence theorem}\label{ss_subgroups}

In the setting of moduli interpretations for noncongruence modular curves, there is another reason to consider stack quotients: there exist finite index subgroups $\Gamma_1,\Gamma_2$ such that $\Gamma_1$ is congruence, $\Gamma_2$ is noncongruence, and yet $\cH/\Gamma_1 = \cH/\Gamma_2$, where here equality is in the sense that the action images of $\Gamma_1,\Gamma_2$ in $\Aut(\cH)\cong\PSL_2(\bR)$ are identical \cite{KSV11}. In other words, a modular curve can be simultaneously congruence and noncongruence at the same time! This does not happen if we consider the stack quotient $[\cH/\Gamma]$. 

\begin{defn} For a finite index subgroup $\Gamma\le\SL_2(\bZ)$, we call $[\cH/\Gamma]$ the \emph{modular stack} associated to $\Gamma$. It is congruence if $\Gamma$ is congruence, and otherwise it is noncongruence. We say that a connected finite \'{e}tale cover $\cM\ra\cM(1)_\bC$ is congruence if its analytification is isomorphic to a congruence modular stack relative to the uniformization given by $\bE/\cH$. 
\end{defn}

We recall some basic properties of $[\cH/\Gamma]$ here:

\begin{enumerate}
\item Every space we consider can be viewed as a stack in the appropriate category (e.g., analytic spaces, schemes, topological spaces,...). In particular, the category of spaces is a full subcategory of the category of stacks.
\item There is a homeomorphism $c : [\cH/\Gamma]\rightarrow\cH/\Gamma$, initial amongst all maps from $[\cH/\Gamma]$ to spaces. The modular curve $\cH/\Gamma$ is called the ``coarse space'' of $[\cH/\Gamma]$, and $c$ is the ``coarse map''.
\item For any $\Gamma$-invariant open subset $U\subset\cH$, $c|_{U/\Gamma} : [U/\Gamma]\ra U/\Gamma$ is an isomorphism if and only if $\Gamma$ acts freely on $U$. In particular, $c$ is an isomorphism if and only if $\Gamma$ acts freely on $\cH$, or equivalently, if $\Gamma$ is torsion-free.
\item For any subgroup $\Gamma'\le\Gamma$, there is a natural map $[\cH/\Gamma']\ra[\cH/\Gamma]$ which is an unramified covering map.
\end{enumerate}	


\begin{thm}[Galois correspondence {\cite[Theorem 18.19]{Noo05}}]\label{noohi_topological_galois} Let $\cX$ be a connected topological stack which is locally path-connected and semilocally simply connected, and $x_0\in\cX$ a point. Let $\Sets$ denote the category of sets, and $\cC_\cX$ be the category of covering maps $\cY\ra\cX$, where $\cY$ is a topological stack. Let $F_{x_0}$ be the functor
$$F_{x_0} : \cC_\cX\ra\Sets$$
which sends any $\cY\in\cC_\cX$ to its fiber $\cY_{x_0}$. Then there is a fundamental group $\pi_1(\cX,x_0)$ equipped with a functorial action, called the \emph{monodromy representation}, on $F_{x_0}(\cY)$ for $\cY\in\cC_\cX$. With this action, $F_{x_0}$ defines an equivalence of categories
$$F_{x_0} : \cC_{\cX}\lra\Sets_{\pi_1(\cX,x_0)}$$
where $\Sets_{\pi_1(\cX,x_0)}$ denotes the category of sets with $\pi_1(\cX,x_0)$-action.
\end{thm}
If $\cX$ is a topological space, then any covering $\cY$ of $\cX$ is also a topological space. In this case we recall that the monodromy action is given as follows: For an element of $\pi_1(\cX,x_0)$ represented by a loop $\gamma$ based at $x_0$, the monodromy action of $\gamma$ on the fiber $p^{-1}(x_0)$ sends any $y_0\in p^{-1}(x_0)$ to the endpoint of the unique path lifting $\gamma$ starting at $y_0$.

The theorem in particular applies to the stacks $[\cH/\Gamma]$. We recall some familiar consequences:
\begin{enumerate}
\item Since $\cH$ is simply connected, for any subgroup $\Gamma\le\SL_2(\bZ)$, $\pi_1([\cH/\Gamma],x_0) \cong \Gamma$.\footnote{In this way, the modular stack $[\cH/\Gamma]$ ``remembers'' the subgroup $\Gamma$, whereas the modular curve $\cH/\Gamma$ only remembers the image of $\Gamma$ in $\Aut(\cH)\cong\PSL_2(\bR)$.}
\item If $\cY\ra[\cH/\Gamma]$ is a covering, then the connected components of $\cY$ are in bijection with the orbits of $\pi_1([\cH/\Gamma],x_0)$ on $\cY_{y_0}$.
\item\label{part_every_covering} Every connected covering of $[\cH/\SL_2(\bZ)]$ is isomorphic to $[\cH/\Gamma]\ra[\cH/\SL_2(\bZ)]$ for some subgroup $\Gamma\le\SL_2(\bZ)$. The degree of the covering is the index $[\SL_2(\bZ):\Gamma]$.
\item If $\cY\ra[\cH/\Gamma]$ is a connected covering, and $y_0\in\cY$ is a point lying above $x_0$, then $\pi_1(\cY,y_0)\cong\Stab_{\pi_1([\cH/\Gamma],x_0)}(y_0)$. Two subgroups $\Gamma,\Gamma'\le\SL_2(\bZ)$ give rise to isomorphic coverings of $[\cH/\SL_2(\bZ)]$ if and only if they are conjugate in $\SL_2(\bZ)$. The covering is Galois if and only if the stabilizer is normal, in which case the Galois group is isomorphic to the quotient $\Gamma/\Stab_{\pi_1([\cH/\Gamma],x_0)}(y_0)$.
\end{enumerate}

To describe arithmetic models of noncongruence modular curves, we will also need a Galois correspondence for algebraic stacks. For schemes, this is due to Grothendieck in \cite[\S4-5]{SGA1}, via the machinery of Galois categories. The analogous theory for algebraic stacks was worked out by Noohi:
\begin{thm}[Galois correspondence {\cite[Theorem 4.2]{Noo04}}]\label{thm_Galois} Let $\cX$ be a connected algebraic stack, and $x_0\in\cX$ a geometric point. Let $\FiniteSets$ denote the category of finite sets. Let $\cC_\cX$ denote the category of finite \'{e}tale morphisms $\cY\ra\cX$, where $\cY$ is an algebraic stack\footnote{To be precise, stacks finite \'{e}tale over $\cX$ forms a (2,1)-category, in the sense that morphisms between any two given stacks finite \'{e}tale over $\cX$ form a \emph{groupoid}. Thus between any two morphisms there is a set of 2-isomorphisms between them. By $\cC_\cX$ we mean the category obtained from this (2,1)-category by declaring all 2-isomorphisms to be the identity \cite[\S4]{Noo04}.}, and let $F_{x_0}$ denote the ``fiber'' functor
$$F_{x_0} : \cC_{\cX}\lra\ul{\textbf{FiniteSets}}$$
which sends $\cY\in\cC_\cX$ to the geometric fiber $\cY_{x_0}$. The \'{e}tale fundamental group of $(\cX,x_0)$, denoted $\pi_1(\cX,x_0)$, is by definition the automorphism group of the functor $F_{x_0}$. This is a profinite group, and is equipped with a (tautological) functorial action on the finite sets $F_{x_0}(\cY)$ for any $\cY\in\cC_\cX$. Then $F_{x_0}$ defines an equivalence of categories
$$F_{x_0} : \cC_\cX\lra\ul{\textbf{FiniteSets}}_{\pi_1(\cX,x_0)}$$
where $\ul{\textbf{FiniteSets}}_{\pi_1(\cX,x_0)}$ denotes the category of finite sets equipped with a continuous action of $\pi_1(\cX,x_0)$. Under $F_{x_0}$, connected components correspond to $\pi_1(\cX,x_0)$-orbits; in particular, connected objects map to transitive $\pi_1(\cX,x_0)$-sets.
\end{thm}

The main tool that allows us to pass between the topological and algebraic settings is the following:

\begin{thm}[Riemann existence theorem {\cite[Theorem 20.4]{Noo05}}] Let $\cX$ be a connected algebraic stack that is locally of finite type over $\bC$, and let $\cX^\tp$ be the associated topological stack (see \cite{Noo05} for more details). The functor $\cY\mapsto \cY^\tp$ defines an equivalence of categories between the category of finite \'{e}tale maps $\cY\ra\cX$ and the category of finite covering stacks of $\cX^\tp$.
\end{thm}

The \'{e}tale fundamental group of $\cX$ is isomorphic to the inverse limit of the automorphism groups of finite Galois covers of $\cX$ \cite[Corollary 5.4.8]{Sza09}. By taking associated topological covers, this defines a map
$$\pi_1(\cX,x_0)\lra\pi_1(\cX^\tp,x_0)$$
The Riemann existence theorem then implies:
\begin{thm}[Comparison theorem]\label{thm_comparison} Let $\cX$ be a connected algebraic stack locally of finite type over $\bC$. Assume that $\cX^\tp$ is locally path connected and semilocally simply connected. Then the map $\pi_1(\cX^\tp,x_0)\ra\pi_1(\cX,x_0)$ factors through an isomorphism
$$\pi_1(\cX^\tp,x_0)^\wedge\rightiso \pi_1(\cX,x_0)$$
where $\cdot^\wedge$ denotes profinite completion.
\end{thm}

\begin{remark} The Riemann existence theorem shows that ``finite \'{e}tale'' is a good algebraic analog of ``finite covering map''. By contrast, the exponential map $\exp : \bC\ra\bC^\times$ is an (infinite degree) topological covering which is not algebraic; this partially explains the restriction to finite covers in the algebraic Galois correspondence.\footnote{In \cite{BS15}, Bhatt and Scholze have developed the \emph{pro-\'{e}tale} fundamental group, which classifies \emph{geometric covers}. These are locally constant sheaves in the pro-\'{e}tale topology, which includes all finite \'{e}tale covers. A recent preprint \cite{YZ22} shows that for a scheme $X$ locally of finite type over $\bC$, the geometric covers of $X$ are exactly the algebraizable covers of $X^\an$; moreover they give examples which show that the pro-\'{e}tale fundamental group is not determined by the topological fundamental group.}
\end{remark}

Let $u_\bE : \cH\ra\cM(1)_\bC^\tp$ be the map given by the universal framed family $\bE/\cH$. Let $S$ be a torus, and $f$ a framing on $S$. If $\ff : \cM\ra\cM(1)_\bC$ is any connected finite \'{e}tale covering, then by analytification and pullback we obtain finite index subgroups $\Gamma\le\SL_2(\bZ)$, $\Delta\le\Gamma(S)$, and topological covers $[\cH/\Gamma]$ and $[\cT(S)/\Delta]$ of $[\cH/\SL_2(\bZ)]$ and $[\cT(S)/\Gamma(S)]$ respectively. These objects fit into the commutative diagram

\usetikzlibrary{decorations.pathmorphing}
\[\begin{tikzcd}
	{[\cT(S)/\Delta]}\ar[r,"\ol{\Psi_f}"]\ar[d,"\pr_{\cT(S)}"] & {[\cH/\Gamma]}\ar[d,"\pr_\cH"]\ar[r,"\ol{u_\bE}"] & \cM^\tp\ar[d,"\ff^\tp"] & \cM_\bC\ar[l,rightsquigarrow,"\tp"]\ar[d,"\ff"]\\
{[\cT(S)/\Gamma(S)]}\ar[r,"\ol{\Psi_f}"] & {[\cH/\SL_2(\bZ)]}\ar[r,"\ol{u_\bE}"] & \cM(1)_\bC^\tp & \cM(1)_\bC\ar[l,rightsquigarrow,"\tp"]
\end{tikzcd}\]
where all straight horizontal maps are isomorphisms, and $\ol{\Psi_f}$ (resp. $\ol{u_\bE}$) is induced by $\Psi_f$ (resp. $u_\bE$). Moreover, the subgroups $\Delta,\Gamma$ are determined up to conjugation by the diagram. The Riemann existence theorem can be understood as saying that no information is lost in passing between the various columns of the diagram.

\section{Nonabelian level structures and branched coverings of elliptic curves}\label{section_nonabelian}
In this section we describe the generalized notion of level structure, and explain how it gives a moduli interpretation to every noncongruence modular curve. The goal is to produce, for any finite group $G$, a moduli stack of elliptic curves equipped with a $G$-Galois cover only ramified above the origin, where the cover should be treated as a generalized ``level structure''; in particular, isomorphisms in this stack should be determined by isomorphisms of elliptic curves, and the forgetul map to $\cM(1)$ should be finite \'{e}tale. The most geometric thing to consider is the moduli stack of admissible $G$-covers $\cAdm(G)$ \cite{ACV03}, but this isn't ideal because covers can have automorphisms which do not come from automorphisms of the base elliptic curve, and accordingly the forgetful map fails to be finite. In \S\ref{ss_combinatorial_structures}, we obtain a finite \'{e}tale forgetful map by replacing covers by representations of the fundamental group, leading to the notion of a nonabelian level structure as first considered by Deligne and Mumford for unramified covers \cite{DM69}. The corresponding stacks $\cM(G)$ are finite \'{e}tale over $\cM(1)$ and are direct generalizations of the classical congruence moduli problems. However they are not very geometric; for example, a $\bQ$-point does not in general correspond to a $G$-cover over an elliptic curve over $\bQ$. Next, in \S\ref{ss_AdmG}, we will describe the geometric approach via the moduli of smooth admissible covers $\cAdm^0(G)$, and will describe their relation to $\cM(G)$. The stacks $\cAdm^0(G)$ moreover have a natural compactification $\cAdm(G)$ (see \S\ref{ss_compactification}), which can be used to obtain compactifications of $\cM(G)$.

In the remainder of this section, $G$ will always denote a finite group, and we will always make the following tameness asumption:
\begin{equation}\label{eq_tame}
\text{All schemes are over $\bS := \Spec\bZ[1/|G|]$}
\end{equation}
This means that when we say ``let $S$ be a scheme'', we mean: ``let $S$ be a scheme on which $|G|$ is invertible''. Similarly, $\cM(1)$ denotes the moduli stack of elliptic curves over $\bS$-schemes. Our moduli stacks are essentially ``Hurwitz stacks'', classifying branched covers of elliptic curves. We begin by recalling some basic properties of branched covers of curves.

\subsection{Branched covers over an algebraically closed field $k$} \label{ss_admissible_covers_over_k}
Let $D$ be a smooth projective curve\footnote{By this, we mean a connected smooth projective $k$-scheme of pure dimension 1.} over an algebraically closed field $k$. A \emph{branched cover} of $D$ is a finite generically \'{e}tale map $\pi : C\ra D$, where $C$ is connected, smooth, and projective over $k$. Such a map is always flat \cite[00R4]{stacks}, and in characteristic 0, the generic \'{e}taleness is also automatic. If $\pi$ is not \'{e}tale at $x\in C$, then we say $x$ is a \emph{ramified point}, and $\pi(x)$ is a \emph{branch point}. The \emph{ramification index} of $x$ is the ramification index of the extension of discrete valuation rings $\cO_{C,x}/\cO_{D,\pi(x)}$ \cite[09E4]{stacks}; we say $\pi$ is ramified at $x$ if the ramification index is $> 1$.

If $\pi$ induces a Galois extension of function fields, and is moreover equipped with a $D$-linear action of a finite group $G$ on $C$ which induces an isomorphism $G\cong\Gal(k(C)/k(D))$, then $\pi$ induces an isomorphism $C/G\cong D$, and we say that $\pi : C\ra D$ is a \emph{branched $G$-cover} of $D$. For $x\in C$, our tameness assumption implies that the stabilizer $G_x := \Stab_G(x)$ is cyclic, and its order is equal to the ramification index $e_x$ of $\pi$ at $x$; thus the ramified points are exactly those which have nontrivial $G$-stabilizers. The stabilizer $G_x$ acts on the cotangent space $T_x^*$ via a faithful \emph{local representation}:
$$\chi_x : G_x\lra \GL(T_x^*) = k^\times$$
If $x'\in C$ is a point in the same $G$-orbit (equivalently, lying over the same point of $E$), then the characters $\chi_x,\chi_{x'}$ are related by an inner automorphism of $G$. The collection of conjugacy classes of characters, one for each of the finitely many non-free $G$-orbits, is called a \emph{Hurwitz data} (or ramification data) for the branched $G$-cover $\pi$ \cite[\S2.2]{BR11}.

\begin{defn} Let $\Sigma\subset D$ be a finite set of closed points. An \emph{admissible cover} (resp. admissible $G$-cover) of the marked curve $(D,\Sigma)$ is a branched cover (resp. branched $G$-cover) $\pi : C\ra D$ whose branch points are contained in $\Sigma$. A morphism of admissible $G$-covers of $(D,\Sigma)$ is a $G$-equivariant morphism commuting with the projection maps to $D$. If $D$ is an elliptic curve (resp. a pointed torus), the divisor $\Sigma$ will be taken to be the origin (resp. the marked point); thus an admissible $G$-cover of an elliptic curve (resp. pointed torus) is only allowed to be branched above the origin (resp. the marked point).
\end{defn}

\begin{prop}\label{prop_automorphism_for_smooth_covers} Let $\pi : C\ra D$ be a (smooth) admissible $G$-cover. Then its automorphism group as a $G$-cover is $Z(G)$.	
\end{prop}
\begin{proof} An automorphism of the cover $\pi$ realized by $g\in G$ is $G$-equivariant if and only if $g\in Z(G)$.	
\end{proof}

Suppose now that $(E,O)$ is an elliptic curve over $k$, and $\pi : C\ra E$ is an admissible $G$-cover relative to the divisor $\Sigma = O\in E$. Then $\pi^{-1}(E^\circ)\ra E^\circ$ is an \'{e}tale $G$-cover. Since there is at most one $G$-orbit of ramified points, all ramified points have the same ramification index, which we simply call the ramification index of $\pi$. If $x\in C$ lies over $O$ and $\zeta_e$ is a primitive $e$th root of unity, then the conjugacy class of $\chi_x^{-1}(\zeta_e)$ is called the \emph{Higman invariant} of $\pi$ relative to $\zeta_e$.

If $k = \bC$, $x_0\in E^\circ$ is a base point and $a,b\in\pi_1(E^\circ,x_0)$ is a basis for its fundamental group with intersection number $a\cap b = +1$, then the commutator $[b,a] := bab^{-1}a^{-1}$ is represented by a positively oriented loop in $E^\circ(\bC)$ winding once around the puncture. We say that $[b,a]$ is a \emph{generator of inertia} at $O\in E$. For any $y_0\in C$ lying over $x_0$, the monodromy action of $\pi_1(E^\circ(\bC),x_0)$ on the fiber $\pi^{-1}(x_0)$, described in \S\ref{ss_subgroups}, gives a surjective representation
$$\rho_{\pi,y_0} : \pi_1(E^\circ(\bC),x_0)\lra G$$
where $\gamma\in \pi_1(E^\circ(\bC),x_0)$ is mapped to the unique $g_{\gamma,y_0}\in G$ satisfying $g_{\gamma,y_0}(y_0) = \gamma\cdot y_0$, where $\gamma\cdot y_0$ denotes the monodromy action. Different choices of $y_0\in\pi^{-1}(x_0)$ lead to $G$-conjugate monodromy representations. Since $\pi$ can be recovered from $\pi|_{\pi^{-1}(E^\circ)}$ as the normalization of $E$ inside $\pi^{-1}(E^\circ)$, we get bijections
\begin{eqnarray*}
\{\text{admissible $G$-covers of $E$}\}/\cong & \stackrel{\text{restrict}}{\lra} & \{\text{unramified $G$-covers of $E^\circ$}\}/\cong \\
 & \stackrel{\text{topologize}}{\lra} & \{\text{unramified $G$-covers of $E^\circ(\bC)$}\}/\cong \\
 & \stackrel{\text{monodromy}}{\lra} & \Epi^\ext(\pi_1(E^\circ(\bC),x_0),G) \\
 & \stackrel{\text{comparison}}{\lra} & \Epi^\ext(\pi_1(E^\circ,x_0),G)
\end{eqnarray*}
where the last map is given by the comparison theorem \ref{thm_comparison}. Since $\pi_1(E^\circ(\bC),x_0)$ is free of rank 2, we have

\begin{thm}\label{thm_2_generation} Let $E$ be an elliptic curve over an algebraically closed field $k$, and $G$ a finite group (of order invertible in $k$). There exist admissible $G$-covers of $E$ if and only if $G$ is generated by two elements.
\end{thm}
\begin{proof} If $k = \bC$, this follows from the above. The general case follows from the theory of specialization, see Proposition \ref{prop_geometric_pi1} below.
\end{proof}

The monodromy element $\rho_{\pi,y_0}([b,a])$ associated to the generator of inertia $[b,a]$ is related to the local representation $\chi_{y_0}$ via
$$\rho_{\pi,y_0}([b,a]) = \chi_{y_0}^{-1}(\zeta_e)$$
where here we take $\zeta_e := \exp(2\pi i/n)$. Different choices of $y_0$ lead to conjugate monodromy elements, and hence in the analytic setting the Higman invariant relative to $\zeta_e := \exp(2\pi i/n)$ is exactly the monodromy element associated to a small positively oriented loop around $O$. This description relates ramification to $G$ being nonabelian:

\begin{prop} Let $E$ be an elliptic curve, and $\pi : C\ra E$ an admissible $G$-cover. Then the following are equivalent:
\begin{enumerate}
\item\label{part_pi_unramified} $\pi$ is unramified.
\item\label{part_G_abelian} $G$ is abelian.
\item\label{part_C_genus_1} $C$ is a curve of genus 1.	
\end{enumerate}
\end{prop}
\begin{proof} By the Hurwitz formula \cite[\S II, Theorem 5.9]{Sil09}, a tamely ramified branched cover of an elliptic curve has genus $\ge 1$, with equality if and only if the cover is unramified. Thus our tameness assumption \eqref{eq_tame} implies that \ref{part_pi_unramified} is equivalent to \ref{part_C_genus_1}. Since $C$ is connected, the monodromy representation $\rho_\pi : \pi_1(E^\circ(\bC),x_0)\ra G$ associated to the unramified cover $\pi^{-1}(E^\circ(\bC))\ra E^\circ(\bC)$ is surjective. If $a,b\in\pi_1(E^\circ,x_0)$ is a basis with $a\cap b = +1$, then $G$ is generated by the images of $a$ and $b$, and hence is abelian if and only if their images commute, or equivalently if the monodromy image of $[b,a]$ is 1, or equivalently if $\pi$ is unramified.	
\end{proof}

\subsection{$G$-structures via representations of fundamental groups}\label{ss_combinatorial_structures}

In this section we define $G$-structures on elliptic curves, where $G$ is a finite group, which recovers the classical congruence level structures when $G$ is abelian. Such level structures were first considered by Deligne and Mumford for \'{e}tale covers of proper curves of genus $g\ge 2$ \cite[\S5]{DM69}. We begin with a quick geometric definition, which we then reinterpret combinatorially using Galois theory, leading to an explicit algorithm for computing the geometry of the associated moduli spaces.

Let $\cT_G^\pre : \cM(1)\ra\Sets$ be the presheaf (i.e., functor), which associates to any object $E/S$ of $\cM(1)$ the set of isomorphism classes of $G$-torsors over the punctured elliptic curve $E^\circ := E - O$ which have geometrically connected fibers over $S$.\footnote{A $G$-torsor $f : X\ra Y$ is a finite \'{e}tale map equipped with a $Y$-linear action of $G$ which acts freely and transitively on every geometric fiber.} Let $\cT_G$ denote the sheafification of $\cT_G^\pre$ relative to the \'{e}tale topology on $\cM(1)$ inherited from $(\Sch/\bS)_\et$. This means that a covering of $E/S$ in this topology consists of an \'{e}tale covering $\{S_i\ra S\}_{i\in I}$ of $S$, together with the corresponding \'{e}tale covering $\{E_{S_i}\ra E\}_{i\in I}$ of $E$.

\begin{defn}[$G$-structures, see {\cite[\S2]{Chen18}}, {\cite[\S2.5]{Chen21}}] A $G$-structure (or Teichm\"{u}ller structure of level $G$) on an elliptic curve $E/S$ is by definition a section of $\cT_G(E/S)$. Let $\cM(G)$ denote the category whose objects are pairs $(E/S,\alpha)$, where $E/S$ is an elliptic curve and $\alpha\in\cT_G(E/S)$ is a $G$-structure. Morphisms in $\cM(G)$ are morphisms in $\cM(1)$ which respect the $G$-structure. There is a natural forgetful map
$$\pi : \cM(G)\lra\cM(1)$$
obtained by forgetting the level structure.
\end{defn}

\begin{remark} A $G$-structure $\alpha$ on $E/S$ is thus given by an \'{e}tale covering $\{S_i\ra S\}_{i\in I}$ of $S$ and a collection of $G$-covers $\{X_i\ra E_{S_i}^\circ\}_{i\in I}$, such that for every $i,j$ the restrictions of the covers $X_i\ra E_{S_i}$ and $X_j\ra E_{S_j}$ to the ``overlap'' $S_{ij} := S_i\times_S S_j$ are isomorphic. However, since there is no requirement that these isomorphisms over the various $S_{ij}$'s satisfy a cocycle condition, the $G$-structure $\alpha$ might not come from an actual $G$-cover of $E^\circ/S$. When $S = \Spec\bQ$, a $G$-structure $\alpha$ on $E/\bQ$ is given by a $G$-cover $f : X\ra E^\circ_{\Qbar}$ whose \emph{field of moduli} if $\bQ$. It comes from an actual $G$-cover of $E^\circ$ if and only if $\bQ$ is also a field of definition of $f$. The field of moduli is the intersection of all fields of definition, but it may not be a field of definition \cite[2.6-2.7]{CH85}. If $Z(G) = 1$, then $G$-covers have no nontrivial automorphisms, so any cocycle condition is automatic. In this case we would have $\cT_G^\pre = \cT(G)$ and if $S = \Spec\bQ$, then in this case the field of moduli is equal to the minimal field of definition. See \cite[Remark 2.5.1]{Chen21}	
\end{remark}

The connection between $\cM(G)$ and noncongruence modular curves stems from the following key result:

\begin{thm}\label{thm_fet} The forgetful map $\ff : \cM(G)\lra\cM(1)$ is finite \'{e}tale. In particular, $\cM(G)$ is a Deligne-Mumford stack smooth over $\bS := \bZ[1/|G|]$.
\end{thm}


\begin{remark} While the definitions of $\cT_G,\cM(G)$ make sense over $\bZ$, this theorem is not true over $\bZ$. Indeed, the non-finite generation of the \'{e}tale fundamental groups of affine curves in positive characteristic \cite{HOPS18} implies that $\ff$ cannot even be finite when considered as a map of $\bZ$-stacks.
\end{remark}

In the remainder of this section we prove Theorem \ref{thm_fet} by giving an explicit \'{e}tale-local description of the map $\ff$. 

Let $S$ be a connected scheme, let $(E/S,O)$ be an elliptic curve, and let $E^\circ := E - O$. Let $\ol{s}\in S$ be a geometric point. Let $\bL$ denote the set of all primes invertible on $S$, and let $M$ be the smallest (closed) subgroup of $\pi_1(E^\circ,g(\ol{s}))$ such that $K := \Ker(\pi_1(E^\circ,g(\ol{s}))\ra\pi_1(S,\ol{s}))$ is pro-$\bL$.\footnote{A pro-$\bL$ group is a profinite group all of whose finite quotients have orders only divisible by primes in $\bL$.} Define
$$\ol{\Pi} := \pi_1(E^\circ_{\ol{s}},g(\ol{s}))^{\bL}\qquad \Pi := \pi_1(E^\circ,g(\ol{s}))/M\qquad\Delta := \pi_1(S,\ol{s})$$
where $\cdot^{\bL}$ denotes the maximal pro-$\bL$ quotient. Note that by the comparison theorem \ref{thm_comparison} and the specialization theorem \cite[Theorem 5.7.10]{Sza09}, we have
\begin{prop}\label{prop_geometric_pi1} $\ol{\Pi}$ is the pro-$\bL$ completion of a free group of rank 2. In particular, $\ol{\Pi}$ is finitely generated.
\end{prop}

\begin{cor}\label{cor_nonempty} $\cM(G)$ is nonempty if and only if $G$ can be generated by two elements. In this case $\ff : \cM(G)\ra\cM(1)$ is surjective (and finite \'{e}tale).	
\end{cor}
\begin{proof} Surjectivity follows from the fact that $\ff$ is finite \'{e}tale, and hence open and closed, so its image is a connected component of $\cM(1)$, hence the entirety of $\cM(1)$ as long as $\cM(G)$ is nonempty.	
\end{proof}

The maps $E^\circ_{\ol{s}}\ra E^\circ\ra S$ induce a sequence
\begin{equation}\label{eq_HES}
\begin{tikzcd}
	1\ar[r] & \ol{\Pi}\ar[r] & \Pi\ar[r] & \Delta\ar[r] & 1
\end{tikzcd}
\end{equation}
\begin{prop}\label{prop_HES_exact} If $G$ is not the trivial group, then the sequence \eqref{eq_HES} is exact.
\end{prop}
\begin{proof} This is \cite[Proposition 2.1.4]{Chen18}. We sketch the argument here: By \cite[Expos\'{e} XIII, Proposition 4.1]{SGA1}, the sequence $\ol{\Pi}\ra\Pi\ra\Delta\ra 1$ is exact, so it suffices to check injectivity of $\ol{\Pi}\ra\Pi$. If $E^\circ/S$ admits a section, then this injectivity is \cite[Expos\'{e} XIII, Proposition 4.3, Exemples 4.4]{SGA1}. If $G$ is not trivial, then our tameness assumption \eqref{eq_tame} that $|G|$ is invertible on $S$ implies that $\bL$ is nonempty, so there is a prime $p\in\bL$. We can then reduce to the split case by base changing along the map $H\ra S$, where $H$ is the Galois closure of any non-identity component of the finite \'{e}tale $S$-group $E[p]$.
\end{proof}

Thus, if $G$ is nontrivial, then the sequence \eqref{eq_HES} defines an outer representation:
$$\rho_{E/S} : \Delta\lra\Out(\ol{\Pi})$$
Let $\Epi^\ext(\ol{\Pi},G) := \Epi(\ol{\Pi},G)/\Inn(G)$ denote the set of conjugacy classes of surjective homomorphisms $\ol{\Pi}\ra G$. Then $\rho_{E/S}$ defines an action of $\Delta$ on $\Epi^\ext(\ol{\Pi},G)$. Theorem \ref{thm_fet} follows from the following more precise result.
\begin{prop}\label{prop_local_structure_of_TG} Assume $G$ is not the trivial group. Let $S$ be a connected scheme and $\ol{s}\in S$ a geometric point. Let $(E/S,O)$ an elliptic curve, with $E^\circ := E - O$ and corresponding map $S\ra\cM(1)$. Let $\cT_{G,E/S} := \cM(G)\times_{\cM(1)} S$. Then $\cT_{G,E/S}$ is finite \'{e}tale over $S$. Its geometric fiber is also the geometric fiber of $\ff : \cM(G)\ra\cM(1)$, and is in bijection with
\begin{equation}\label{eq_fiber}
\cM(G)_\ol{s} = \cT_{G,E/S,\ol{s}} \rightiso \Epi^\ext(\ol{\Pi},G)
\end{equation}
Relative to this bijection, the monodromy action of $\Delta := \pi_1(S,\ol{s})$ on $\Epi^\ext(\ol{\Pi},G)$ is given by $\rho_{E/S}$. In other words, $\cT_G|_{E/S}$ is represented by the finite \'{e}tale $S$-scheme $\cT_{G,E/S}$, and
$$\cT_G(E/S) \rightiso\{\varphi\in\Epi^\ext(\ol{\Pi},G)\;|\; \la{\delta}\varphi = \varphi\text{ for all $\delta\in\Delta$}\}$$
\end{prop}
\begin{proof} The main idea is to first reduce to the case where $E^\circ$ admits a section $g : S\ra E^\circ$, which can be done by the same trick as in \ref{prop_HES_exact}, using the existence of a prime $p$ invertible on $S$. In this case, the induced map $g_* : \Delta\ra\Pi$ splits the sequence \eqref{eq_HES}, and the outer representation $\rho_{E/S}$ is induced by the honest action of $\Delta$ on $\ol{\Pi}$ defined by $g_*$. By the Galois correspondence, $G$-torsors on $E^\circ$ correspond to $G$-conjugacy classes of homomorphisms $\rho : \Pi\lra G$, which is surjective if and only if the corresponding torsor is connected. Any such homomorphism is in turn determined by the pair $(\varphi := \rho|_{\ol{\Pi}}, \psi := \rho|_{g(\Delta)})$ satisfying
\begin{equation}\label{eq_equivariance}
\varphi(\la{\delta}\gamma) = \psi(\delta)\varphi(\gamma)\psi(\delta)^{-1}\qquad\text{for all $\gamma\in\ol{\Pi},\delta\in\Delta$}	
\end{equation}
Moreover, $\rho$ corresponds to a covering with geometrically connected $S$-fibers if and only if $\varphi := \rho|_{\ol{\Pi}}$ is surjective. Because $\ol{\Pi}$ is finitely generated, there is a finite \'{e}tale map $S'\ra S$ such that \emph{every} $G$-torsor on $E^\circ_{S'}$ with geometrically connected $S'$-fibers has a monodromy representation with $\psi$ the trivial homomorphism. This implies that $\cT_G^\pre|_{E_{S'}/S'}$ is the \emph{constant sheaf}, and hence agrees with $\cT_G|_{E_{S'}/S'}$. The structure of $\cT_{G,E/S}$ then follows from Galois descent. See \cite[\S2-3.1]{Chen18} for more details.	
\end{proof}

\subsubsection{Comparison with classical congruence level structures}\label{sss_comparison_with_congruence}
Recall from Proposition \ref{prop_geometric_pi1} that $\ol{\Pi}$ is the maximal pro-$\bL$ quotient of a profinite free group of rank 2. It's abelianization $\ol{\Pi}^\ab$ is thus the maximal pro-$\bL$ quotient of $\Zhat^2$, which is moreover equipped, via the exact sequence \eqref{eq_HES}, with an action of $\Delta = \pi_1(S,\ol{s})$. Via the Galois correspondence, $\ol{\Pi}^\ab$ can be identified, as a $\Delta$-module, with the fiber over $\ol{s}$ of the pro-object $\{E[n]\}_{n\ge 1}$ whose terms range over $n$-torsion subgroup schemes $E[n]\subset E$. When $S = \Spec K$ for a field $K$, $\ol{s}$ corresponds to an algebraic closure $\ol{K}$ of $K$, $\Delta$ is the absolute Galois group $\Gal(\ol{K}/K)$, and $\ol{\Pi}^\ab$ is just the pro-$\bL$ Tate module $\prod_{\ell\in\bL}T_\ell(E)$ of $E$ \cite[\S III.7]{Sil09}

If $G = (\bZ/n)^2$, then any $\Delta$-invariant surjection $\varphi : \ol{\Pi}\ra G$ factors uniquely through the maximal abelian $n$-torsion quotient $\ol{\Pi}^{\ab,n}$ of $\ol{\Pi}$, which is isomorphic to $(\bZ/n)^2$, but is equipped with the typically nontrivial $\Delta$-action coming from its structure as the fiber over $\ol{s}$ of the finite \'{e}tale group scheme $E[n]_S$. Thus $\varphi$ induces a $\Delta$-invariant isomorphism
$$\ol{\Pi}^{\ab,n}\rightiso G = (\bZ/n)^2$$
which corresponds, via the Galois correspondence, to an isomorphism of group schemes $E[n]_S\cong(\bZ/n)^2_S$. Taking the preimages of the canonical basis of $(\bZ/n)^2_S$ gives a basis of $E[n]_S$, which is exactly a classical ``full level $n$ structure'', associated to the principal congruence subgroup $\Gamma(n)\subset\SL_2(\bZ)$ \cite[\S3.1]{KM85}. Over $\bZ[1/n]$, the moduli stack
$$\cM(n) := \cM((\bZ/n)^2)$$
is finite \'{e}tale over $\cM(1)$, and splits into $\phi(n)$ isomorphic components over $\bZ[1/n,\zeta_n]$, classified by the Weil pairing; it is a scheme if and only if $n\ge 3$ \cite[\S4.7]{KM85}.

If $G = \bZ/n$, then again a $\Delta$-invariant surjection $\varphi : \ol{\Pi}\ra G$ factors uniquely through a surjection
$$\ol{\Pi}^{\ab,n}\lra G = \bZ/n$$
corresponding to a surjection of finite \'{e}tale group schemes
$$E[n]_S\lra (\bZ/n)_S$$
Taking Cartier duals and using autoduality of $E[n]_S$, one obtains an injection $\mu_{n,S}\hookrightarrow E[n]_S$, which is not quite the same as a $\Gamma_1(n)$-structure \cite[\S3.2]{KM85}. To realize $\Gamma_1(n)$-structures in our setting, one should instead consider $G$-structures where $G$ is the finite \'{e}tale $S$-group scheme $\mu_n$. The discussion above remains valid, with the appropriate adjustments; in particular, in Proposition \ref{prop_local_structure_of_TG}, one should replace the $\Delta$-invariance condition in the description of $\cT_G(E/S)$ with $\Delta$-equivariance. Having done this, one finds that $\mu_{n,S}$-structures are equivalent to classical $\Gamma_1(n)$-structures. See \cite[Proposition 2.2.12]{Chen18} for more details. The moduli stack $\cM(\bZ/n)$ is a scheme if and only if $n\ge 4$ \cite[Cor 2.7.3]{KM85}.

\subsection{$G$-structures via smooth admissible $G$-covers}\label{ss_AdmG}
In this section we describe the moduli stack of admissible $G$-covers \cite{ACV03}, and relate it to $\cM(G)$. We first discuss the theory for smooth curves. The nodal case will be treated in \S\ref{ss_compactification}.


\begin{defn} Let $(E/S,O)$ be an elliptic curve. An admissible $G$-cover of $(E/S,O)$ is a finite map $\pi : C\ra E$ equipped with an action of $G$ leaving $\pi$ invariant, such that:
\begin{enumerate}
\item\label{part_gcf} $C\ra S$ is a smooth proper curve with geometrically connected fibers,
\item $\pi$ induces an isomorphism $C/G\rightiso E$, and
\item\label{part_G_torsor} $\pi$ is \'{e}tale over $E^\circ := E - O$.
\end{enumerate}	
\end{defn}


Thus, for any geometric point $\ol{s}\in S$, the map $\pi_\ol{s} : C_\ol{s}\ra E_\ol{s}$ is an admissible $G$-cover of $(E_\ol{s},O)$ in the sense of \S\ref{ss_admissible_covers_over_k}. The \'{e}tale local description of families of admissible covers is the same as a constant family, see \cite[Proposition 6.1.4(a)]{Chen21}.


\begin{defn}\label{def_admG} Let $\cAdm^0(G)$ denote the moduli stack of admissible $G$-covers of elliptic curves: The objects are admissible $G$-covers of elliptic curves, and morphisms are diagrams
\[\begin{tikzcd}
	C'\ar[r]\ar[d] & C\ar[d] \\
	E'\ar[r]\ar[d] & E\ar[d] \\
	S'\ar[r] & S
\end{tikzcd}\]
where each square is cartesian. In particular, to give a map $S\ra\cAdm^0(G)$ is exactly to give an admissible $G$-cover $C\ra E\ra S$. We note that each of the three horizontal arrows in the diagram is included in the data of the morphism.\footnote{See \cite[\S4.1, \S4.3]{ACV03} for a general definition, including the non-equivariant and nonsmooth cases.}
\end{defn}

Forgetting the cover yields a map
$$\ff : \cAdm^0(G)\lra\cM(1)\quad\text{sending}\quad (C\ra E)\mapsto E$$
\begin{thm} The forgetful map $\ff$ is proper, \'{e}tale, and quasi-finite. In particular, $\cAdm^0(G)$ is a smooth Deligne-Mumford stack of relative dimension 1 over $\bZ[1/|G|]$. It is surjective if and only if $G$ is 2-generated.
\end{thm}
\begin{proof} By Corollary \ref{cor_nonempty}, $\cAdm^0(G)$ is nonempty if and only if $G$ is 2-generated. The properness of $\ff$ is \cite[Corollary 3.0.5]{ACV03}. The \'{e}taleness of $\ff$ can be proved via deformation theory, see \cite[Theorem 5.1.5]{BR11} and \cite[Proposition 2.5.3]{Chen21}. Quasi-finiteness is a consequence of $\ff$ being proper and \'{e}tale.
\end{proof}

While a map of schemes is proper and \'{e}tale if and only if it is finite \'{e}tale, for morphisms of stacks, the definition of ``finite'' requires that the map be \emph{representable} \cite[0CHT]{stacks}, which implies that it should induce injections on automorphism groups. If $G$ has nontrivial center, then any nontrivial $g\in Z(G)$ is a nontrivial automorphism of $\pi$, as an object of $\cAdm^0(G)$, which maps to the identity automorphism of $E$ in $\cM(1)$. Thus, if $G$ has nontrivial center, the morphism $\ff$ is not representable, and hence not finite. This non-representability implies that if $\ol{x} : \Spec\Omega\ra\cM(1)$ is a geometric point, then the fiber $\cAdm^0(G)_{\ol{x}}$ is not just a finite set, but a stack (albeit with a finite topological space). This makes it more complicated to set up a Galois correspondence, and hence more complicated to understand the relation between $\cAdm^0(G)$ and noncongruence modular stacks.

For smooth admissible $G$-covers, the only automorphisms come from the center, and hence we can resolve the nonrepresentability of $\ff$ by considering all morphisms in $\cAdm^0(G)$ as being ``up to $Z(G)$''. This process produces a``quotient'' $\cAdm^0(G)\ra \cAdm^0(G)\fs Z(G)$. The general procedure is described in \cite[\S5]{ACV03} and \cite[\S5]{Rom05}, where it is called ``rigidification''. There is a map $\cAdm^0(G)\ra\cM(G)$ sending the admissible cover $\pi : C\ra E$ to the isomorphism class of the $G$-torsor of $E^\circ$ obtained by restriction. In \cite[\S2.5.2]{Chen21} we show:

\begin{thm} The forgetful map $\ff : \cAdm^0(G)\ra\cM(G)$ factors as
$$\cAdm^0(G)\lra \cAdm^0(G)\fs Z(G)\rightiso\cM(G)$$
where the first map is a homeomorphism inducing an isomorphism on coarse schemes.
\end{thm}


\subsection{Functoriality}\label{ss_functoriality}
If $E/S$ is an elliptic curve and $C\ra E$ an admissible $G$-cover, then for any surjection $q : G\twoheadrightarrow H$, $C^\circ/\ker(q)$ is an $H$-torsor of $E^\circ$. This defines a morphism of sheaves $\cT_G\ra\cT_H$, whence a morphism of stacks finite \'{e}tale over $\cM(1)$
$$\cM(q) : \cM(G)\lra\cM(H)$$
At the level of geometric fibers above a fixed elliptic curve $E_0/\bC$ with fundamental group $\Pi := \pi_1(E^\circ,x_0)$, this corresponds to the $\pi_1(\cM(1),E_0)$-equivariant map
\begin{eqnarray*}
q_* : \Epi^\ext(\Pi,G) & \lra & \Epi^\ext(\Pi,H) \\	
\varphi & \mapsto & q\circ\varphi
\end{eqnarray*}
where the monodromy action of $\pi_1(\cM(1),E_0)$ on these sets is through its action on $\Pi$ (Proposition \ref{prop_local_structure_of_TG}). A group theoretic lemma of Gasch\"{u}tz \cite[Proposition 2.5.4]{RZ10} implies that $q_*$ is surjective. Via the Galois correspondence, this has the following geometric consequence:

\begin{prop}[Gasch\"{u}tz] For any surjection of groups $q : G\ra H$, the map $\cM(q) : \cM(G)\lra \cM(H)$ is \emph{surjective}. Two surjections $q,q' : G\ra H$ induce the same map $\cM(G)\ra\cM(H)$ if and only if $q,q'$ differ by conjugation.
\end{prop}

\begin{cor} Let $\cC$ denote the category whose objects are finite 2-generated groups and whose morphisms are surjections up to conjugation. Then $\cM(*) : G\mapsto \cM(G)_\bQ$ defines a faithful and epimorphism-preserving functor from $\cC$ to the category of stacks finite \'{e}tale over $\cM(1)_\bQ$.	
\end{cor}

An immediate consequence of this is the following. Recall from \S\ref{sss_comparison_with_congruence} that the moduli stacks $\cM(n)_\bC := \cM((\bZ/n)^2)_\bC$ are disjoint unions of the congruence modular stacks $[\cH/\Gamma(n)]$. Since every finite 2-generated abelian group is a quotient of $(\bZ/n)^2$ for some $n$, we find that

\begin{prop}[Abelian groups are congruence]\label{prop_abelian_congruence} Let $G$ be a finite 2-generated abelian group which is killed by $n$. Then we have a surjection $(\bZ/n)^2\ra G$, inducing a surjection $\cM(n)\ra\cM(G)$. In particular, every component of $\cM(G)_\bC$ is isomorphic to $[\cH/\Gamma]$ for some congruence subgroup $\Gamma$ of level $n$.	
\end{prop}

\subsubsection{Moduli interpretation of the quotient $\cM(G)^\abs := \cM(G)/\Out(G)$}\label{sss_absolute_moduli}
The free action of $\Out(G)$ on $\Epi^\ext(\Pi,G)$ defines a free action of $\Out(G)$ on $\cM(G)$, and the quotient $\cM(G)/\Out(G)$ is also finite \'{e}tale over $\cM(1)$. The geometric points of $\cM(G)/\Out(G)$ are in bijection with Galois covers of elliptic curves only branched above the origin whose Galois group is isomorphic to $G$, but without a particular choice of the isomorphism. Following \cite[\S1.2]{FV91}, we will write $\cM(G)^\abs := \cM(G)/\Out(G)$. An object of $\cM(G)^\abs$ is an elliptic curve equipped with an ``absolute'' $G$-structure.

\subsubsection{Remarks on functoriality}\label{sss_unfull}
The functor $G\mapsto\cM(G)$ is not full. For example, if $\D_6$ is the dihedral group of order 6, then there is an isomorphism $\cM(\bZ/2)\cong\cM(\D_6)$, but of course no surjection $\bZ/2\ra \D_6$ (see \S\ref{ss_dihedral}). Interestingly, for finite simple $G$, the functor is, empirically speaking, not too far from being full: The stacks $\cM(G)^\abs$ have a total of 860 components as $G$ ranges over the 36 smallest nonabelian finite simple groups (ending with the Janko group $J_1$ of order 175560). Of these, 719 components are primitive covers of $\cM(1)$, in the sense that they admit no nontrivial intermediate covers. We note that two of these components, coming from $\cM(\PSU_3(\bF_4))$ and $\cM(\PSU_3(\bF_5))$, have degree 1 over $\cM(1)$, see \S\ref{ss_burau}. Of these 719, each has monodromy group which is either alternating or symmetric.\footnote{Note that of these 36 groups, 17 are of type $\PSL_2(\bF_p)$, for which the monodromy groups are proven to be predominantly alternating or symmetric \cite[Theorem 1.11]{MP18}.} The remaining 141 non-primitive covers $\cM$ fall into one of three types:

\begin{enumerate}
	\item 25 of these admits a unique (nontrivial) intermediate cover. That cover is isomorphic to $\cM(\bZ/2)\cong[\cH/\Gamma_1(2)]$. These covers are components of $\cM(G)^\abs$, where $G$ is the alternating group $\fA_7,\fA_8$, the orthogonal group $\Or_5(\bF_3)$, or the Mathieu group $\M_{12}$.
	
This implies that from such $G$-structures, one can functorially produce a point of order 2. It would be interesting to explain such a construction geometrically.
	
	\item 115 of these components admits an involution $\alpha$ such that $\cM/\langle\alpha\rangle$ is the unique intermediate cover. Such $\cM$ are components of $\cM(G)^\abs$ where $G = \fA_7,\PSL_3(\bF_3),\PSU_3(\bF_3), \M_{11},\fA_8, \PSL_3(\bF_4), \Or_5(\bF_3)$, the Suzuki group $\Sz(8), \PSU_3(\bF_4), \M_{12}, \PSU_3(\bF_5)$, and the Janko group $\J_1$.

Let $E$ be an elliptic curve over $\bQ$ with \'{e}tale fundamental group $\pi_1(E^\circ_\Qbar)$, containing the topological fundamental group $\pi_1(E^\circ(\bC))$ as a dense subgroup. In these cases, the involution $\alpha$ corresponds to a permutation of $\Epi^\ext(\pi_1(E^\circ_\Qbar),G)$ which centralizes the permutation image of $\Out^+(\pi_1(E^\circ(\bC))$ (this image is the monodromy group of $\cM/\cM(1)$). In some, but not all, cases, one can check by computer that $\alpha$ comes from an element of $\Out^+(\pi_1(E^\circ(\bC))$ of determinant $-1$. In the other cases, one might guess that it comes from an element of $\Gal(\Qbar/\bQ)$ acting on $\pi_1(E^\circ_\Qbar)$ (see \S\ref{sss_arithmetic}). It would be interesting to explain this involution $\alpha$.
 
	\item The remaining cover $\cM$ is the only one to admit two nontrivial intermediate covers $\cM_1,\cM_2$. This cover $\cM$ is the unique component of $\cM(\PSU_3(\bF_5))^\abs$ of degree 40; the admissible covers it parametrizes have ramification index 5. The intermediate covers and their degrees are given as follows:
$$\cM(\PSU_3(\bF_5))^\abs\supset\cM\stackrel{2}{\lra}\cM_1\stackrel{2}{\lra}\cM_2\stackrel{10}{\lra}\cM(1)$$
The monodromy group of $\cM/\cM_2$ is the dihedral group $\D_8$ (of order 8).
\end{enumerate}

\subsection{Compactifying $\cM(G)$ via admissible $G$-covers}\label{ss_compactification}

In this paper we will mainly focus on the stacks $\cAdm^0(G)$ and $\cM(G)$. However for the sake of completeness, in this section we briefly describe how the moduli stacks $\cAdm^0(G),\cM(G)$ can be compactified by considering admissible covers of stable curves. We refer to \cite[\S2]{Chen21}, \cite[\S4]{ACV03}, \cite[\S4]{BR11} for more details.

\begin{defn}[Stable curves] A prestable curve is a flat proper and finitely presented morphism $f : C\ra S$ whose geometric fibers are connected schemes of pure dimension 1 whose only singularities are ordinary double points. The arithmetic genus of its fibers is locally constant on $S$; if it is constant with value $g$, we say that $C/S$ has genus $g$. A prestable $n$-pointed curve is a prestable curve equipped with $n$ disjoint sections $\sigma_1,\ldots,\sigma_n : S\ra C$ whose images are contained in the smooth locus of $f$. The sections $\sigma_1,\ldots,\sigma_n$ are also called \emph{markings}.

A prestable $n$-pointed curve $(C/S,\{\sigma_i\}_{i=1,\ldots,n})$ is \emph{stable} if for every geometric point $\ol{s} : \Spec k\ra S$, the geometric fiber $C_{\ol{s}}$ satisfies any of the following equivalent conditions \cite[\S5, Lemma 1.2.1]{Manin99}
\begin{itemize}
\item $C_\ol{s}$ has only finitely many $k$-automorphisms which fix the sections $\sigma_1,\ldots,\sigma_n$.
\item For any irreducible component $Z\subset C_\ol{s}$ with normalization $Z'$, if $Z'$ has genus 0, then it must contain at least three special points, and if $Z'$ has genus 1, then it must contain at least one special point. Here a point is special if it is the preimage of a node or the image of a section.
\item $\omega_{C_\ol{s}}(\sum_i\sigma_i)$ is ample, where $\omega_{C_\ol{s}}$ is the dualizing sheaf.	
\end{itemize}
\end{defn}

\begin{defn}[1-generalized elliptic curves] A 1-generalized elliptic curve is a stable 1-pointed curve of genus 1. The section will typically be denoted $O$.
\end{defn}

\begin{defn}[Balanced actions] Let $C$ be a prestable curve over an algebraically closed field $k$, equipped with a $k$-linear action of a finite group $G$. For any node $p\in C$, the normalization map $C'\ra C$ induces a decomposition of the cotangent space $T^*_{C,p}$ into a sum of two 1-dimensional subspaces (the branches of the node). The $G$-action is \emph{balanced at $p$} if the action of $G_p := \Stab_G(p)$ preserves this decomposition and acts faithfully via mutually inverse characters on each summand.
\end{defn}

\begin{defn}[Admissible $G$-covers] Let $(E/S,O)$ be a 1-generalized elliptic curve. Its generic locus $E_\gen$ is the complement of all nodes and markings. An admissible $G$-cover of $(E/S,O)$ consists of a finite map $\pi : C\ra E$ equipped with an action of $G$ on $C$ leaving $\pi$ invariant, such that
\begin{enumerate}
\item $C\ra S$ is a prestable curve,
\item $\pi$ induces an isomorphism $C/G\rightiso E$,
\item $\pi$ is \'{e}tale over $E_\gen$, and
\item the $G$-action is balanced at every node of every geometric fiber.
\end{enumerate}
\end{defn}

\begin{remark} This definition differs from that of \cite[Definition 2.1.4]{Chen21}. The equivalence of the two definitions is described in \cite[\S2.4]{Chen21}, using \cite[6.1.4]{Chen21}. In particular, it is shown there that the balanced condition implies that nodes of $C$ must map to nodes of $E$, and that $(C,\pi^{-1}(O))$ is a stable marked curve.
\end{remark}

Let $\cAdm(G)$ denote the moduli stack of admissible $G$-covers, where morphisms are diagrams as in Definition \ref{def_admG}. Let $\ol{\cM(1)}$ denote the moduli stack of 1-generalized elliptic curves, with coarse scheme $\ol{M(1)}$. As in the smooth case, we have a forgetful map
$$\ff : \cAdm(G)\lra\ol{\cM(1)}$$
sending an admissible cover $\pi : C\ra E$ to the 1-generalized elliptic curve $E$.

\begin{thm} Properties of $\cAdm(G)$ and $\ff$.
\begin{enumerate}
\item $\cAdm(G)$ is a smooth proper Deligne-Mumford stack of pure dimension 1 over $\bZ[1/|G|]$, and contains $\cAdm^0(G)$ as a dense open substack.
\item The map $\ff$ is flat, proper, and quasi-finite.
\item $\cAdm(G)$ admits a coarse scheme $\Adm(G)$, and the map on coarse schemes $\Adm(G)\ra\ol{M(1)}$ induced by $\ff$ is finite.
\end{enumerate}
\end{thm}
\begin{proof} See \cite[Theorem 2.1.11]{Chen21}.
\end{proof}

As in the smooth case, $Z(G)$ is a constant subgroup scheme of the automorphism group scheme of any admissible $G$-cover. Thus we can compactify $\cM(G)$ via the rigidifcation $\cAdm(G)\fs Z(G)$.

\begin{thm} Let $\ol{\cM(1)} := \cAdm(G)\fs Z(G)$. Then we have
\begin{enumerate}
\item $\ol{\cM(G)}$ is a smooth proper Deligne-Mumford stack of pure dimension 1 over $\bZ[1/|G|]$.
\item $\ol{\cM(G)}$ contains $\cM(G)$ as a dense open substack.
\item The map $\ol{\cM(G)}\lra\ol{\cM(1)}$ induced by $\ff$ is flat, proper, and quasi-finite.
\end{enumerate}
\end{thm}
\begin{proof} See \cite[Proposition 2.5.10]{Chen21}.
\end{proof}

\begin{remark} Let $E$ be a non-smooth 1-generalized elliptic curve, and $\pi : C\ra E$ an admissible $G$-cover. Taking normalizations, we obtain a smooth admissible $G$-cover $\pi : C'\ra\bP^1$, only ramified above the three points: preimages of the node and origin of $E$. In short, the boundary points, or ``cusps'', of $\cAdm(G)$ correspond to three-point covers (or Belyi maps, or dessins d'enfant) $C'\ra\bP^1$, equipped with a $G$-equivariant bijection between the ramified fibers at $0,\infty\in\bP^1$. This dictionary is worked out in detail in \cite[\S4]{Chen21}. It would be very interesting to work out how this can be used to transport information between the worlds of three-point covers and that of noncongruence modular curves. Note that $C'$ may not be connected (equivalently, $C$ may not be irreducible). This implies that $\pi$ can sometimes admit automorphisms which do not come from elements of $G$; consequently, the map $\ol{\cM(G)}\ra\ol{\cM(1)}$ can fail to be representable.
\end{remark}

\subsection{Coarse schemes}
Recall that the coarse scheme of an algebraic stack $\cX$ is a map $c : \cX\ra X$ with $X$ a scheme which satisfies:
\begin{enumerate}
\item Any map $\cX\ra T$ with $T$ a scheme factors uniquely through 	$c$, and
\item For any algebraically closed field $k$, $c$ induces a bijection $\cX(k)/\cong\rightiso X(k)$.
\end{enumerate}
The coarse scheme, if it exists, is uniquely determined by (a). In the case of $\cM(G)$ over $\bZ[1/|G|]$, the coarse scheme $M(G)$ exists as a $\bZ[1/|G|]$-scheme and can be constructed as follows. For an integer $n$, recall that $\cM(n) := \cM((\bZ/n)^2)$. For an integer $n\ge 2$ dividing $|G|$, $\cN := \cM(G)\times_{\cM(1)}\cM(n^2)\lra\cM(G)$ is finite \'{e}tale over the scheme $\cM(n^2)$ and hence is itself a scheme. The projection to $\cM(G)$ is Galois with group $\GL_2(\bZ/n^2)$, and hence $\cM(G)$ can be recovered as the stack quotient $[\cN/\GL_2(\bZ/n^2)]$. The coarse scheme $M(G)$ can then be realized as the scheme quotient
$$M(G) \cong \cN/\GL_2(\bZ/n^2)$$

\begin{prop}\label{prop_coarse} Let $n\ge 2$ be an integer, and let $\cM$ be a stack finite \'{e}tale over $\cM(1)_{\bZ[1/n]}$. Then $\cM$ admits a coarse scheme, and we have
\begin{enumerate}
\item $M$ is smooth over $\bZ[1/n]$.
\item For a $\bZ[1/n]$-scheme $S$, we say that $M$ satisfies \emph{coarse base change} with respect to $S$ if $M_S := M\times_{\bZ[1/n]} S$ is the coarse scheme of $\cM_S$. Coarse base change is satified in any of the following conditions:
\begin{itemize}
\item $S$ is a regular Noetherian scheme,
\item $S$ is flat over $\bZ[1/n]$,
\item 6 is invertible on $S$.
\end{itemize}
\item $M$ is finite over the $j$-line $\Spec\bZ[1/n][j] \cong M(1)_{\bZ[1/n]}$.
\item The map $\cAdm^0(G)\ra\cM(G)$ induces an isomorphism on coarse schemes. In particular, it is a homeomorphism.
\end{enumerate}
\end{prop}
\begin{proof} Part (a) is the consequence of a general result that the quotient of a smooth affine curve over a regular noetherian scheme $S$ by an $S$-linear action of a finite group is itself smooth. When $S$ is a field this follows from the more elementary fact that rings of invariants of integrally closed domains are themselves integrally closed. For the case where $S$ is regular noetherian, see the Appendix, ``Notes on Chapters 8 and 10'' on p508-510 of \cite{KM85}.

By our construction of coarse schemes, we find that the question of coarse base change is essentially equivalent to the question of whether quotients by finite group actions commute with base change. The latter question is quite interesting and rather subtle. It holds in great generality in the flat case \cite[0DTF]{stacks}, and in the tame case \cite[Proposition A7.1.3]{KM85} (this covers the case when 6 is invertible on $S$). Also see \cite[Proposition 8.1.6]{KM85}. For $S$ regular noetherian, this follows from the proof of part (a), though this uses the fact that $M$ is a curve.

Part (c) follows from the fact that the $M$ is dominated by the finite noetherian $M(1)$-scheme $\cM\times_{\cM(1)}\cM(n)$ \cite[Proposition 8.2.2]{KM85}. Part (d) is a consequence of the universal property of rigidifcation, see \cite[Proposition 2.5.10(a)]{Chen21}.
\end{proof}

\begin{remark} The proper stack $\cAdm(G)$ also admits a coarse scheme $\Adm(G)$, which is smooth over $\bZ[1/|G|]$ for the same reason as \ref{prop_coarse}(a). Its components are thus the smooth compactifications of the components of $M(G)$. See \cite[\S2]{Chen21}.	
\end{remark}

\section{Noncongruence modular curves as moduli of elliptic curves with $G$-structures}\label{section_noncongruence}

\subsection{The components of $\cM(G)_\bC$ as modular stacks}\label{ss_components}
For a finite group $G$, by Theorem \ref{thm_fet} the map $\ff : \cM(G)\ra\cM(1)$ is finite \'{e}tale, and hence the connected components of $\cM(G)_\bC$ are isomorphic to modular stacks $[\cH/\Gamma]$ for various finite index subgroups $\Gamma\le\SL_2(\bZ)$. In this section we describe these subgroups explicitly. We begin with a quick and rough description which is enough for most purposes.

\subsubsection{Combinatorial sketch of the decomposition $\cM(G)_\bC = \bigsqcup_{\Gamma}[\cH/\Gamma]$} \label{sss_combinatorial_decomposition}
Let $E_0$ be an elliptic curve over $\bC$. Let $x_0\in E_0^\circ(\bC)$ be a base point, and let $\Pi := \pi_1(E_0^\circ(\bC),x_0)$ be the topological fundamental group of its space of complex points, a free group of rank 2. Let $a : \Pi\ra\bZ^2$ be a surjection, identifying $\bZ^2$ with the abelianization of $\Pi$. By a theorem of Nielsen (see Theorem \ref{thm_MCG_actions}), $a$ induces an isomorphism
$$a_* : \Out^+(\Pi)\rightiso\SL_2(\bZ)$$
The group $\Out^+(\Pi)$ acts naturally on $\Epi^\ext(\Pi,G)$ by acting on $\Pi$, and through $a_*$ we obtain an action of $\SL_2(\bZ)$ on $\Epi^\ext(\Pi,G)$. Viewing $\Epi^\ext(\Pi,G)$ as the fiber $\ff^{-1}(E_0)\subset\cM(G)$ as in Proposition \ref{prop_local_structure_of_TG}, by the Galois correspondence we find that the $\SL_2(\bZ)$ orbits are in bijection with the connected components of $\cM(G)_\bC$. More precisely, the $\SL_2(\bZ)$ orbit of $\varphi$ corresponds to a component $\cM\subset\cM(G)_\bC$ which is isomorphic to $[\cH/\Gamma_\varphi]$, where $\Gamma_\varphi := \Stab_{\SL_2(\bZ)}(\varphi)$, and the quotient is via the right action \eqref{eq_right_action}. Different choices of $a$ or $\varphi$ (in the same $\SL_2(\bZ)$-orbit) lead to conjugate subgroups $\Gamma_\varphi\le\SL_2(\bZ)$.

\begin{defn} In the situation above, we say that the component $\cM\subset\cM(G)_\bC$ is \emph{uniformized} by the finite index subgroup $\Gamma_\varphi\le\SL_2(\bZ)$. Any two subgroups uniformizing $\cM$ are conjugate in $\SL_2(\bZ)$.
\end{defn}

\begin{remark} Recall that $\cH/\Gamma$ (quotient via the right action \eqref{eq_right_action}) is the \emph{mirror image} of the quotient by the usual left action by m\"{o}bius transformations \eqref{eq_left_action}, see \S\ref{sss_mirror}.	
\end{remark}

\begin{remark}\label{remark_computing} This description makes the geometry of $\cM(G)_\bC$ highly amenable to computation. Computing the components of $\cM(G)_\bC$ amounts to computing the stabilizers $\Gamma_\varphi$ of the $\Out^+(\Pi)\cong\SL_2(\bZ)$-action on the finite set $\Epi^\ext(\Pi,G)$, which is efficiently doable in a variety of computer algebra packages. Given $\Gamma_\varphi$, the ramification indices of $\cH/\Gamma_\varphi\ra\cH/\SL_2(\bZ)$ at points above 0, 1728, and $i\infty$ are exactly the cycles in the permutation images of $\spmatrix{1}{1}{-1}{0},\spmatrix{0}{1}{-1}{0},\spmatrix{1}{1}{0}{1}$ respectively on the coset space $\SL_2(\bZ)/\pm\Gamma_\varphi$ \cite[\S1.3]{ASD71}. From this, one can calculate the genus using the Riemann-Hurwitz formula, and in low degrees one can often compute equations for the modular curves via the Belyi map $\cH/\Gamma_\varphi\ra\cH/\SL_2(\bZ)$ \cite{SV14}. Some examples are discussed in \S\ref{section_examples}.
\end{remark}


\subsubsection{Components of $\cM(G)_\bC$ via Teichm\"{u}ller uniformization}\label{sss_Teichmuller_uniformization}
Here we give a more careful explanation the picture described in \S\ref{sss_combinatorial_decomposition}. Recall that the fiber of $\ff : \cM(G)_\bC\ra\cM(1)_\bC$ above $E_0$ is identified with the set of isomorphism classes of admissible $G$-covers of $E_0$. By the Galois correspondence, we have a bijection
$$\ff^{-1}(E)\cong \Epi^\ext(\Pi,G)$$
Let $S$ be the underlying topological space of $E_0$. We have a Teichm\"{u}ller uniformization
$$u_\cT : \cT(S)\stackrel{\Psi_{f_0}}{\lra}\cH\stackrel{u_\bE}{\lra}\cM(1)$$
where $\Psi_{f_0}$ is the isomorphism from \S\ref{sss_Teichmuller} associated to a fixed framing $f_0$ of $S$, and $u_\bE$ is the map given by the elliptic curve $\bE/\cH$ of \S\ref{sss_framings}. The map $u_\cT$ is a quotient map for the natural right action of $\Gamma(S)$ on $\cT(S)$, and maps $[E_0,\id_{E_0}]\in\cT(S)$ to $E_0\in\cM(1)$. Let $\varphi\in\Epi^\ext(\Pi,G)\cong\ff^{-1}(E_0)$, with corresponding admissible $G$-cover $\pi : C\ra E_0$, which we view as a cover
$$\pi : T\lra S$$
where $T$ is the underlying topological space of $C$. Let $\cM(\pi)\subset\cM(G)_\bC$ be the connected component containing $\pi$. The map $u_\cT$ factors through
\begin{eqnarray*}
u_\pi : \cT(S) & \lra & \cM(\pi) \\	
{[E,m]} & \mapsto & m\circ\pi : T\stackrel{\pi}{\lra} S\stackrel{m}{\lra} E
\end{eqnarray*}
so that $u_\pi$ is the universal cover of $\cM(\pi)$. Let $m_0 : S\ra E_0$ be the identity map at the level of topological spaces, and let $\Gamma_\pi\le\Gamma(S^\circ) = \Gamma(S)$ be the subgroup consisting of those $\alpha$ satisfying
\begin{equation}\label{eq_abuse}
m_0\circ\alpha\circ\pi\cong m_0\circ\pi	
\end{equation}
as $G$-covers of $E^\circ$, where here we treat $\alpha$ as a homeomorphism. Then $\cM(\pi)$ is isomorphic to the stack quotient $[\cT(S)/\Gamma_\pi]$. This picture can be transported to $\cH$ via the isomorphism $\Psi_{f_0} : \cT(S)\rightiso\cH$, which is equivariant relative to the isomorphism (see \S\ref{ss_topological_properties})
\begin{equation}\label{eq_isoms}
\Gamma(S)\cong\Gamma(S^\circ)\cong\Out^+(\Pi)\cong\SL(H_1(S,\bZ))\stackrel{(f_0^{-1})_*}{\lra}\SL_2(\bZ)
\end{equation}

\begin{prop} Under the isomorphisms \eqref{eq_isoms}, $\Gamma_\pi$ corresponds to the $\SL_2(\bZ)\cong\Out^+(\Pi)$-stabilizer $\Gamma_\varphi$ of $\varphi\in\Epi^\ext(\Pi,G)$.
\end{prop}
\begin{proof} If $m : S\ra E$ is a marking, then $m\circ\pi\cong (m^{-1})^*\pi$. Since $m_0$ is the identity on topological spaces, $\alpha\in\Gamma_\pi$ if and only if $(\alpha^{-1})^*\pi\cong \pi$, or equivalently $\varphi\circ(\alpha^{-1})_* = \varphi$ as elements of $\Epi^\ext(\Pi,G)$, where $(\alpha^{-1})_*$ is the outer automorphism of $\Pi$ induced by $\alpha^{-1}$.	
\end{proof}
Thus $\Psi_{f_0}$ induces isomorphisms
$$\cM(\pi)\cong[\cT(S)/\Gamma_\pi]\cong[\cH/\Gamma_\varphi]$$
In particular, the coarse scheme $M(\pi)$ is the modular curve $\cH/\Gamma_\varphi$.

\begin{remark}\label{remark_topologically_equivalent} The stack $\cM(\pi)$ can be given an intrinsic description as follows. Let $\pi : C\ra E$ and $\pi' : C'\ra E'$ be admissible $G$-covers. We say that $\pi$ is topologically equivalent to $\pi'$ if there exists a $G$-equivariant homeomorphism $C\cong C'$ and a homeomorphism $E\cong E'$ compatible with the covering maps. Then $\cM(\pi)$ is the moduli space of all admissible $G$-covers topologically equivalent to $\pi$.
\end{remark}

\begin{remark}\label{remark_origami} This Teichm\"{u}ller uniformization is closely related to the construction of Teichm\"{u}ller curves in the context of Teichm\"{u}ller dynamics, and gives a nonlinear way to recover the symmetric space description of $\cH$, and the double-coset descriptions of modular curves. This will be explained further in \S\ref{sss_EM} below.
\end{remark}

\subsection{Universality of the moduli stacks $\cM(G)$}\label{ss_universality}
In \S\ref{ss_components} described how we may associate, to any group $G$, a collection of modular stacks $[\cH/\Gamma]$ which arise as the components of $\cM(G)_\bC$. By results of Asada and Ellenberg-McReynolds, essentially all modular stacks can be produced in this way. We first state their results, phrased in a way which is most relevant to our setting. At the end of the section we record some remarks on the original context of their results. We begin with the result of Asada.

\begin{thm}[\cite{Asa01}]\label{thm_asada} Let $\Pi$ be a free group of rank 2. For any finite index subgroup $\Gamma\le\Out^+(F_2)$, there exists a finite group $G$ and a surjection $\varphi : \Pi\lra G$ such that $\Gamma$ contains the $\Out^+(\Pi)$-stabilizer $\Gamma_\varphi$ of $\varphi$ in $\Epi^\ext(\Pi,G)$.
\end{thm}

\begin{cor} For any finite index subgroup $\Gamma\le\SL_2(\bZ)$, the modular stack $[\cH/\Gamma]$ is dominated, as a cover of $\cM(1)_\bC$, by a component of $\cM(G)_\bC$ for some finite group $G$.
\end{cor}

It follows that every modular stack can be viewed as the moduli space of elliptic curves equipped with an equivalence class of $G$-structures, for some suitable notion of equivalence, whence the assertion of \S\ref{ss_level_structures_intro}. A group-theoretic proof of Asada's theorem is given in \cite[\S5]{BER11}, where they also show to construct, from $\Gamma$, a suitable $\varphi : \Pi\ra G$ such that $\Gamma\supset\Gamma_\varphi$.

The result of Asada is made more precise by a theorem of Ellenberg-McReynolds \cite[Theorem 1.2]{EM12}, which has the following interpretation in terms of $G$-structures:



\begin{thm}[{\cite{EM12}}] For any subgroup $\Gamma\le\Gamma(2) := \Ker(\SL_2(\bZ)\ra\SL_2(\bZ/2))$ containing $-I$, there is an integer $d$ and a transitive subgroup $G\le\fS_d$ such that the modular stack $[\cH/\Gamma]$ is isomorphic to a component of $\cM(G)_\bC/N$, where $N := N_{\fS_d}(G)/\Inn(G)$.
\end{thm}

Since conjugacy classes of subgroups of $\Pi$ correspond to isomorphism classes of admissible covers of $E$, we obtain the following geometric interpretation:

\begin{cor} For any subgroup $\Gamma\le\Gamma(2)$ containing $-I = \spmatrix{-1}{0}{0}{-1}$, there is an admissible cover $\pi : T\ra S$ of a pointed torus $S$ such that the modular curve $\cH/\Gamma$ is isomorphic to the coarse moduli space of covers of elliptic curves topologically equivalent to $\pi$.\footnote{If $S,S'$ are pointed tori and $\pi : T\ra S$, $\pi' : T'\ra S'$ are admissible covers, we say that $\pi,\pi'$ are topologically equivalent if there are homeomorphisms $S\ra S'$ and $T\ra T'$ which commute with the covering maps (see Remark \ref{remark_topologically_equivalent}).}
\end{cor}

By Belyi's theorem, every smooth projective algebraic curve over $\Qbar$ is the compactification of a cover of $\cH/\Gamma(2)$, hence a modular curve. Thus it follows that

\begin{cor} Every algebraic curve over $\Qbar$ is isomorphic to a component of $\cM(G)_\Qbar/N$, where $G$ is an appropriate transitive subgroup of $\fS_d$ and $N := N_{\fS_d}(G)/\Inn(G)$.	
\end{cor}

\subsubsection{Remarks on Asada's theorem and congruence subgroup topologies} In \cite{Asa01}, Theorem \ref{thm_asada} is referred to as \emph{the congruence subgroup property of $\Out^+(\Pi)$}. In general, for any finitely generated group $P$ and normal subgroup $N\le P$, define
$$\Gamma[N] := \{\gamma\in\Aut(P)\;|\; \text{$\gamma(N) = N$ and $\gamma$ induces the identity on $P/N$}\}$$
These are finite index subgroups which correspond to the groups $\Gamma(n)$ and $\Gamma_1(n)$ in the case $P = \bZ^2$. The profinite topology induced by the groups $\Gamma(N)$ is called the \emph{congruence subgroup topology} on $\Aut(P)$, and the statement that every finite index subgroup contains some $\Gamma[N]$ is equivalent to the statement that the congruence subgroup topology $\tau_c$ on $\Aut(P)$ coincides with the full profinite topology $\tau_{pf}$; in this case we say that $\Aut(P)$ (and hence $\Out(P)$) \emph{has the congruence subgroup property}. The topology $\tau_c$ is natural in the sense that if $\what{P}$ denotes the profinite completion of $P$, then the closure of $\Aut(P)$ inside $\Aut(\what{P})$ (equipped with the compact-open topology) is isomorphic to the $\tau_c$-completion of $\Aut(P)$. From this perspective, the congruence subgroup property for $\Aut(P)$ is equivalent to the injectivity of the natural map $\what{\Aut(P)}\ra\Aut(\what{P})$ induced by the canonical inclusion $\Aut(P)\hookrightarrow\Aut(\what{P})$. See \cite{BER11} and \cite[\S4.4]{RZ10} for more details.

Similarly, we have analogous notions of congruence subgroups for finite index subgroups of $\Aut(P)$ and $\Out(P)$. The fact that $\SL_2(\bZ)$ is finite index in both $\Aut(\bZ^2)$ and $\Out(\Pi)$ leads to two notions of congruence subgroups, each leading to different notions of level structures and moduli interpretations, the former classical (abelian), the latter nonabelian. The theorem of Asada shows that $\SL_2(\bZ)$ has the congruence subgroup property in the latter sense. On the other hand, it is well known that $\SL_2(\bZ)$ does not have the congruence subgroup property in the former classical sense; this failure is equivalent to the fact that the map
$$\what{\SL_2(\bZ)}\lra\SL_2(\Zhat)$$
induced by the inclusion $\SL_2(\bZ)\hookrightarrow\SL_2(\Zhat)$ has a nontrivial kernel, called the \emph{congruence kernel}\footnote{By a result of Mel'nikov, this kernel is a free profinite group of countable rank \cite{Mel76}.}. We note that while $\what{\SL_2(\bZ)}$ is the fundamental group of $\cM(1)_\bC$, the modern adelic formulation of modular forms only sees the group $\SL_2(\Zhat)$. In \cite[Theorem 5.5.11]{Chen18}, it is shown that the unbounded denominators conjecture, now a theorem \cite{CDT21}, is equivalent to the statement that the congruence kernel of $\SL_2(\bZ)$ is exactly the image of the \'{e}tale fundamental group of a certain ring of bounded denominators power series.

\subsubsection{Remarks on Ellenberg-McReynold's theorem and Veech groups}\label{sss_EM}
\begin{thm}[{Ellenberg-McReynolds \cite[Theorem 1.2]{EM12}}] Every finite index subgroup of $\Gamma(2)$ containing $\{\pm I\}$ is the Veech group of a square-tiled surface.	
\end{thm}

We explain these terms and how they relate to our setting. Let $E_0 := \bC/\langle 1,i\rangle$ be the ``square elliptic curve''. If one chooses a nonzero differential $\omega$ on $E_0$, integrating $\omega$ endows $E_0$ with a translation structure: an atlas $\mu = \mu_\omega$ of local charts whose transition functions are translations. We will identify $\bR^2\cong\bC$ using the basis $1,i\in\bC$, relative to which we will view the charts of $\mu$ as landing in $\bR^2$. If $\pi : C\ra E_0$ is a covering only branched above the origin, then the translation structure $\mu$ can be pulled back to a translation structure on $C$ with singularities at ramified points; the surface $C$ is seen to be tiled by translates of the square fundamental domains of $E_0$. In general a \emph{square-tiled surface} is a translation surface which can be obtained in this way. Let $C^\circ := C - \pi^{-1}(O)$. An affine diffeomorphism of the translation surface $(C^\circ,\pi^*\mu)$ is a diffeomorphism $f : C^\circ\ra C^\circ$ which locally on charts is given by $z\mapsto Az+b$, where $A\in\GL_2^+(\bR)$ and $b\in\bR^2$. The matrix $A$ is its \emph{derivative}. The \emph{Veech group} $\Gamma(C^\circ,\pi^*\mu)$ is the group of derivatives of all affine diffeomorphisms of $C^\circ$. Let $\Gamma(C^\circ,\pi^*\mu)_G$ be the subgroup of $G$-equivariant diffeomorphisms.

Let $S$ be the underlying topological space of $E_0$. For any flat structure $\nu$ on $S$ with associated complex structure $[\nu]$, the elliptic curve $(S,[\nu])$ marked by the identity map $\id : S\ra (S,[\nu])$ defines a point of $\cT(S)$. One easily checks that the map
\begin{equation}\label{eq_Teichm\"{u}ller}
\begin{array}{rcl}
\GL_2^+(\bR) & \lra & \cT(S) \\
A & \mapsto & [(S,[A\mu]),\id]	
\end{array}\quad\text{induces an isomorphism}\quad\bC^\times\bs\GL_2^+(\bR)\rightiso\cT(S)
\end{equation}
where $\bC^\times$ is identified with the subgroup $\SO_2(\bR)\cdot\bR_{>0}\le\GL_2^+(\bR)$. Suppose now that $\pi : C\ra E_0$ is an admissible $G$-cover, and $M(\pi)$ is its component in $M(G)_\bC$. Using the work of Schmithusen \cite{Sch05}, we have:

\begin{prop} Via \eqref{eq_Teichm\"{u}ller}, the Teichm\"{u}ller uniformization $\cT(S)\ra M(\pi)$ of \S\ref{sss_Teichmuller_uniformization} induces an isomorphism
\begin{equation}\label{eq_dcu}
[\bC^\times\bs\GL_2^+(\bR)/\Gamma(C^\circ,\pi^*\mu)_G]\rightiso \cM(\pi)	
\end{equation}
\end{prop}
\begin{proof} To see the isomorphism at the level of coarse schemes, the key fact is that every affine diffeomorphism of $(C^\circ,\nu)$ descends to one on $(E^\circ,\nu)$ \cite[Prop 2.6(3)]{Sch04}. At the level of stacks, one should use an equivariant analog of \cite[Prop 2.7]{Sch04}.	
\end{proof}

The set of all flat structures on $S$ is the complement of the zero section of the cotangent bundle on $\cT(S)$, and is a torsor under $\GL_2^+(\bR)$. In the description above, we implicitly use the ``cotangential base point'' $\mu$ to identify this bundle with $\GL_2^+(\bR)$. In the context of Teichm\"{u}ller dynamics, the action of $\GL_2^+(\bR)$ describes a geodesic flow; the image of $M(\pi)$ in the moduli space of curves of genus $g(C)$ is an isometry relative to the hyperbolic metric on $M(\pi)$ and the Teichm\"{u}ller metric on the moduli space of curves, and is called a \emph{Teichm\"{u}ller curve}. Also see \cite{Chen17, HS07, HS09, Her12, Zor06}.

The isomorphisms \eqref{eq_Teichm\"{u}ller} and \eqref{eq_dcu} are thus nonlinear versions of the symmetric space descriptions of $\cH\cong\cT(S)$ and modular curves, which recover the classical picture when the Veech group is congruence. In addition to telling a nice story, this process of deforming $\mu$ by $\GL_2^+(\bR)$ gives a geometric way to interpolate between mapping classes in $\Gamma(S^\circ)$, similar to how one can interpolate between elements of $\SL_2(\bZ)$ using elements of $\SL_2(\bR)$. This can be used to give a geometric description of the Hecke operators $T_p$ on $\cM(\pi)$. This will be worked out in a future paper.

\subsection{Representability and fine moduli}
A stack $\cX$ is representable if it is isomorphic to a space, or equivalently if the coarse map $c : \cX\ra X$ is an isomorphism. If $\cX$ is a moduli stack, then we say that the representing space $X$ is a \emph{fine moduli space}. In \S\ref{ss_components} we described the components of $\cM(G)_\bC$ as modular stacks $[\cH/\Gamma]$, and hence the component is representable (by a scheme) if and only if $\Gamma$ acts freely on $\cH$, or equivalently if $\Gamma$ is torsion-free. In general, a necessary condition is that the objects of $\cX$ should have no nontrivial automorphisms. If $\cX$ is an algebraic stack admitting a coarse scheme then this condition is also sufficient \cite[04SZ,03XX]{stacks}. Since automorphisms of $\cM(G)$ are determined by automorphisms of elliptic curves, and an automorphism of a family elliptic curves is the identity if and only if it restricts to the identity on a single geometric fiber \cite[Cor 6.2]{MFK94}, we find that it $\cM(G)$ is representable if and only if its geometric points have no nontrivial automorphisms. This implies that representability is invariant under base change by any $S\ra\Spec\bZ[1/n]$.

In characteristic 0, automorphisms can be viewed as representing mapping classes of the underlying topological surface, and hence by the discussion of \S\ref{ss_components}, we find that a substack $\cM\subset\cM(G)_\bQ$ is representable \emph{if and only if} every component of $\cM_\bC$ is uniformized by a torsion-free subgroup $\Gamma$. In positive characteristic, an argument using Deuring's lifting theorem \cite[\S13.5]{Lang87} allows us to reduce to the characteristic 0 case. In particular, we have

\begin{thm} Let $A$ be a subring of $\bC$ in which $|G|$ is invertible, and let $\cM$ be a component of $\cM(G)_A$. Then $\cM$ is representable (by a scheme) if and only if every component of $\cM_\bC$ is uniformized by a torsion-free subgroup of $\SL_2(\bZ)$.
\end{thm}
\begin{proof} See \cite[Theorem 3.5.3]{Chen18}.	
\end{proof}

\subsection{Arithmetic and geometry of noncongruence modular curves}\label{ss_arithmetic_geometry}
The arithmetic and geometry of $\cM(G)$ can be described in terms of two actions on the fundamental group of a punctured elliptic curve. The arithmetic is determined by the action of the absolute Galois group, and the geometry is determined by the action of the mapping class group. In this section we describe this framework. Here we work universally over $\bS = \Spec\bQ$.


Let $(E,O)$ be an elliptic curve over $\bQ$. The base change $E_\Qbar$ defines a geometric point of $\cM(1)$. Let $\Gamma_{E/\bQ} := \pi_1(\cM(1),E_\Qbar)$, $\Gamma_{E_\Qbar} := \pi_1(\cM(1),E_\Qbar)$, and $\Gal_\bQ := \Gal(\Qbar/\bQ)$. Then we have an exact sequence
\[\begin{tikzcd}
1\ar[r] & \Gamma_{E_\Qbar}\ar[r] & \Gamma_{E/\bQ}\ar[r] & \Gal_\bQ\ar[r]\ar[l,bend right=20,"E"'] & 1	
\end{tikzcd}\]
which is split by the $\bQ$-point corresponding to $E$, making $\Gamma_{E/\bQ}$ into a semidirect product $\Gamma_{E/\bQ}\cong\Gamma_{E_\Qbar}\rtimes\Gal_\bQ$. Let $x_0\in E^\circ(\bC)$, let $\Pi := \pi_1(E^\circ(\bC),x_0)$ be the topological fundamental group, and let $\Pi_\Qbar := \pi_1(E^\circ_\Qbar,x_0)$ be the \'{e}tale fundamental group, the profinite completion of $\Pi$. Let $\Pi_\bQ := \pi_1(E^\circ,x_0)$. Viewing $E_\Qbar$ as a geometric fiber of the universal elliptic curve $\cE\ra\cM(1)$, we have an exact sequence
$$1\lra\Pi_\Qbar\lra\pi_1(\cE,x_0)\lra\Gamma_{E/\bQ}\lra 1$$
which leads to an outer representation
$$\rho_{E^\circ/\bQ} : \Gamma_{E/\bQ}\cong\Gamma_{E_\Qbar}\rtimes\Gal_\bQ\lra\Out(\Pi_\Qbar)$$
From \S\ref{ss_combinatorial_structures}, we see that the fiber of the forgetful map $\ff : \cM(G)\ra\cM(1)$ is identified with
$$\ff^{-1}(E_\Qbar) = \Epi^\ext(\Pi_\Qbar,G) = \Epi^\ext(\Pi,G)$$
and by the Galois correspondence, the arithmetic structure of $\cM(G)$ is completely determined by the monodromy action of $\Gamma_{E/\bQ}\cong \Gamma_{E_\Qbar}\rtimes\Gal_\bQ$ on this set, which has two parts: the geometric part, coming from the action of the profinite mapping class group $\Gamma_{E_\Qbar}$, and the arithmetic part, coming from the action of the Galois group $\Gal_\bQ$.

\subsubsection{Geometric monodromy: Action of the mapping class group}\label{sss_geometric_monodromy}
The topological fundamental group of $\cM(1)^\tp\cong[\cH/\SL_2(\bZ)]\cong[\cT(E(\bC))/\Gamma(E(\bC))]$ is the mapping class group $\Gamma(E(\bC))\cong\Gamma(E^\circ(\bC))$. By the comparison theorem, this group embeds into the \'{e}tale fundamental group $\Gamma_{E_\Qbar}$, identifying the latter with the profinite completion of the former. The geometric monodromy action of $\Gamma_{E_\Qbar}$ on $\Pi_\Qbar$, and hence $\ff^{-1}(E_\Qbar)$ is induced by the natural outer action of $\Gamma(E^\circ(\bC))\subset\Gamma_{E_\Qbar}$ on the dense subgroup $\Pi\subset \Pi_\Qbar$. As described in \S\ref{ss_topological_properties}, this action induces an isomorphism $\Gamma(E^\circ(\bC))\rightiso\Out^+(\Pi)$.

Thus, the geometric monodromy action is just the tautological outer action of the special outer automorphism group of a free group of rank 2 on the free group. The induced actions on the sets $\Epi^\ext(\Pi,G)$ determines the geometry of $\cM(G)_\bC$, and remains a remarkably subtle problem. Even the most basic question of determining the set of connected components $\pi_0(\cM(G)_\bC)$ of $\cM(G)_\bC$ is quite difficult. For this, a first observation is that the components correspond to orbits of the monodromy action of the mapping class group $\Gamma(E^\circ(\bC))$, which preserves the ramification type. For us, this means

\begin{prop}\label{prop_higman} Let $a,b$ be a basis for the free group $\Pi$. The following equivalent statements are true:
\begin{enumerate}
	\item $\Out^+(\Pi)$ preserves the conjugacy class of the commutator $[b,a]$.\footnote{This can also be verified group theoretically on generators of $\Out^+(\Pi)$. One generating set is represented by the automorphisms $(a,b)\mapsto (a,ab)$ and $(a,b)\mapsto (b,a^{-1})$.}
	\item For any finite group $G$ and $\varphi\in\Epi^\ext(\Pi,G)$, $\varphi([b,a])$ is $G$-conjugate to $(\varphi\circ\gamma)([b,a])$ for any $\gamma\in\Out^+(\Pi)$.
	\item The Higman invariant is constant on connected components of $\cM(G)_\bC$.
\end{enumerate}
\end{prop}

One can gain an appreciation for the problem by considering it in the context of understanding the connected components of Hurwitz spaces in terms of discrete geometric invariants of covers. Let $H_{g,n}$ denote the moduli space of covers of genus $g$ curves with $n$ branch points. We summarize three foundational results in this area:

\begin{itemize}
\item Let $H_{0,n}^{sb}\subset H_{0,n}$ be the subspace classifying covers with \emph{simple branching}\footnote{This means that above every branch point there is exactly one ramification point, and that point has index 2. In some sense this is the generic case.}. In the late 1800's, Clebsch, Luroth, and Hurwitz showed that the components of $\cH_{0,n}^{sb}$ are classified completely by a single discrete invariant: the degree $d$. If $n$ is even and $n\ge 2d-2$ then there is exactly one component classifying covers of degree $d$. Otherwise, there are none (see \cite{Cleb1873, Hurw1891, Ful69}). Taking $d$ large enough so that every curve of genus $g$ admits a degree $d$ map to $\bP^1$ with simple branching over $n$ points, this led to the first proof of the connectedness of the moduli stack $\cM_g$ of curves of genus $g$.
\item Let $H_{g,n}(G)$ denote the moduli space classifying $G$-covers of genus $g$ curves with $n$ branch points. In 1988, Conway and Parker showed that for a fixed group $G$ and ramification data satisfying certain properties, then the subspace of $H_{0,n}(G)$ classifying covers with the given ramification data is \emph{connected} for $n\gg 0$ (depending on $G$). We call results of this type \emph{branch stabilization}, see \cite{FV91} and \cite[Proposition 3.4]{EVW16}. In \cite{EVW16}, Ellenberg, Venkatesh, and Westerland used this result to establish analogs of the Cohen-Lenstra heuristics for function fields. While the meat of their paper establishes a result on the \emph{stable} homology of Hurwitz spaces, the issue of connected components (equivalently, understanding the zeroth Betti number) played a crucial role in their analysis.
\item If $D$ is a closed oriented surface of genus $g$, any $G$-cover $\pi : C\ra D$ corresponds to a map $D\ra BG$, whence a map $H_2(D)\ra H_2(BG)\cong H_2(G)$. The image of the fundamental class in $H_2(D)$ determines an element of $H_2(G)$, called the \emph{Euler class} associated to $\pi$. In 2006, Dunfield and Thurston \cite{DT06} showed that for fixed $G$, the components of $H_{g,0}(G)$ are classified completely by the Euler class for $g\gg 0$ (depending on $G$). We call results of this type \emph{genus stabilization}.
\end{itemize}

\begin{remark} There are also results interpolating between branch and genus stabilization, see \cite{CLP16, Lonne18}.	
\end{remark}

\begin{remark} For $(g,n) = (0,3)$, $H_{0,3}$ is 0-dimensional, but one can still try to classify the components of $H_{0,3,\bQ}$. This is the classical problem of finding discrete invariants for Galois orbits of dessins d'enfants \cite{Schn97}.	
\end{remark}

In the context described above, the problem of understanding $\pi_0(\cM(G)_\bC)$ lies in the realm of what might be called ``monodromy stabilization'', about which much less is known. For this, one should restrict the problem to certain families of finite group $G$. We summarize some recent work in this direction:

\begin{itemize}
\item In \cite[Corollary 1.5]{Mcm05}, as a byproduct of his classification of Teichm\"{u}ller curves in genus 2, McMullen showed that for $n\ge 4$, the subspace of $\cM(\fA_n)_\bC/\Out(\fA_n)$ classifying covers with Higman invariant the class of either $(12)(34)$ or $(123)$ has at most two components.
\item In a forthcoming joint work with P. Deligne \cite{CD17}, we show that for certain metabelian groups $G$ satisfying a certain homological criterion, the components of $\cM(G)_\bC$ are entirely classified by the Higman invariant. This will be further discussed in \S\ref{ss_metabelian}.
\item In \cite{Chen21}, building on work of Bourgain, Gamburd and Sarnak \cite{BGS16arxiv}, the author shows that for $p$ sufficiently large, the substack of $\cM(\SL_2(\bF_p))/\Out(\SL_2(\bF_p)$ classifying covers with ramification index $2p$ is connected. This problem is related to the Diophantine geometry of the Markoff equation. This will be further discussed in \S\ref{ss_markoff}.
\end{itemize}

The methods used in the results summarized above varied greatly. McMullen's theorem was proven using techniques of Teichm\"{u}ller dynamics and additive number theory; the result on metabelian groups used mostly commutative and homological algebra, and the result on $\SL_2(\bF_p)$ involves a wide variety of techniques in Diophantine geometry over finite fields, the theory of character varieties, and an analysis of degenerations of $\SL_2(\bF_p)$-covers of elliptic curves.

\begin{remark} The combinatorial problem of understanding the $\Out^+(\Pi)$-orbits on $\Epi^\ext(\Pi,G)$ is related to the problem of understanding \emph{Nielsen equivalence classes} or \emph{$T$-systems} of generating pairs of finite groups. More generally, one can replace $\Pi$ with a free group $F_k$ of higher rank $k\ge 2$. Here, it is expected that if $k$ is larger than the size of a minimal generating set of $G$, then the action of $\Aut(F_k)$ on $\Epi^\ext(F_k,G)$ should be \emph{transitive} \cite[Conj 1]{Gar08}. This can be viewed as a combinatorial analog of branch/genus stabilization, and is known for solvable groups \cite{Dun70}, or if $k\ge \log_2(|G|)$ \cite[Cor 3.3]{Lub11}, but is wide open in general. The special case where $G$ is simple is Wiegold's conjecture \cite[Conj 6.1]{Lub11}, first raised in the 1970's.	 The problem of understanding the orbits in rank $k = 2$ is even more open; in \cite[p12]{Pak00}, Pak writes ``it is quite embarrassing how little is known about this problem.''
\end{remark}

\subsubsection{Arithmetic monodromy: Action of the absolute Galois group}\label{sss_arithmetic}

The action of $\Gal_\bQ$ on $\ff^{-1}(E_\Qbar)$ is induced by the outer representation $\rho_{E^\circ/\bQ}$, which can also be recovered from the homotopy exact sequence associated to $E^\circ/\bQ$:
\begin{equation}\label{eq_HES_E}
\begin{tikzcd}
1\ar[r] & \Pi_\Qbar\ar[r] & \Pi_\bQ\ar[r] & \Gal_\bQ\ar[r] & 1
\end{tikzcd}
\end{equation}
The action of $\Gal_\bQ$ on $\ff^{-1}(E_\Qbar)$ can be viewed as a piece of the outer Galois representation $\rho_{E^\circ/\bQ}|_{\Gal_\bQ}$, in the same way that the Galois action on the $n$-torsion $E[n]$ is a piece of the Galois action on the Tate module.


Recall that a smooth curve over $\Qbar$ is \emph{hyperbolic} if it is obtained by removing $n$ points from a smooth proper curve of genus $g$, where $2g-2+n > 0$. Thus, $E^\circ$ is a hyperbolic curve, and $\rho_{E^\circ/\bQ}$ satisfies the following properties, which apply to all outer Galois representations associated to hyperbolic curves:
\begin{prop}\label{prop_anabelian} Let $X$ be a smooth hyperbolic curve over a number field $K$. Let $\Pi_{X,\Qbar} := \pi_1(X,x_0)$ for some base point $x_0\in X$, and let $\rho_{X} : \Gal_K\ra\Out(\Pi_{X,\Qbar})$ be the outer representation induced by the homotopy exact sequence as in \eqref{eq_HES_E}.
\begin{enumerate}
\item $\rho_{X}$ is faithful.
\item If $Y$ is another smooth hyperbolic curve over $K$, the natural map
$$\Isom_K(X,Y)\rightarrow\Isom_{\Gal_K}(\Pi_{X,\Qbar},\Pi_{Y,\Qbar})/\Inn(\Pi_{Y,\Qbar})$$
is a bijection, where the right side is the subset of conjugacy classes of isomorphisms which are $\Gal_K$-equivariant.
\end{enumerate}
\end{prop}
\begin{proof} Part (b) is due to the work of Tamagawa \cite{Tam97} (for affine curves), and Mochizuki \cite{Moc96} (for proper curves), see \cite{MNT01}. The most difficult part of part (a) is the case where $X$ is proper of genus $g\ge 2$, in which case it is a result of Hoshi and Mochizuki \cite[Theorem C]{HM11}. The affine case is \cite[Theorem 2.1]{Mat94}.\footnote{The proofs of the cases $(g,n) = (0,3)$ or $(1,1)$ can be sketched as follows. A nonfaithful action would imply that there is a proper subfield $L\subset\Qbar$ over which all $G$-covers $X$ are defined. If $(g,n) = (0,3)$, $X \cong\bP^1 - \{0,1,\infty\}$ and so this is impossible by Belyi's theorem. If $(g,n) = (1,1)$, this would imply that every component of $M(G)/\Out(G)$ has an $L$-rational point, and hence is defined over $L$. Since every algebraic curve over $\Qbar$ is dominated by a component of $M(G)/\Out(G)$, this would imply every algebraic curve over $\Qbar$ has an $L$-rational point, and is hence defined over $L$, which is false.}
\end{proof}

In particular, the anabelian representation $\rho_{E^\circ/\bQ}$ determines $E$ up to isomorphism\footnote{This is effective in the sense that $E$ can be algorithmically reconstructed from $\rho_{E^\circ/\bQ}$ \cite{Moc04}.}, and hence captures the full complexity of $E$ (as well as $\Gal_\bQ$). We note that the same holds for the $\Gal_\bQ$ action on the fundamental group of $\bP^1 - \{0,1,\infty\}$, which is also free profinite of rank 2. Compared to this situation, in the elliptic case the study of the representation $\rho_{E^\circ/\bQ}|_{\Gal_\bQ}$ (and hence $\Gal_\bQ$) can be naturally stratified via its relation to the abelianized representation $\rho_{E/\bQ} : \Gal_\bQ\ra\Aut(\Pi_\Qbar^\ab)\cong\GL_2(\Zhat)$ associated to the Tate module of $E$. The study of this representation has been famously fruitful, leading to a proof of Fermat's last theorem. It stands to reason that a study of various intermediate quotients may also be similarly fruitful. In \S\ref{ss_metabelian}, we describe a step in this direction by describing the metabelianization of $\rho_{E^\circ/\bQ}$, which remarkably is essentially isomorphic to the abelianization. It would be interesting to see what can be said in the 3-step solvable quotient.

In another direction, one might hope that the Galois action on certain restricted but highly nonabelian quotients of $\Pi_\Qbar$ may be amenable to study. In \S\ref{ss_markoff}, we will say a little about its action on the pro-$\SL_2(\bF_p)$ quotient. A key property that will be leveraged in all situations is the particularly simple $\Gal_\bQ$ action on the generator of inertia of $\Pi_\Qbar$:

\begin{lemma}[``Branch cycle lemma'']\label{lemma_BCL} Let $a,b$ be a positively oriented basis of $\Pi = \pi_1(E^\circ(\bC),x_0)$, and let $c$ be the image of $[b,a]$ in $\Pi_\Qbar = \pi_1(E^\circ_\Qbar,x_0)$ under the embedding $\Pi\hookrightarrow\Pi_\Qbar$ induced by the inclusion $\Qbar\hookrightarrow\bC$. Let $\chi : \Gal_\bQ\ra\Zhat^\times$ denote the cyclotomic character. Then there exists a $\gamma\in\Pi_\Qbar$ such that
$$\sigma(c) = \gamma c^{\chi(\sigma)}\gamma^{-1}\qquad\text{for all $\sigma\in\Gal_\Qbar$}$$
In particular, $\Gal_\bQ$ preserves the rational class\footnote{The rational class of a group element $g$ is the union of all conjugacy classes of $g^i$ where $i$ is coprime to the order $|g|$.} of the Higman invariant of $G$-covers of $E^\circ_\Qbar$.
\end{lemma}
\begin{proof} The proposition holds more broadly for any generator of inertia at a puncture of any curve. See for example \cite[Theorem 2.1.1]{Nak94}.
\end{proof}

\begin{remark}\label{remark_tangential} The lemma is best understood using the theory of \emph{tangential base points}, which can be described as a map $t : \Spec\Qbar\ls{t}\ra E^\circ_\Qbar$ induced by a map
$$\Spec\bQ\ps{t}\cong\Spec\what{\cO}_{E,O}\lra E$$
The fundamental group of $\Spec\Qbar\ls{t}$ is isomorphic to the Tate twist $\Zhat(1)$ as a $\Gal_\bQ$-module\footnote{The Tate twist $\Zhat(1)$ denotes the $\Gal_\bQ$ module with underlying group $\Zhat$ and with $\Gal_\bQ$ action given by exponentiating by the cyclotomic character $\chi_\cyc : \Gal_\bQ\ra\Zhat^\times$. An algebraic closure of $\Qbar\ls{t}$ is given by the field $\Qbar\ls{t^{1/\infty}} := \varinjlim_n \Qbar\ls{t^{1/n}}$ of \emph{Puiseux series}. The fundamental group of $\Spec\Qbar\ls{t}$ is isomorphic to $\Gal(\Qbar\ls{t^{1/\infty}}/\Qbar\ls{t})$.}, and the image of a generator of $\Zhat(1)$ inside $\pi_1(E^\circ_\Qbar,t)$ is a generator of inertia at the puncture. Using the base point $t$, the element $\gamma$ in \ref{lemma_BCL} can be taken to be the identity.
\end{remark}

\subsubsection{Field of definition of components of $\cM(G)_\Qbar$}
The components of $\cM(G)_\bQ$ are typically not geometrically connected. For any such component $\cM$, the absolute Galois group $\Gal_\bQ$ acts transitively on the components of $\cM_\Qbar$. Via the Galois correspondence, the components of $\cM_\Qbar$ correspond to a $\Gal_\bQ$-orbit on the $\Epi^\ext(\Pi_\Qbar,G)/\Gamma_{E_\Qbar} = \Epi^\ext(\Pi,G)/\Out^+(\Pi)$. If a component $\cY\subset\cM(G)_\Qbar$ is the base change of a geometrically connected component over $\cM(G)_K$ for some number field $K$, we say that $\cY$ is \emph{defined over $K$}. In this sense there is a minimal field of definition, given by the fixed field of the $\Gal_\bQ$-stabilizer of the $\Gamma_{E_\Qbar}$-orbit corresponding to $\cY$. We call this field the \emph{field of definition of the component $\cY$}.


\begin{remark} We note that the field of definition of $\cY$ as a component of $\cM(G)_\Qbar$ will typically differ from the field of definition of $\cY$ as a curve. For example, the components of $\cM(n)_\Qbar$ have field of definition $\bQ(\zeta_n)$, but as curves they are in fact all defined over $\bQ$, see \cite[\S3.4]{BBCL20}.
\end{remark}

\section{A tour through some examples}\label{section_examples}
In this section we examine some examples of $\cM(G)_\bC$ for certain groups $G$. More data can be obtained from \cite[Appendix B]{Chen18}; or by contacting the author. In this section we work universally over $\bC$. By default, $\Pi$ will denote the topological fundamental group of a punctured elliptic curve: a free group of rank 2. 

\subsection{The dihedral groups $G = \D_6, \D_8, \D_{10}$}\label{ss_dihedral}
The table below describes the components of $\cM(G)_\bC$ for the dihedral groups $G = \D_6,\D_8,\D_{10}$.
$$\begin{array}{llllllllllll}
	|G| & G & m & d & c_2 & c_3 & -I & \text{cusp widths} & g & \text{Hig} & \text{AbsMon} & \text{c/nc}\\
	\hline
	6 & \D_6 & 1 & 3 & 1 & 0 & 1 & 1^12^1 & 0 & 3A & \fS_3 & \text{cng} \\
	8 & \D_8 & 1 & 6 & 0 & 0 & 1 & 2^3 & 0 & 2A & \fS_3 & \text{cng}\\
	10 & \D_{10} & 2 & 3 & 1 & 0 & 1 & 1^12^1 & 0  & 5A,5B & \fS_3 & \text{cng} \\
\end{array}$$
\textbf{Reading the table.} Each row describes the data associated with a certain conjugacy class of subgroup $\Gamma$ of $\SL_2(\bZ)$ (equivalently, isomorphism class of a component $\cM$ of $\cM(G)$, where isomorphism is as covers of $\cM(1)$). The field $m$ stands for ``multiplicity'', which is the number of isomorphic components of $\cM(G)$ that appear with the given characteristics. Thus, amongst the groups $G$ listed, $\cM(G)$ is connected for all of them except for $\cM(\D_{10})$, which has two isomorphic components.

The field $d$ describes the degree of the component over $\cM(1)$, or equivalently the index of the uniformizing subgroup of $\SL_2(\bZ)$. The field $c_2$ (resp. $c_3$) is the number of unramified points of the component $M$ over $M(1)$ of order 2 (resp. 3). The field $-I$ is 1 or 0, according to whether $-I\in\Gamma$ or $-I\notin\Gamma$ respectively. Thus the component is a (fine moduli) scheme (equivalently $\Gamma$ is torsion-free) if and only if the fields $c_2,c_3,-I$ are all zero. In general, the component is generically a scheme if and only if $-I\notin\Gamma$. In this case there are exactly $c_2$ stacky points with stabilizer $\bZ/4$, and $c_3$ stacky points with stabilizer $\bZ/3$ or $\bZ/6$. The field ``cusp widths'' describes the ramification indices of the map of coarse schemes $M\ra M(1)$ above the missing point at ``$j = \infty$''. It takes the format ``$\text{ramification index}^{\text{\# points with that index}}$''. The fields $(d,c_2,c_3,-I,\text{cusp widths})$ are together called the \emph{signature} of $\cM$ (or $\Gamma$).

The field $g$ denotes the genus of the component $M$ as an algebraic curve, which can be computed from the signature using the Hurwitz formula. The field ``Hig'' describes a label for the Higman invariant of covers parametrized by $\cM$ (see \S\ref{ss_admissible_covers_over_k}). Such a class will be labeled as a number followed by a letter, where the number denotes the order of any element in the class (equivalently, the ramification index), and the letter is an arbitrary label which distinguishes distinct conjugacy classes. We note that there is one value of the Higman invariant for each component of $\cM(G)$ isomorphic to $\cM$; there are cases where isomorphic components can parametrize covers with different Higman invariants (or even different ramification indices)\footnote{See the warning after Table 1 of \cite{Chen18}.}. All of this data was computed, using the computer algebra package GAP, from the subgroup $\Gamma$; see Remark \ref{remark_computing}.

The field AbsMon describes the ``absolute monodromy'' of $\cM$: This is the monodromy group of the image $\cM^\abs$ of $\cM$ inside $\cM(G)^\abs/\Out(G)$ as a cover of $\cM(1)$. This will be a list of entries of the form ``${_n}P$'', one for each minimal intermediate cover of $\cM(1)$, where $n$ denotes the degree of $\cM^\abs$ over this minimal cover, and $P$ is a description of the primitive permutation monodromy group of this minimal cover. If $n = 1$, it will be omitted. Thus, in the data for $\D_6,\D_8,\D_{10}$, the stacks $\cM^\abs$ are all degree 3 over $\cM(1)$ with monodromy group $\fS_3$.

Finally, the field c/nc shows ``cng'' if the subgroup is congruence, and ``ncng'' if the subgroup is noncongruence.

\textbf{Discussion of the cases $G = \D_6,\D_8,\D_{10}$}. The stack $\cM(\D_6)$ is connected and congruence of degree 3 over $\cM(1)$. In fact all index 3 subgroups of $\SL_2(\bZ)$ are congruence (the noncongruence subgroups of lowest index have index 7), and there are not so many: from the signature of $\cM$, one finds that its uniformizing subgroup $\Gamma$ is conjugate to the group $\Gamma_1(2)$ consisting of matrices congruent to $\spmatrix{1}{*}{0}{1}$ mod 2, which also uniformizes $\cM(\bZ/2)$. This is no accident, and comes from the map
$$\cM(\D_6)\ra\cM(\bZ/2)\cong[\cH/\Gamma_1(2)]$$
induced by the abelianization $\D_6\ra\bZ/2$. The map is explicitly given by sending a $\D_6$-cover $\pi : C\ra E$ to the $\bZ/2$-cover $C/\D_6'\ra E$, where $\D_6'$ is the derived subgroup. The fact that $\Gamma_1(2)$ simultaneously uniformizes $\cM(\D_6)$ and $\cM(\bZ/2)$ implies that this map is an \emph{isomorphism}. This implies that given a 2-isogeny of elliptic curves $E'\ra E$, there is a \emph{canonical} way to extend $E'\ra E$ to a $\D_6$-cover of $E$ only branched above the kernel of the isogeny; this holds over any base on which 6 is invertible.

The case of $\D_{10}$ is similar. In this case $\cM(\D_{10})$ has two (isomorphic) components, each uniformized by $\Gamma_1(2)$. The fact that they are isomorphic can be explained by the fact that $\Out(\D_{10})\cong\bZ/2$ acts transitively on them. The two components are in this case separated by the Higman invariant.

The case of $\D_8$ differs from the above because $\D_8^\ab\cong\bZ/2\times\bZ/2$, and hence the abelianization induces a map
$$\cM(\D_8)\ra \cM(2)\cong[\cH/\Gamma(2)]$$
which is an isomorphism, since both have degree 6 over $\cM(1)$. We have the more general result:

\begin{thm} For any integer $k\ge 3$, the stack $\cM(\D_{2k})$ has $\frac{\phi(k)}{2}$ components, each isomorphic to $[\cH/\Gamma_1(2)]$ if $k$ is odd, and $[\cH/\Gamma(2)]$ if $k$ is even. Each component is defined over $\bQ(\zeta_k+\zeta_k^{-1})$.
\end{thm}
\begin{proof} See \cite[Theorem 4.2.2]{Chen18}.
\end{proof}

\subsection{Metabelian level structures are congruence}\label{ss_metabelian}

The descriptions of $\cM(\D_6),\cM(\D_8),\cM(\D_{10})$ are special cases of the theory of $\cM(G)$ for metabelian $G$, which we describe here. In this section we work, by default, over $\bQ$.

Let $G$ be a finite group, and $\ff : \cM(G)\ra\cM(1)$ the forgetful map. Let $E$ be an elliptic curve over $\bQ$, and $x_0\in E^\circ(\bC)$. Recall:
$$\begin{array}{rclrcl}
\Gamma_{E/\bQ} & := & \pi_1(\cM(1),E_\Qbar) & \Pi_\bQ & := & \pi_1(E^\circ,x_0)\\
\Gamma_{E_\Qbar} & := & \pi_1(\cM(1)_\Qbar,E_\Qbar) & \Pi_\Qbar & := & \pi_1(E^\circ_\Qbar,x_0)\\
\Gamma(E^\circ(\bC)) & := & \text{mapping class group of $E^\circ(\bC)$} & \Pi & := & \pi_1(E^\circ(\bC),x_0)
\end{array}$$

The $\bQ$-point of $\cM(1)$ given by $E$ makes $\Gamma_{E/\bQ}$ into a semidirect product $\Gamma_{E/\bQ} \cong \Gamma_{E_\Qbar}\rtimes\Gal_\bQ$, and the monodromy action of $\Gamma_{E/\bQ}$ on fiber $\ff^{-1}(E_\Qbar) = \Epi^\ext(\Pi_\Qbar,G)$ is induced by the outer representation
$$\rho_{E^\circ/\bQ} : \Gamma_{E/\bQ}\cong \Gamma_{E_\Qbar}\rtimes\Gal_\bQ\lra\Out(\Pi_\Qbar)$$
from which we have actions of $\Gamma_{E_\Qbar}\cong\what{\Gamma(E^\circ(\bC))}\cong\what{\Out^+(\Pi)}$ and $\Gal_\bQ$. Let $\Pi_\Qbar^\meta$ be the maximal pro-metabelian quotient, and let $\rho_{E^\circ/\bQ}^\meta$ the associated outer representation on $\Pi_\Qbar^\meta$. If $G$ is metabelian, then
$$\ff^{-1}(E_\Qbar) = \Epi^\ext(\Pi_\Qbar,G) = \Epi^\ext(\Pi_\Qbar^\meta,G)$$
with monodromy actions induced by $\rho_{E^\circ/\bQ}^\meta$. Let $\Pi_\Qbar^\ab$ be the abelianization, which is isomorphic as a $\Gamma_{E/\bQ}$-module to the product of all $\ell$-adic Tate modules $\prod_\ell T_\ell(E)$. A main result is that the image of $\rho_{E^\circ/\bQ}^\meta$ in $\Out(\Pi_\Qbar^\meta)$ is \emph{isomorphic} to its image in $\Out(\Pi_\Qbar^\ab)$. Here is a precise statement.

\begin{thm}[C., Deligne \cite{CD17}]\label{thm_metabelian_congruence} In $\Pi_\Qbar$, there is an inertia subgroup $\cI$, isomorphic to $\Zhat$, which is canonical up to conjugation.\footnote{The inertia subgroup $\cI$ is associated to a $\bQ$-rational tangential base point of $E^\circ$, see Remark \ref{remark_tangential}.} Let $\Out(\Pi_\Qbar^\meta,\cI)$ be the group of outer automorphisms preserving the conjugacy class of $\cI$. Then the representation $\rho_{E^\circ/\bQ}$ maps $\Gamma_{E/\bQ}$ onto $\Out^+(\Pi_\Qbar^\meta,\cI)$, and $\Out(\Pi_\Qbar^\meta,\cI)$ maps isomorphically onto $\GL(\Pi_\Qbar^\ab)$ via the natural map $\Out(\Pi_\Qbar^\meta)\ra\GL(\Pi_\Qbar^\ab)$.
\end{thm}

The theorem can be expressed as the existence of the diagonal isomorphism in the following commutative diagram:
\begin{equation}\label{eq_monodromy_diagram}
\begin{tikzcd}
\Gamma_{E/\bQ}\arrow[r,"\cong"] & \Gamma_{E_\Qbar}\rtimes\Gal_\bQ\arrow[r]\arrow[rr,bend left=20,"\rho_{E^\circ/\bQ}^\meta"] & \Out(\Pi_\Qbar^\meta,\cI)\arrow[r,hookrightarrow]\arrow[rd,"\cong"] & \Out(\Pi_\Qbar^\meta)\arrow[d,twoheadrightarrow] \\
\Gal_\bQ\arrow[ru,hookrightarrow]\ar[rrrd, twoheadrightarrow, bend right=10,"\chi_{\cyc}"'] & \Gamma_{E_\Qbar}\arrow[u,hookrightarrow]\arrow[r,twoheadrightarrow] & \SL(\Pi_\Qbar^\ab)\arrow[r,hookrightarrow] & \GL(\Pi_\Qbar^\ab)\ar[d,twoheadrightarrow,"\det"] \\
 & & & \Zhat^\times
\end{tikzcd}
\end{equation}

Picking bases, the surjection $\Gamma_{E_\Qbar}\twoheadrightarrow\SL(\Pi_\Qbar^\ab)$ in the diagram is identified with the pro-congruence quotient $\what{\Out^+(\Pi)}\cong\what{\SL_2(\bZ)}\twoheadrightarrow\SL_2(\Zhat)$. In particular, it implies that the action of $\Gamma_{E_\Qbar}$ on $\Pi_\Qbar^\meta$, and hence on $\ff^{-1}(E_\Qbar)$, has congruence stabilizers.

\begin{cor} Let $G$ be a metabelian group. Then the components of $\cM(G)_\bC$ are all congruence.	
\end{cor}

Theorem \ref{thm_metabelian_congruence} says nothing about the congruence level of the components of $\cM(G)_\bC$. The abelianization $G\ra G^\ab$ implies that every component of $\cM(G)_\bC$ covers a component of $\cM(G^\ab)_\bC$. Since components of $\cM(G^\ab)_\bC$ have congruence level equal to the exponent $e_{G^\ab}$ of $G^\ab$, every component of $\cM(G)_\bC$ must have level at least $e_{G^\ab}$. On the other hand, if $e$ is the exponent of $G$, let $M_e$ be the free rank 2 metabelian group of exponent $e$. This is a finite group, and choosing a surjection $M_e\ra G$ shows that every component of $\cM(G)_\bC$ is sandwiched between a component of $\cM(M_e)_\bC$ and one of $\cM(G^\ab)_\bC$. We also show:

\begin{thm} Let $G$ be a metabelian group of exponent $e$. The components of $\cM(G)_\Qbar$ are all $\Gal_\bQ$-conjugates of each other, and are congruence modular stacks of level dividing $e$.	
\end{thm}

A central idea in the proofs of these results is to consider the commutator subgroup $G'\le G$ as a $G^\ab$-module, relative to which we can define the cohomology group $H^1(G^\ab,G')$, as well as a certain quotient group $\IOut_1(G)$. We say that $G$ is \emph{rigid} if $\IOut_1(G)$ vanishes\footnote{Precisely, $\IOut_1(G)$ is the quotient of the group of automorphisms which induce the identity on both $G^\ab$ and $G'$ by the subgroup of inner automorphisms.}. A property of rigid groups is that for them, the map $\cM(G)_\bC\ra\cM(G^\ab)_\bC$ induces an isomorphism on connected components, and that components of $\cM(G)_\bC$ are completely classified by the Higman invariant. In general, the components of $\cM(G)$ are Galois over their images in $\cM(G^\ab)$ with Galois group a subgroup of the (abelian group) $\IOut_1(G)$. Examples of rigid groups include $\D_{2k}$ for $k\not\equiv 0\mod 4$, central extensions of abelian groups, metabelian groups $G$ for which $|G'|$ and $|G^\ab|$ are coprime, and the free profinite metabelian group $\Pi_\Qbar^\meta$. Theorem \ref{thm_metabelian_congruence} can be understood as a consequence of the rigidity of $\Pi_\Qbar^\meta$.

\subsection{The groups $\fS_4,\SL_2(\bF_3),\fS_4,\fA_5,\PSL_2(\bF_7)$}\label{ss_noncongruence_examples}
The table below describes all components of $\cM(G)_\bC$ for $G = \SL_2(\bF_3),\fS_4, \fA_5$, and $\PSL_2(\bF_7)$. Note the exceptional isomorphism $\fA_5\cong\PSL_2(\bF_5)$. See \S\ref{ss_dihedral} for a description of the columns.

$$\begin{array}{llllllllllll}
	|G| & G & m & d & c_2 & c_3 & -I & \text{cusp widths} & g & e & \text{AbsMon} & \text{c/nc}\\
	\hline
	24 & \SL_2(\bF_3) & 1 & 32 & 0 & 1 & 0 & 3^24^16^1 & 0 & 4A & {_4}\fA_4 & \text{ncng}\\
	24 & \fS_4 & 1 & 9 & 1 & 0 & 1 & 2^1 3^1 4^1 & 0 & 3A & {_3}\fS_3 & \text{ncng} \\
	60 & \fA_5 & 2 & 10 & 0 & 1 & 1 & 2^13^15^1 & 0 & 5A,5B & \fS_{10} & \text{ncng} \\
	60 & \fA_5 & 1 & 18 & 0 & 0 & 1 & 2^13^25^2 & 0 & 3A & \fA_9 & \text{ncng} \\
	168 & \PSL_2(\bF_7) & 2 & 7 & 1 & 1 & 1 & 3^1 4^1 & 0 & 7A & \fS_7 & \text{ncng} \\
	168 & \PSL_2(\bF_7) & 1 & 32 & 0 & 1 & 0 & 2^1 3^1 4^1 7^1 & 0 & 4A & \fA_{16} & \text{ncng} \\
	168 & \PSL_2(\bF_7) & 1 & 32 & 0 & 1 & 0 & 2^1 3^1 4^1 7^1 & 0 & 4A & \fA_{16} & \text{ncng} \\
	168 & \PSL_2(\bF_7) & 1 & 36 & 0 & 0 & 0 & 1^1 3^2 4^1 7^1 & 0 & 3A & \fS_{18} & \text{ncng} \\
\end{array}$$



The moduli stacks $\cM(\fS_4)$ and $\cM(\SL_2(\bF_3))$ are connected but noncongruence. These are the smallest examples of groups $G$ yielding noncongruence moduli spaces. We note that $\fS_4,\SL_2(\bF_3)$ both have solvable length 3, but there also exist groups $G$ of solvable length 3 for which $\cM(G)$ is congruence. The smallest such example of the semidirect product of the quaternion group $Q_8$ acting on $\bZ/3\times\bZ/3$. This action is given by the two order 4 generators $i,j\in Q_8$ acting by the matrices $\spmatrix{-1}{1}{1}{1}, \spmatrix{0}{1}{-1}{0}$. This group has order $8\cdot 9 = 72$, and has index $[72,41]$ in GAP's SmallGroupsLibrary \cite{GAP4}.

The abelianizations of $\fS_4,\SL_2(\bF_3)$ are $\bZ/2,\bZ/3$ respectively. Thus the moduli spaces are noncongruence covers of $\cM(\bZ/2)\cong[\cH/\Gamma_1(2)],\cM(\bZ/3)\cong[\cH/\Gamma_1(3)]$ respectively. In both cases, these are the maximal congruence modular stacks through which they factor.

The stack $\cM(\fA_5)$ has three components. The group $\Out(\fA_5)$ has order 2, and acts by swapping two components, and leaving one invariant. Label the components $\cM_1,\cM_2,\cM_3$, with $\cM_1\cong\cM_2$, being interchanged by $\Out(\fA_5)$. Since the action of $\Gal_\bQ$ preserves signatures, we find that $\cM_3$ is defined over $\bQ$, and $\cM_1,\cM_2$ are both defined over a quadratic number field. In $\fA_5$, elements of order 5 are 5-cycles, and the unique rational class of 5-cycles splits into two conjugacy classes, which are exactly the Higman invariants associated to $\cM_1$ and $\cM_2$. By the branch cycle lemma \ref{lemma_BCL} the action of $\Gal_\bQ$ factors through the cyclotomic character and swaps these Higman invariants, and hence swaps the components $\cM_1,\cM_2$. Thus they are both defined over the unique quadratic subfield $\bQ(\zeta_5+\zeta_5^{-1})$ of $\bQ(\zeta_5)$.

The stack $\cM(\PSL_2(\bF_7))$ has 5 components. Let $\cM_1\cong\cM_2$ be the two isomorphic components of degree 7 over $\cM(1)$, which are exchanged by the action of $\Out(\PSL_2(\bF_7))\cong\bZ/2$. Let $\cM_3,\cM_4$ be the two components of degree 32; they have the same signatures, but are not isomorphic. Let $\cM_5$ be the remaining component of degree 36; its signature is unique, and hence is defined over $\bQ$. In $\PSL_2(\bF_7)$, there is a unique rational class of elements of order 7, which split into two conjugacy classes, which are the Higman invariants associated to $\cM_1,\cM_2$. Thus we find that $\cM_1\cong\cM_2$ are both defined over $\bQ(\zeta_7+\zeta_7^{-1})$. The same argument does not work for $\cM_3,\cM_4$, whose associated Higman invariants are the same. Nonetheless, we can compute their fields of definition by examining the data for $\cM(\SL_2(\bF_7))$:

$$\begin{array}{llllllllllll}
	|G| & G & m & d & c_2 & c_3 & -I & \text{cusp widths} & g & e & \text{AbsMon} & \text{c/nc}\\
	\hline
	336 & \SL_2(\bF_7) & 2 & 28 & 2 & 1 & 1 & 3^2 6^1 8^2 & 0 & 14A, 14B & {_4}\fS_{7} & \text{ncng} \\
	336 & \SL_2(\bF_7) & 1 & 128 & 0 & 1 & 0 & 3^2 4^2 6^1 7^2 8^2 14^1 & 1 & 8A & {_4}\fA_{16} & \text{ncng} \\
	336 & \SL_2(\bF_7) & 1 & 128 & 0 & 1 & 0 & 3^2 4^2 6^1 7^2 8^2 14^1 & 1 & 8B & {_4}\fA_{16} &\text{ncng} \\
	336 & \SL_2(\bF_7) & 1 & 144 & 0 & 0 & 0 & 3^4 4^1 6^2 7^2 8^2 14^1 & 1 & 3A & {_4}\fS_{18} & \text{ncng}
\end{array}$$

From this we see that the natural map $q : \SL_2(\bF_7)\ra\PSL_2(\bF_7)$ induces a map $\cM(q) : \cM(\SL_2(\bF_7))\ra\cM(\PSL_2(\bF_7))$ which gives a bijection on connected components. Moreover, from the absolute monodromy, we see that the components of $\cM(\PSL_2(\bF_7))^\abs$ are the minimal subcovers of the components of $\cM(\SL_2(\bF_7))^\abs$. The components $\cM_3,\cM_4\subset\cM(\PSL_2(\bF_7))$ lift to components $\tilde{\cM}_3,\tilde{\cM}_4\subset\cM(\SL_2(\bF_7))$, which have distinct Higman invariants, labelled 8A, 8B, which lie in the same rational class. It follows that the components $\tilde{\cM}_3,\tilde{\cM}_4$ are both defined over the field $\bQ(\zeta_8+\zeta_8^{-1})$. Since the map $\cM(q)$ is defined over $\bQ$, it follows that $\cM_3,\cM_4$ are also defined over $\bQ(\zeta_8+\zeta_8^{-1})$.

This analysis made use of the existence of an extension of $\PSL_2(\bF_7)$ in which the unique conjugacy class of order 4 in $\PSL_2(\bF_7)$ splits into distinct classes. This analysis can be made more canonical by observing that $\SL_2(\bF_7)$ is the \emph{Schur cover} of $\PSL_2(\bF_7)$. In general, if $G$ is a finite perfect\footnote{This means $G$ has trivial abelianization.} group, amongst the perfect central extensions of $G$ with kernel $H_2(G,\bZ)$, there is an extension $\tilde{G}$ of maximum order which is unique up to isomorphism of extensions. Such an extension is called a \emph{Schur cover}, and $H_2(G,\bZ)$ is the \emph{Schur multiplier}. Given such a group, and a surjection
$$\varphi : \Pi\ra G$$
where $\Pi$ is a free group with generators $a,b$, let $\tilde{g},\tilde{h}$ be arbitrary lifts of $\varphi(a),\varphi(b)$ to $\tilde{G}$. Then by centrality of the extension, the commutator $[\tilde{g},\tilde{h}]$ is well-defined in terms of $\varphi$, and the $\tilde{G}$-conjugacy class of the commutator is also well defined in terms of the conjugacy class of $\varphi$. Moreover, the class of $[\tilde{g},\tilde{h}]$ is fixed by the action of $\Out^+(\Pi)$ (i.e., the mapping class group, or the fundamental group of $\cM(1)_\bC$), and the action of $\Gal_\bQ$ acts via the cyclotomic character as in \ref{lemma_BCL}. The conjugacy class of $[\tilde{g},\tilde{h}]$ is thus locally constant on $\cM(G)$, and is a refinement of the Higman invariant of an admissible $G$-cover of an elliptic curve. Sadly, this refined invariant, even combined with the signature, is in general not enough to separate components of $\cM(G)$, and hence not enough to describe the Galois action on components\footnote{The smallest nontrivial example is in $\cM(\PSL_3(\bF_3))$, which has exactly two (nonisomorphic) components of degree 360, identical signatures, and identical refined Higman invariants. Note that $\PSL_3(\bF_3)$ has trivial Schur multiplier.}. The Schur multiplier also appears in the result of Conway-Parker described in \S\ref{sss_geometric_monodromy}, see the appendix of \cite{FV91}. This leads one to ask:

\begin{question} Is it possible to describe the set of groups $G$ for which the refined Higman invariant combined with the signature can jointly distinguish all components of $\cM(G)_\bC$?	
\end{question}


\subsection{The inverse Galois problem for the Suzuki group $\Sz(8)$}

The inverse Galois problem asks: Can every finite group be realized as the Galois group of a number field over $\bQ$? Let $H_{0,n}(G)$ be the Hurwitz moduli space of $G$-covers of $\bP^1$ with $n$ branch points. In \cite{FV91}, Fried and Volklein showed how finding a $\Gal_\bQ$-invariant surjection onto $G$ is equivalent to finding a rational point on a Hurwitz spaces $H_{0,n}(G)$ for some $n\ge 3$. In this section we describe an analogous approach using the ``Hurwitz spaces'' $\cM(G)$, as illustrated by the following example for $G = \Sz(8)$.

The table below describes the components of $\cM(\Sz(8))$, where $\Sz(8)$ is the Suzuki group, a finite simple group of order 29120.

$$\begin{array}{llllllllllll}
	|G| & G & m & d & c_2 & c_3 & -I & \text{cusp widths} & g & \text{Hig} & \text{AbsMon} & \text{c/nc}\\
	\hline
29120 &  \Sz(8)  & 3 & 84 & 0 & 0 & 0 &  1^{1}4^{3}5^{3}7^{2}  & 0 & 7A, 7B, 7C & {_2}\fS_{42} & \text{ncng} \\
29120 &  \Sz(8)  & 1 & 192 & 0 & 0 & 1 &  5^{6}7^{12}13^{6}  & 5 & 2A & \fA_{64} & \text{ncng} \\
29120 &  \Sz(8)  & 1 & 192 & 0 & 0 & 1 &  5^{6}7^{12}13^{6}  & 5 & 2A & \fA_{64} & \text{ncng} \\
29120 &  \Sz(8)  & 3 & 234 & 4 & 0 & 1 &  2^{3}5^{9}7^{15}13^{6}  & 3 & 13A, 13B, 13C & \fS_{234} & \text{ncng} \\
29120 &  \Sz(8)  & 3 & 462 & 8 & 0 & 1 &  2^{3}5^{13}7^{28}13^{15}  & 8 & 7A, 7B, 7C & \fS_{462} & \text{ncng} \\
29120 &  \Sz(8)  & 3 & 468 & 0 & 0 & 0 &  1^{3}4^{13}5^{7}7^{15}13^{3}  & 0 & 13A, 13B, 13C & {_2}\fS_{234} & \text{ncng} \\
29120 &  \Sz(8)  & 3 & 588 & 0 & 0 & 0 &  1^{3}4^{13}5^{12}7^{20}13^{3}  & 0 & 7A, 7B, 7C & {_2}\fS_{294} & \text{ncng} \\
29120 &  \Sz(8)  & 3 & 624 & 0 & 0 & 0 &  2^{4}4^{14}5^{13}7^{15}13^{6}  & 1 & 13A, 13B, 13C & {_2}\fA_{312} & \text{ncng}\\
29120 &  \Sz(8)  & 3 & 624 & 0 & 0 & 0 &  2^{4}4^{14}5^{13}7^{15}13^{6}  & 1 & 13A, 13B, 13C & {_2}\fA_{312} & \text{ncng}\\
29120 &  \Sz(8)  & 3 & 624 & 0 & 0 & 0 &  2^{4}4^{14}5^{13}7^{15}13^{6}  & 1 & 13A, 13B, 13C & {_2}\fA_{312} & \text{ncng}\\
29120 &  \Sz(8)  & 1 & 660 & 0 & 0 & 0 &  1^{3}4^{9}5^{21}7^{21}13^{3}  & 0 & 5A & {_2}\fS_{110} & \text{ncng} \\
29120 &  \Sz(8)  & 1 & 690 & 12 & 0 & 1 &  2^{3}5^{12}7^{39}13^{27} & 15 & 5A & \fS_{230} & \text{ncng} \\
29120 &  \Sz(8)  & 3 & 1008 & 0 & 0 & 0 &  2^{4}4^{16}5^{21}7^{30}13^{9}  & 3 & 7A, 7B, 7C & {_2}\fA_{504} & \text{ncng} \\
29120 &  \Sz(8)  & 3 & 1008 & 0 & 0 & 0 &  2^{4}4^{16}5^{21}7^{30}13^{9}  & 3 & 7A, 7B, 7C & {_2}\fA_{504} &\text{ncng} \\
29120 &  \Sz(8)  & 3 & 1008 & 0 & 0 & 0 &  2^{4}4^{16}5^{21}7^{30}13^{9}  & 3 & 7A, 7B, 7C & {_2}\fA_{504} &\text{ncng} \\
29120 &  \Sz(8)  & 3 & 1200 & 0 & 0 & 0 &  2^{4}4^{10}5^{24}7^{45}13^{9}  & 5 & 5A & {_2}\fA_{600} & \text{ncng} \\
29120 &  \Sz(8)  & 1 & 1536 & 0 & 0 & 0 &  4^{24}5^{36}7^{48}13^{12}  & 5 & 4A & {_2}\fA_{256} & \text{ncng} \\
29120 &  \Sz(8)  & 1 & 1536 & 0 & 0 & 0 &  4^{24}5^{36}7^{48}13^{12}  & 5 & 4B & {_2}\fA_{256} & \text{ncng}
\end{array}$$
The stack $\cM(\Sz(8))$ has a total of 42 components, all noncongruence. The ``clumps'' of multiplicity 3 are all acted on transitively by the action of $\Out(\Sz(8))\cong\bZ/3$. Note the severe failure of the Higman invariant to separate components. Nonetheless, for every clump except those of degree 1008 or 1200, we can still use the $\Gal_\bQ$-invariance of the signature to compute their fields of definition: because all conjugacy classes in $\Sz(8)$ of a given order lie in a single rational class, $\Gal_\bQ$ will act transitively on each such clump, and hence the field of definition will be the unique index 3 subfield of the cyclotomic field $\bQ(\zeta_e)$, where $e$ is the ramification index (the order of the Higman invariant).

Let $\cM$ be the unique component of degree 660. Since $\Gal_\bQ$ preserves signatures, $\cM$ is defined over $\bQ$. Since the fields $c_2,c_3,-I$ all vanish, we find that $\cM = M$ is a (fine moduli) scheme of genus $g = 0$. It has exactly three cusps of width 1, which means that either one is a $\bQ$-rational point, or they correspond to a point defined over a cubic extension, in which case a Riemann-Roch calculation implies that $M$ has a $\bQ$-rational point. In either case, $M\cong\bP^1_\bQ - \{\text{cusps}\}$. The universal elliptic curve $\cE\ra\cM$ is a scheme $E/M$, and is equipped with an admissible $\Sz(8)$-cover $\pi : C\ra E$. If $t\in M(\bQ)$ is a $\bQ$-rational point, then by specializing, we obtain an $\Sz(8)$-cover $\pi_t : C_t\ra E_t$. For any nonidentity $P\in E_t(\bQ)$, the fiber $\pi_t^{-1}(P)$ is thus an \'{e}tale $\bQ$-algebra with a geometrically free and transitive action of $\Sz(8)$; if it is irreducible, then it would be the spectrum of a number field $K$ with $\Gal(K/\bQ)\cong\Sz(8)$, thus realizing $\Sz(8)$ as a Galois group over $\bQ$.\footnote{According to an article on David Zywina's website, $\Sz(8)$ has not yet been realized as a Galois group over $\bQ$.} Using the fact that $\Sz(8)$ is a perfect group, one can show

\begin{prop}\label{prop_irreducible} For any $t\in M(\bQ)$, for all but finitely many $P\in E_t(\bQ)$, $\pi^{-1}(P)$ is irreducible.	
\end{prop}
\begin{proof} See the discussion in \cite[\S4.3]{Chen18}.
\end{proof}

Thus, to realize $\Sz(8)$ as a Galois group over $\bQ$, it suffices to find a single $\bQ$-fiber of $E$ with positive rank. We note that this is implied by some standard conjectures in number theory. This results in the following ridiculous theorem:

\begin{thm} If the parity conjecture, the ABC conjecture, and the Chowla conjecture are true\footnote{Specifically, we will need hypotheses $\cA_2$ and $\cB_2$ of \cite[\S2.2]{Hel09}. These are implied by the ABC and Chowla conjectures respectively, see discussion in \cite[\S2.2]{Hel09}.}, then $\Sz(8)$ is a Galois group over $\bQ$.
\end{thm}
\begin{proof} By the parity conjecture, the parity of the rank of an elliptic curve over a number field $K$ is equal to its root number. Thus if the root number $W(E_t)$ of $E_t$ is nonconstant as $t$ varies over $M(\bQ)$, then $E_t$ would have positive rank for those values of $t$ for which $W(E_t) = -1$, which by the above discussion would imply the result. In \cite[Main Theorem 2]{Hel09}, Helfgott shows (conditional on ABC and Chowla) that if the generic fiber $E_{\bQ(t)}$ has a single place of multiplicative reduction, then the average of the root numbers $W(E_t)$ as $t\in M(\bQ)$ is equal to 0.	 Because $\pi : C\ra E$ extends to an admissible $\Sz(8)$-cover over $\cAdm(\Sz(8))$ (whose cusps classify $\Sz(8)$-covers of \emph{stable curves}), the condition that $E_{\bQ(t)}$ has a place of multiplicative reduction would be satisfied if at least one of the cusps of $\cM\subset\cAdm(\Sz(8))$ is not a stacky point (i.e., has trivial automorphism group). This can be verified group-theoretically, using \cite[Theorem 4.10.3]{Chen21}.
\end{proof}

\begin{remark} Another approach to finding a $\bQ$-rational fiber of positive rank is by computing equations for the universal elliptic curve. This comes with its own challenges. Even the first step, computing the Belyi map $M\ra M(1)$ is nontrivial, see \cite{SV14}.
\end{remark}

\begin{remark} The approach of Fried and Volklein \cite{FV91} to the inverse Galois problem encounters many similar issues. For them, finding rational points on Hurwitz stacks was essential, a problem which happened to be easy for us in the case of $\cM\subset\cM(\Sz(8))$. After that, given a rational point, corresponding to a $G$-cover $X\ra\bP^1$, the Hilbert irreducibility theorem implies the existence of (infinitely many) irreducible fibers, each giving a realization of $G$ as a Galois group over $\bQ$. Proposition 
	\ref{prop_irreducible} is an analog of the Hilbert irreducibility theorem in the case of $G$-covers of elliptic curves only branched above one point, where $G$ is a perfect group.
\end{remark}

\subsection{Comparing noncongruence to nonabelian}\label{ss_noncongruence_nonabelian}

By the results of Asada and Ellenberg-McReynolds, every noncongruence modular curve appears as a quotient of a component of $\cM(G)_\bC$ for some $G$. In this section we discuss the question of understanding when components of $\cM(G)_\bC$ are noncongruence.

\begin{defn} Let $\Pi$ be a free group of rank 2, and let $G$ be a finite group.

\begin{itemize}
	\item We say that $G$ (or $\cM(G)_\bC$) is \emph{congruence} if all components of $\cM(G)_\bC$ are congruence. Combinatorially, this means that all $\Out^+(\Pi)\cong\SL_2(\bZ)$-stabilizers of $\Epi^\ext(\Pi,G)$ are congruence (see \S\ref{sss_combinatorial_decomposition}). Otherwise we say that $G$ (or $\cM(G)_\bC$) is noncongruence.
	\item We say that $G$ (or $\cM(G)_\bC$) is \emph{purely noncongruence} if all components of $\cM(G)_\bC$ are noncongruence, or equivalently, that all $\Out^+(\Pi)\cong\SL_2(\bZ)$-stabilizers of $\Epi^\ext(\Pi,G)$ are noncongruence.
	\item We say that a subgroup $\Gamma\le\SL_2(\bZ)$ is \emph{totally noncongruence} if it is noncongruence and its congruence closure\footnote{This is the intersection of all congruence subgroups containing $\Gamma$.} $\Gamma^c$ is the entirety of $\SL_2(\bZ)$. We say that a component $\cM\subset\cM(G)_\bC$ is \emph{totally noncongruence} if $\cM$ is uniformized by a totally noncongruence subgroup of $\SL_2(\bZ)$ (equivalently, $\cM\ra\cM(1)_\bC$ has no nontrivial intermediate congruence covers).
\end{itemize}
\end{defn}

Below we describe two criteria for components of $\cM(G)_\bC$ to be noncongruence. In \S\ref{sss_philosophy} we will describe a general philosophy and some questions.

\subsubsection{A criterion in terms of $G$}\label{sss_G_criterion}
\begin{defn}[Wohlfahrt level] For finite index $\Gamma\le\SL_2(\bZ)$, the \emph{level} of $\Gamma$ is the least common multiple of its cusp widths, or equivalently the least common multiple of the ramification indices of the map $\cH/\Gamma\ra\cH/\SL_2(\bZ)$ above the cusp $i\infty$ of $\cH/\SL_2(\bZ)$. 
\end{defn}
Since this is defined in terms of modular curves, this is really a property of $\pm\Gamma := \langle \Gamma,-I\rangle$. Recall that if $\Gamma$ is congruence, then its congruence level is the minimum positive integer $n\ge 1$ such that $\Gamma\supset\Gamma(n)$. If $-I\in\Gamma$, then these two notions of level agree for congruence $\Gamma$. In general, we have

\begin{thm}[Wohlfahrt, Kiming-Sch\"{u}tt-Verrill]\label{thm_WKSV} Let $\Gamma\le\SL_2(\bZ)$ be a finite index subgroup of level $l$. Then $\Gamma$ is congruence if and only if it contains $\Gamma(2l)$, in which case its congruence level is either $l$ or $2l$. If moreover $-I\in\Gamma$, then $\Gamma$ is congruence if and only if $\Gamma\supset\Gamma(l)$, and in this case $\Gamma$ has congruence level $l$.	
\end{thm}
\begin{proof} For example, the ``Sanov'' subgroup $\Gamma := \langle\spmatrix{1}{2}{0}{1},\spmatrix{1}{0}{2}{1}\rangle$ is index 2 inside $\Gamma(2) = \langle\Gamma,-I\rangle$. It has level 2 but congruence level 4. See \cite{Woh63} and \cite{KSV11} for a proof of this theorem.
\end{proof}

\begin{cor}\label{cor_congruence_closure} The congruence closure $\Gamma^c$ of $\Gamma$ is generated by $\Gamma$ and $\Gamma(2l)$.
\end{cor}
\begin{proof} Let $\Delta$ be any congruence subgroup containing $\Gamma$ of level $d$. Then $d\mid l$, and hence $\Delta\supset\Gamma(2d)\supset\Gamma(2l)$. This shows that $\Gamma^c\supset\langle\Gamma,\Gamma(2l)\rangle$, and hence $\Gamma^c = \langle\Gamma,\Gamma(2l)\rangle$.	
\end{proof}

For a subgroup $\Gamma\subset\SL_2(\bZ)$, we can compare $\Gamma$ to the principal congruence subgroup $\Gamma(n)$ via the commutative diagram with exact rows (c.f. \cite{Sch12}):
\begin{equation}\label{eq_congruence_deficiency}
\begin{tikzcd}
	1\ar[r] & \Gamma(n)\ar[r]\ar[d,dash,"f_n"] & \SL_2(\bZ)\ar[r,"\pr_n"]\ar[d,dash,"d"] & \SL_2(\bZ/n)\ar[r]\ar[d,dash,"e_n"] & 1 \\
	1\ar[r] & \Gamma(n)\cap\Gamma\ar[r] & \Gamma\ar[r] & \pr_n(\Gamma)\ar[r] & 1
\end{tikzcd}
\end{equation}
Here the vertical arrows are subgroup inclusions, labeled by their respective indices. We always have $d = e_nf_n$, and Corollary \ref{cor_congruence_closure} implies that for a subgroup $\Gamma$ of level $l$ and congruence closure $\Gamma^c$, $f_{2l} = [\Gamma:\Gamma^c]$ and $e_{2l} = [\SL_2(\bZ):\Gamma^c]$. We will refer to $f_{2l}$ as the congruence deficiency, and $e_{2l}$ as the congruence degree. Thus, $\Gamma$ is congruence if and only if it has congruence deficiency 1 and congruence degree $d$, and totally noncongruence if and only if it has congruence degree 1 and congruence deficiency $d > 1$.\footnote{In \cite{Sch12}, Schmithusen called $f_l$ the congruence deficiency of $\Gamma$, and used the ideas above to prove that certain subgroups coming from square-tiled surfaces are noncongruence. In light of Theorem \ref{thm_WKSV}, we feel that $f_{2l}$ better deserves that name.} These ideas lead to the following noncongruence criterion \cite[Theorem 4.4.10]{Chen18}.

\begin{thm}\label{thm_noncongruence_criterion_A} Let $\Pi$ be a free group of rank 2, with generators $a,b$. Let $G$ be a nontrivial finite group generated by $x,y,z$ satisfying $xyz = 1$. Suppose the orders $|x|,|y|,|z|$ satisfy the following property:
\begin{itemize}
\item[$(*)$] The integers $|x|,|y|,|z|$ are pairwise coprime.
\end{itemize}
Let $\SL_2(\bZ)$ act on $\Epi^\ext(\Pi,G)$ via the isomorphism $\Out^+(\Pi)\cong\SL_2(\bZ)$ given by abelianization. Then the $\SL_2(\bZ)$-stabilizer of the surjection $\varphi_{x,y} : \Pi\ra G$ sending $a,b\mapsto x,y$ is \emph{totally noncongruence}. In particular, $G$ is noncongruence.
\end{thm}

\begin{remark} As will be evident from the proof, the theorem has some room for flexibility. If the pairwise gcd's in $(*)$ are not too large compared to the index of $\Gamma_{\varphi_{x,y}}$, then we can still show that $\Gamma_{\varphi_{x,y}}$ is noncongruence, though maybe not totally noncongruence.
\end{remark}

\begin{remark} The condition $(*)$ is equivalent to the statement that the gcd $\delta$ of the set $\{|x||y|,|x||z|,|y||z|\}$ is 1. If $G$ is abelian, then all three pairwise least common multiples of $|x|,|y|,|z|$ are equal to the exponent $e(G)$, and hence the gcd $\delta$ would be divisible by $e(G)$ in this case. Thus $\delta$ can be viewed as a measure of nonabelianness of $G$ relative to the surjection $\varphi_{x,y}$, and the theorem links this nonabelianness to the noncongruenceness of the stabilizer $\Gamma_{\varphi_{x,y}}$.	
\end{remark}

\begin{proof} Define $\varphi_{z,x}$ and $\varphi_{y,z}$ analogously to $\varphi_{x,y}$. We note that $\varphi_{x,y},\varphi_{z,x},\varphi_{y,z}$ lie in the same $\SL_2(\bZ)\cong\Out^+(\Pi)$-orbit, and hence their stabilizers $\Gamma_{\varphi_{x,y}},\Gamma_{\varphi_{z,x}},\Gamma_{\varphi_{y,z}}$ are conjugate. Further observe that
\begin{equation}\label{eq_parabolic}
\spmatrix{1}{|x|}{0}{1},\spmatrix{1}{0}{|y|}{1}\in\Gamma_{\varphi_{x,y}},\quad\spmatrix{1}{|z|}{0}{1},\spmatrix{1}{0}{|x|}{1}\in\Gamma_{\varphi_{z,x}},\quad\text{and}\quad \spmatrix{1}{|y|}{0}{1},\spmatrix{1}{0}{|z|}{1}\in\Gamma_{\varphi_{y,z}}
\end{equation}

For ease of notation, let $\Gamma_1,\Gamma_2,\Gamma_3$ be these three subgroups. Let $l$ denote their common level.

Since $G$ is nontrivial, condition $(*)$ implies that $|x|,|y|,|z|$ are not all identical, and hence the $\Gamma_i$'s are proper subgroups of $\SL_2(\bZ)$. Thus, it suffices to show that the congruence degree $e_{2l}$ of \eqref{eq_congruence_deficiency} is equal to 1 for any (equivalently, all) of $\Gamma_1,\Gamma_2,\Gamma_3$; equivalently, we want to show that $\Gamma_i$ surjects onto $\SL_2(\bZ/2l)$ for some (equivalently all) $i$. Let $2l = \prod p_i^{r_i}$ be the prime factorization, with corresponding direct \emph{sum} decomposition
$$\SL_2(\bZ/l)\cong \bigoplus_{i}\SL_2(\bZ/p_i^{r_i})$$
Since each summand $\SL_2(\bZ/p_i^{r_i})$ is generated by the images of $\spmatrix{1}{1}{0}{1},\spmatrix{1}{0}{1}{1}$ mod $p_i^{r_i}$, \eqref{eq_parabolic} and condition $(*)$ implies that each summand is contained in the image of at least one of $\Gamma_1,\Gamma_2,\Gamma_3$. On the other hand, since the summands are normal and the $\Gamma_i$'s are all conjugate, the image of any $\Gamma_i$ contains every summand, and hence surjects onto $\SL_2(\bZ/2l)$, as desired.
\end{proof}

\begin{cor} Let $G$ be any of the following groups.
\begin{itemize}
\item $\fA_n$ for $n\ge 5$,
\item $\PSL_2(\bF_p)$ for $p\ge 5$ prime,
\item $\PSL_2(\bF_{2^p})$ for $p$ prime,
\item $\PSL_2(\bF_{3^p})$ for $p$ any odd prime,
\item $\Sz(2^p)$ for $p$ any odd prime,
\item $\PSL_3(\bF_3)$.
\end{itemize}
Then $G$ satisfies the hypothesis of Theorem \ref{thm_noncongruence_criterion_A}. In particular, $\cM(G)_\bC$ has a totally noncongruence component, and any finite 2-generated extension of $G$ is noncongruence.
\end{cor}
\begin{proof} The cases of $\fA_n$ and $\PSL_2(\bF_p)$ can be done by explicitly choosing generating pairs, see \cite[\S4.4]{Chen18}. The remaining cases follows from Thompson's results on minimal finite simple groups, see \cite[\S3]{Thom68}.
\end{proof}

\subsubsection{A monodromic criterion}
Noncongruence modular stacks can also be detected from their monodromy. For finite index $\Gamma\le\SL_2(\bZ)$, let $C(\Gamma) := \cap_{\gamma\in\SL_2(\bZ)}\gamma\Gamma\gamma^{-1}$ be the \emph{normal core} -- the largest normal subgroup of $\Gamma$. The monodromy group of $[\cH/\Gamma]\ra[\cH/\SL_2(\bZ)]$ is the quotient
$$\Mon(\Gamma) := \SL_2(\bZ)/C(\Gamma)$$
If $\Gamma$ is congruence, it contains $\Gamma(n)$ for some $n\ge 1$, and as a result the monodromy group is a quotient of $\SL_2(\bZ)/\Gamma(n)\cong\SL_2(\bZ/n)\cong\prod_{p^r\mid\mid n}\SL_2(\bZ/p^r)$, where we say $p^r\mid\mid n$ if $p^r\mid n$ and $p^{r+1}\nmid n$. Since $\SL_2(\bZ/p^r)$ is an extension of $\SL_2(\bZ/p)$ by a $p$-group, it follows that the simple composition factors of $\Mon(\Gamma)$ are either abelian or of the type $\PSL_2(\bZ/p)$. This can be made more a little more precise in terms of the level:

\begin{thm}\label{thm_monodromic_criterion} Let $\Gamma\le\SL_2(\bZ)$ be a finite index subgroup of (Wohlfahrt) level $l$. If $\Gamma$ is congruence, then the simple composition factors of $\Mon(\Gamma)$ are isomorphic to either
\begin{itemize}
\item $\bZ/p$ for some prime $p\mid 6l$, or
\item $\PSL_2(\bF_p)$ for some prime $p\ge 5$ with $p\mid l$.
\end{itemize}
\end{thm}
\begin{proof} If $\Gamma$ is congruence, then by Theorem \ref{thm_WKSV}, $\Gamma\supset\Gamma(2l)$. The statement follows from the above discussion.	
\end{proof}

\subsubsection{A general philosophy}\label{sss_philosophy}


In Proposition \ref{prop_abelian_congruence}, we saw that abelian groups are congruence. This was generalized to metabelian groups in \S\ref{ss_metabelian}. Guided by these results, Theorem \ref{thm_noncongruence_criterion_A}, and computed data, in \cite{Chen18}, the author adopted the philosophy that highly nonabelian groups should tend to be highly noncongruence.

At the time of writing of this article, the components of $\cM(G)$ and $\cM(G)^\abs := \cM(G)/\Out(G)$ have been computed for all finite simple groups of order up to $|\J_1| = 175560$, where $\J_1$ denotes the Janko group.\footnote{The data for noncongruence $G$ of order $\le 29120$ can be found in \cite[Appendix B]{Chen18}.} All such components are close to being totally noncongruence, in the sense that their uniformizing subgroups have congruence closures of index $\le 3$ inside $\SL_2(\bZ)$ (c.f.  \S\ref{sss_unfull}). Other than 4 components of $\cM(\PSU_3(\bF_4))$ which are congruence, of degree 1 over $\cM(1)$, and 2 components of $\cM(\PSU_3(\bF_5))$ which are congruence of degree 3 over $\cM(1)$, the rest are all noncongruence. Given this, it is natural to propose the following refinement of the philosophy:
\begin{equation}\label{eq_philosophy}
\parbox{\linewidth-5em}{The components of $\cM(G)$ for highly nonabelian groups $G$ should be uniformized by subgroups of low congruence degree.}
\end{equation}

It is unclear how to best formulate a precise, sharp statement of the philosophy. One attempt might proceed as follows. For a modular stack $\cM$, let $e_\cM$ be the congruence degree of a uniformizing subgroup for $\cM$.

\begin{question} Does there exist a universal upper bound on $e_\cM$ as $\cM$ ranges over all components of $\cM(G)$ for nonabelian finite simple groups $G$?
\end{question}

A positive answer would be a step towards making \eqref{eq_philosophy} precise, though it still doesn't explain the empirical rarity of congruence components for simple groups $G$. Towards this, the only conjecture the author is prepared to make at the moment is:

\begin{conj} Nonabelian finite simple groups are noncongruence.	
\end{conj}

Since subgroups of noncongruence subgroups $\Gamma\le\SL_2(\bZ)$ are themselves noncongruence, the property of a finite group $G$ being noncongruence is stable under group extensions. Thus this conjecture would imply that any finite 2-generated extension of a nonabelian finite simple group is noncongruence. For solvable groups, it has been checked that all 2-generated groups of solvable length $\ge 4$ and order $\le 255$ are noncongruence\footnote{There are 60 such groups, see \cite[Appendix B1]{Chen18}.}. It is an interesting question to ask if a lower bound on solvable degree can guarantee noncongruence.

\begin{remark} In \cite[Conjecture 4.4.1]{Chen18}, the author conjectured that nonabelian finite simple groups were \emph{purely} noncongruence. The smallest counterexample $G = \PSU_3(\bF_4)$ was found in 2021, see \S\ref{ss_burau}.
\end{remark}

\subsection{Congruence components of $\cM(\PSU_3(\bF_q))$ and the Burau representation}\label{ss_burau}

Recall that over a finite field $\bF_q$, up to isometry there is a unique nondegenerate hermitian form on $\bF_{q^2}^n$ relative to the involution $a\mapsto a^q$ \cite[Corollary 10.4]{Grove02}. Let $\U_3(\bF_q)\le\GL_3(\bF_{q^2})$ be the group of isometries of any such form. Let $\SU_3(\bF_q)\le\U_3(\bF_q)$ be the subgroup of elements of determinant 1, and let $\PSU_3(\bF_q)$ be the quotient of $\SU_3(\bF_q)$ by the center. Below we present the 33 components of $\cM(\PSU_3(\bF_4))_\bC$:

$$\begin{array}{llllllllllll}
	|G| & G & m & d & c_2 & c_3 & -I & \text{cusp widths} & g & \text{Hig} & \text{AbsMon} & \text{c/nc}\\
	\hline
62400 & \PSU_3(\bF_4) & 4 & 1    & 1 & 1 & 1 & 1^1                                                     & 0 & 5A,B,C,D & \fS_1 & \text{cng} \\
62400 & \PSU_3(\bF_4) & 4 & 9    & 1 & 0 & 1 & 1^1 3^1 5^1                                             & 0 & 15A,B,C,D & \fA_9 & \text{ncng}\\
62400 & \PSU_3(\bF_4) & 4 & 13   & 1 & 1 & 1 & 3^1 5^2                                                 & 0 & 13A,B,C,D & \fA_{13} & \text{ncng}\\
62400 & \PSU_3(\bF_4) & 4 & 184  & 0 & 8 & 0 & 1^2 3^2 4^2 10^2 13^2 15^2                              & 0 & 5A,B,C,D  & {_2}\fA_{96} & \text{ncng}\\
62400 & \PSU_3(\bF_4) & 1 & 2048 & 0 & 4 & 0 & 3^{20} 4^{16} 5^{32} 10^{24} 13^{20} 15^{16}            & 21& 4A & \fA_{512} & \text{ncng}\\
62400 & \PSU_3(\bF_4) & 1 & 2048 & 0 & 4 & 0 & 3^{20} 4^{16} 5^{32} 10^{24} 13^{20} 15^{16}            & 21& 4A & \fA_{512} &\text{ncng}\\
62400 & \PSU_3(\bF_4) & 4 & 2880 & 0 & 6 & 0 & 3^{30} 4^{20} 5^{26} 10^{26} 13^{40} 15^{24}            & 36& 10A,B,C,D & {_2}\fA_{1440} & \text{ncng}\\
62400 & \PSU_3(\bF_4) & 2 & 3280 & 0 & 5 & 0 & 1^2 2^{14} 3^{29} 4^{24} 5^{53} 10^{36} 13^{34} 15^{24} & 28& 5E,5F & \fA_{1640} & \text{ncng}\\
62400 & \PSU_3(\bF_4) & 4 & 3816 & 0 & 0 & 0 & 1^2 3^{42} 4^{30} 5^{28} 10^{36} 13^{50} 15^{34}        & 49& 15A,B,C,D & {_2}\fA_{1908} & \text{ncng}\\
62400 & \PSU_3(\bF_4) & 4 & 4368 & 0 & 9 & 0 & 3^{51} 4^{38} 5^{24} 10^{44} 13^{53} 15^{42}            & 54& 13A,B,C,D & {_2}\fA_{4368} & \text{ncng}\\
62400 & \PSU_3(\bF_4) & 1 & 5328 & 0 & 0 & 0 & 1^2 2^{12} 3^{56} 4^{40} 5^{72} 10^{58} 13^{50} 15^{48} & 54& 3A & \fS_{1332} & \text{ncng}
\end{array}$$

Here, for Higman invariants, for the sake of compactness we have writen ``5A,B,C,D'' to mean ``5A,5B,5C,5D''. One can check that each group of multiplicity $m>1$ is acted upon transitively by $\Out(\PSU_3(\bF_4))\cong\bZ/4$. From the data, we can deduce the fields of definition of every component other than the two of degree 2048; these two can only be said to be defined over a quadratic number field.\footnote{A further computation shows that these two components are mirror images in the sense of \S\ref{sss_mirror}; By Proposition \ref{prop_mirror}, this implies that they are complex conjugates of each other, though this says very little about their fields of definition.}

Let $\Pi$ be the fundamental group of a punctured elliptic curve. The component of degree 1 corresponds to a surjection $\varphi\in\Epi^\ext(\Pi,G)$ which is fixed by $\Out^+(\Pi)\cong\SL_2(\bZ)$. In fact, one can check that $\varphi$ is even $\Out(\Pi)$-invariant, and hence its kernel is a characteristic subgroup of $\Pi$ -- we say that $\varphi$ is \emph{characteristic}. In a forthcoming work joint with Alexander Lubotzky and Pham Huu Tiep, we will explain how this surjection can be constructed from the Burau representation of the braid group $B_4$. We sketch the main points below.

Recall that the Braid group $B_n$ is generated by $\sigma_1,\sigma_2,\ldots,\sigma_{n-1}$ subject to the relations
\begin{itemize}
\item $[\sigma_i,\sigma_j] = 1$	if $|i-j| \ge 2$, and
\item $\sigma_i\sigma_{i+1}\sigma_i = \sigma_{i+1}\sigma_i\sigma_{i+1}$ if $1\le i\le n-2$.
\end{itemize}
It's center $Z(B_n)$ is infinite cyclic, generated by $(\sigma_1\sigma_2\cdots\sigma_{n-1})^n$. For $n = 4$, $B_4$ contains a \emph{normal} free subgroup $F := \langle \sigma_1\sigma_3^{-1}, \sigma_2\sigma_1\sigma_3^{-1}\sigma_2^{-1}\rangle$ on which the conjugation action of $B_4$ induces an isomorphism
$$B_4/Z\rightiso\Aut^+(F)$$
This implies that if $\rho : B_4\lra G$ is any homomorphism, then the class of $\rho|_F : F\ra \rho(F)$ modulo $\Aut(\rho(F))$ is fixed by $\Aut^+(F)$: we say that $\rho|_F$ is $\Aut^+(F)$-invariant. In other words, $\rho|_F$ is almost characteristic. We apply this to the Burau representation of $B_4$, which is a representation
$$\rho_\Burau : B_4\lra\GL_3(\bZ[q,q^{-1}])$$
The image of $\rho_\Burau$ is contained in the special unitary group relative to a certain Hermitian form on $\bZ[q,q^{-1}]^3$ \cite{Sto10, Ven14}. In \cite{CLT23}, we show that

\begin{thm}[{C., Lubotzky, Tiep \cite{CLT23}}] Let $q = p^d$ be a prime power with $q\ge 7$. Then there exist specializations of $\rho_\Burau$ which realize $\SU_3(\bF_q)$ and $\SL_3(\bF_q)$ as characteristic quotients of $F$. Moreover, $\SU_3(\bF_4)$ and $\SU_3(\bF_5)$ can also be attained in this way. The remaining cases $\SL_3(\bF_5),\SL_3(\bF_4),\SL_3(\bF_3),\SL_3(\bF_2)$, and $\SU_3(\bF_3),\SU_3(\bF_2)$ are not characteristic quotients of $F$.
\end{thm}

Let $G$ be either $\PSL_3(\bF_p)$ or $\PSU_3(\bF_p)$ where $p$ is prime. Then $\Out(G)$ has order at most 6, and hence the component of $\cM(G)$ corresponding to a characteristic quotient $F\ra G$ have degree at most 6 over $\cM(1)$. Since all subgroups of index $\le 6$ are congruence, it follows that for $p\ge 7$, $\cM(G)$ has a congruence component. The theorem also implies that $\cM(G)^\abs$ has a component of degree 1 if $G = \PSL_3(\bF_q)$ or $\PSU_3(\bF_q)$, where $q\ge 7$ is a prime power.

\subsection{$\SL_2(\bF_p)$, connected components, and Markoff triples}\label{ss_markoff}
When $G$ is a finite simple group of Lie type, the set $\Epi^\ext(\Pi,G)$ can be given an algebraic structure via the theory of character varieties. In this section we explain how this algebraic structure can be leveraged in the situation $G = \SL_2(\bF_p)$ to better understand the $\Out^+(\Pi)$-orbits on $\Epi^\ext(\Pi,\SL_2(\bF_p))$, and correspondingly the components of $\cM(\SL_2(\bF_p))$. We will also explain how it relates to the Diophantine properties of the Markoff equation.

Let $\cG$ denote an affine algebraic group over a ring $R$. The \emph{representation variety} for $\cG$-representations of $\Pi$ is the scheme $\Hom(\Pi,\cG)$ whose set of $S$-points is exactly the set of representations $\Hom(\Pi,\cG(S))$ for any scheme $S$. Since $\Pi$ is free of rank 2, $\Hom(\Pi,\cG)$ is isomorphic to the product $\cG\times\cG$. The action of $\cG$ by inner automorphisms on $\Hom(\Pi,\cG)$ corresponds to the diagonal action on $\cG\times\cG$, and the GIT quotient
$$X_{\cG}(\Pi) := \Hom(\Pi,\cG)\git\cG$$
is called the \emph{character variety} for $\cG$-representations of $\Pi$. For a modern treatment of character varieties over $\bC$, see \cite{LS17,Sik12} and the references therein.

Suppose now that $\cG = \SL_{2,R}$ over a ring $R$. Let $A[\Pi]_R$ denote the affine ring of $\Hom(\Pi,\SL_{2,R})$, then the character variety $X_{\SL_{2,R}} := \Hom(\Pi,\SL_{2,R})\git\SL_{2,R}$ is the spectrum of the ring of invariants $A[\Pi]_R^{\SL_{2,R}}$. Let $a,b$ be generators of $\Pi$. By classical invariant theory over $\bC$, it is known that $A[\Pi]_\bC^{\SL_{2,\bC}}$ is the polynomial ring generated by the trace functions
$$x := \tr\varphi(a),\quad y := \tr\varphi(b),\quad z := \tr\varphi(AB)$$
This was further extended to the case of general rings $R$ in \cite{BH95}. More precisely, we have
\begin{thm} Let $R$ be any ring. Let $\Tr : \Hom(\Pi,\SL_{2,R})\ra\bA^3_R$ be the map defined on $A$-valued points for various $R$-algebras $A$ by sending $\varphi : \Pi\ra\SL_2(A)$ to $\tr\varphi(a),\tr\varphi(b),\tr\varphi(c)$. Then $\Tr$ induces an isomorphism, also denoted $\Tr$:
$$\Tr : X_{\SL_{2,R}}\rightiso\bA^3_R = \Spec R[x,y,z]$$
\end{thm}
\begin{proof} See \cite[Theorem 5.2.1]{Chen21}.	
\end{proof}
In particular, formation of the character variety commutes with base change, and hence it suffices to work with $X_{\SL_2} := X_{\SL_{2,\bZ}}\cong\bA^3_\bZ$. The natural \emph{right} action of $\Aut(\Pi)$ on the representation variety induces a right action of $\Out(\Pi)$ on the character variety $X_{\SL_2}\cong\bA^3_\bZ$, where it is given by remarkably simple polynomials:
\begin{equation}\label{eq_explicit_action}\begin{array}{rcl}
r : (a,b) & \mapsto & (a^{-1},b) \\
s : (a,b) & \mapsto & (b,a) \\
t : (a,b) & \mapsto & (a^{-1},ab) \\
\end{array}\quad\stackrel{\Tr_*}{\longrightarrow}\quad
\begin{array}{rcl}
R_3 : (x,y,z) & \mapsto & (x,y,xy-z) \\
\tau_{12} : (x,y,z) & \mapsto & (y,x,z) \\
\tau_{23} : (x,y,z) & \mapsto & (x,z,y) \\
\end{array}\end{equation}
Note that $\Aut(\Pi)$ is generated by $r,s,t$. 
Since the automorphisms of $\bA^3$ in the right side of the table act on the \emph{left}, $\Tr_*$ is an \emph{anti-homomorphism}.

The trace of the Higman invariant defines a map $\tau : X_{\SL_2}\ra\bA^1_\bZ$ sending $\varphi$ to $\tr\varphi([a,b])$. The induced map $T : \bA^3_\bZ\ra\bA^1_\bZ$ fits into a commutative diagram
\[\begin{tikzcd}
	X_{\SL_2}\ar[r,"\Tr"]\ar[rd,"\tau"] & \bA^3_\bZ\ar[d,"T"] \\
	 & \bA^1_\bZ
\end{tikzcd}\]
In coordinates, $T$ is given by $T(x,y,z) = x^2 + y^2 + z^2 - xyz - 2$. Because $\Out(\Pi)$ preserves the set of conjugacy classes of $[a,b]^{\pm 1}$, it preserves the fibers of the map $T$.\footnote{Note that a matrix in $\SL_2$ has the same trace as its inverse.}
The slice $T^{-1}(-2)\subset\bA^3_\bZ\cong X_{\SL_2}$ is described by the \emph{Markoff equation}\footnote{This equation first appeared in the work of Markoff \cite{Mar79,Mar80} on Diophantine approximation, but has since appeared in a variety of other contexts (see \cite{Bom07}).}
$$\bX : x^2 + y^2 + z^2 - xyz = 0$$
Through its interpretation as a subvariety of $X_{\SL_2}$, it is shown in \cite[Prop 5.2.17]{Chen21} that $\Tr$ induces a bijection
$$\Epi(\Pi,\SL_2(\bF_p))_{\tau = -2}/\GL_2(\bF_p)\rightiso \bX^*(p) := \bX(\bF_p) - \{(0,0,0)\}$$
The action of $\GL_2(\bF_p)$ on $\SL_2(\bF_p)$ induces the full action of $\Aut(\SL_2(\bF_p))$, and hence the left hand side is the fiber of
$$\cM_p := \cM(\SL_2(\bF_p))_{\bC,\tau = -2}/\Out(\SL_2(\bF_p))$$
over $\cM(1)$, where $\cM(\SL_2(\bF_p))_{\bC,\tau = -2}$ denotes the open and closed substack of $\cM(\SL_2(\bF_p))_\bC$ parametrizing covers whose Higman invariants have trace $-2\in\bF_p$.\footnote{The moduli interpretation for the quotient $\cM_p$ is described in \S\ref{sss_absolute_moduli}.} Thus, the connectivity of $\cM_p$ is equivalent to the transitivity of the $\Out^+(\Pi)$-action on $\bX^*(p)$. Since the natural reduction map
$$\bX(\bZ)\ra\bX(\bF_p)$$
is $\Out(\Pi)$-equivariant, the transitivity on $\bX^*(p)$ implies that this reduction map is \emph{surjective}. In this case we say $\bX$ satisfies \emph{strong approximation at $p$}. We can view ``strong approximation'' as a measure of abundance of the integral points on $\bX$.\footnote{Note that $\bX$ lies in the ``critical range'', where the number of variables equals the dimension, where the Lang-Vojta and Manin conjectures do not predict anything about the behavior of its integral points. This is analogous to the critical range occupied by elliptic curves in the context of the Diophantine geometry of curves.} A conjecture of Baragar, reaffirmed by Bourgain, Gamburd, and Sarnak, states:

\begin{conj}[{Baragar \cite{Bar91}, Bourgain, Gamburd, Sarnak \cite{BGS16arxiv}}]\label{conj_BBGS} The Markoff equation $\bX$ satisfies strong approximation at all primes $p$. Equivalently, $\Out(\Pi)$ acts transitively on $\bX^*(p)$ for all primes $p$.	
\end{conj}

By the explicit equations \eqref{eq_explicit_action}, this action of $\Out^+(\Pi)$ on $\bX^*(p)$ is more amenable to study than its abstract combinatorial action on equivalence classes of surjections $\Pi\ra\SL_2(\bF_p)$. In \cite{BGS16arxiv}, Bourgain, Gamburd, and Sarnak used tools from arithmetic geometry to study this action. They were able to show:

\begin{thm}[Bourgain, Gamburd, Sarnak]\label{thm_BGS} Let $\bE_\bgs$ denote the ``exceptional'' set of primes for which $\Out(\Pi)$ fails to act transitively on $\bX^*(p)$.
\begin{enumerate}
\item For any $\epsilon > 0$, $\#\{p\in\bE_\bgs\;|\; p\le x\} = O(x^\epsilon)$,
\item For every prime $p$, there is a large orbit $\cC(p)$ such that $|\bX^*(p) - \cC(p)|\le p^\epsilon$ for $p$ large.
\item Every orbit has cardinality\footnote{This was subsequently improved to $\gg (\log p)^{7/9}$ in \cite{KMSV20}.}
 $\gg (\log p)^{1/3}$.
\end{enumerate}
\end{thm}
Note that if the logarithmic lower bound in (c) could be promoted to a polynomial lower bound, then together with (b), we would have proven the conjecture for all but finitely many primes $p$. A polynomial lower bound was supplied in \cite{Chen21}:

\begin{thm}[C. \cite{Chen21}]\label{thm_congruence} Every $\Out^+(\Pi)$-orbit on $\bX^*(p)$ has cardinality $\equiv 0\mod p$.
\end{thm}
For context, note that \cite[Prop. 5.3.3]{Chen21}
$$|\bX^*(p)| = \left\{\begin{array}{ll} p(p + 3) & p\equiv 1\mod 4 \\ p(p-3) & p\equiv 3\mod 4\end{array}\right.$$
This congruence implies that the degree of $\cM_p$ over $\cM(1)$ is divisible by $p$. An example of this was already seen in \S\ref{ss_noncongruence_examples}, where we noted that the components of $\cM(\SL_2(\bF-7))_\bC$ classifying covers with ramification index 14 had degree $d = 28$ ($\equiv 0\mod 7$).

\begin{remark} While methods of Bourgain, Gamburd, and Sarnak were largely arithmetic and combinatorial, the methods used to prove Theorem \ref{thm_congruence} involved algebraic geometry over $\bC$; the key idea was to relate the size of any orbit to the degree of a certain line bundle on a component of the compactified stack $\cAdm(\SL_2(\bF_p))_{\bC,\tau = -2}$. The appearance of $p$ in the modulus of the congruence comes from the fact that any admissible $\SL_2(\bF_p)$-cover with Higman invariant of trace $-2$ must have ramification index $2p$.\footnote{Group theoretically, this means that any trace $-2$ element of $\SL_2(\bF_p)$ which can appear as the commutator of a generating pair must have order $2p$.}
\end{remark}

\begin{cor}\label{cor_congruence} The theorem immediately implies:
	\begin{enumerate}
		\item For $p\gg 0$, $\Out^+(\Pi)$ acts transitively on $\bX^*(p)$,
		\item For $p\gg 0$, $\cM_p$ is connected.
		\item For $p\gg 0$, the preimage of $\cM_p$ in $\cM(\SL_2(\bF_p))_\bC$ has exactly two components, which are exchanged by $\Out(\SL_2(\bF_p))$.
		\item For $p\gg 0$, $\bX$ satisfies strong approximation mod $p$.
	\end{enumerate}
	More precisely, the statements hold for all $p$ not in the finite set $\bE_\bgs$.
\end{cor}

The connectivity result unlocks a number of corollaries:
\begin{cor} For primes $p\notin\bE_\bgs$, the genus of the coarse scheme $M_p$ of $\cM_p$ is
$$\genus(M_p) = \frac{1}{12}p^2 + O(p^{3/2})$$
For $p\ge 13$, $p\notin \bE_\bgs$, $\genus(M_p)\ge 2$; Thus for such $p$, for any number field $K$, only finitely many elliptic curves over $K$ admit a geometrically connected $\SL_2(\bF_p)$-cover with ramification index $2p$ defined over $K$.
\end{cor}
\begin{proof} See \cite[\S5.6]{Chen21}, where an exact genus formula is also given. The finiteness result is a consequence of Falting's theorem (formerly Mordell's conjecture).
\end{proof}

\begin{cor} Let $M_p'$ denote the image of $M_p$ in $\cM(\PSL_2(\bF_p))/\Out(\PSL_2(\bF_p))$. For a density 1 set of primes, the monodromy group of $M_p'$ over $M(1)$ is either the full alternating or symmetric group. In particular, for such primes, $M_p'$ and $M_p$ are noncongruence.
\end{cor}
\begin{proof} See \cite[\S5.6]{Chen21}. The identification of the monodromy group follows from an analysis of Meiri-Puder \cite{MP18}, which involves the classification of finite simple groups. Their methods assume transitivity (equivalently, strong approximation) as well as a mild technical condition, which is satisfied for a density 1 set of primes. By the monodromic criterion (Theorem \ref{thm_monodromic_criterion}), this implies that $M_p'$ is noncongruence, and hence $M_p$ is as well.
\end{proof}

Since the methods of Bourgain, Gamburd, and Sarnak are effective, using the polynomial lower bound of Theorem \ref{thm_congruence}, it is possible to find an explicit upper bound for the primes in $\bE_\bgs$. This was done recently by Eddy, Fuchs, Litman, Martin, and Tripeny:
\begin{thm}[{Eddy, Fuchs, Litman, Martin, Tripeny \cite{EFLMT23}}] If $p > 3.448\cdot 10^{392}$, then $p\notin \bE_\bgs$. In other words, for such primes, Conjecture \ref{conj_BBGS} and the statements in the above corollaries hold.
\end{thm}

\begin{remark} The methods of this section indicate that for finite simple groups $G$ of Lie type, the combinatorial group-theoretic action of $\Out^+(\Pi)$ on $\Epi^\ext(\Pi,G)$ can be fruitfully understood as both a topological monodromy action coming from Hurwitz spaces over $\bC$, as well as an arithmetic-geometric action on an appropriate character variety over a finite field. At least in the case of $G = \SL_2(\bF_p)$, the topological perspective leads to rigidity results as in Theorem \ref{thm_congruence}, and the arithmetic perspective leads to asymptotic results as in Theorem \ref{thm_BGS}; combined, they provide a relatively complete picture of the situation. It would be interesting to work out similar results for more general groups $G$, as well as for more general surface groups $\Pi$.
\end{remark}

\begin{remark} The discussion above only covers the case of trace invariant $\tau = -2$. For more general trace invariants, analogs of Theorem \ref{thm_BGS} should still hold, though the rigidity result Theorem \ref{thm_congruence} becomes weaker, see \cite[\S5.4]{Chen21}. Nonetheless, computational data indicates that unlike the situation for $\PSU_3(\bF_q)$, the components of $\cM(G)^\abs$ for $G$ of type $\PSL_2(\bF_q)$ are remarkably uniform. For all such groups $G$ of order $\le |\J_1| = 175560$, the components are classified by the trace invariant\footnote{Here, one should define the trace invariant of a surjection $\varphi : \langle a,b\rangle = \Pi\ra\PSL_2(\bF_p)$ as the trace of the commutator $[\widetilde{\varphi(a)},\widetilde{\varphi(b)}]$, where $\widetilde{\varphi(a)},\widetilde{\varphi(b)}$ are arbitrary lifts of $\varphi(a),\varphi(b)$ to the central extension $\SL_2(\bF_q)$, as described in \S\ref{ss_noncongruence_examples}.}, they all have either alternating or symmetric monodromy over $\cM(1)$, and are all noncongruence. It would be interesting to prove or disprove whether these phenomena hold for all $q$. Note that the ``$T$-classification conjecture'' of McCullough and Wanderley \cite{MW13} asserts that the trace invariant should in general classify all components of $\cM(\PSL_2(\bF_q))$. Given an understanding of connected components, the techniques of \cite{MP18} could likely shed light on the question of monodromy, which by Theorem \ref{thm_monodromic_criterion} would likely imply the noncongruence property.

\end{remark}

\subsection{Fourier coefficients of noncongruence modular forms: unbounded denominators}
\begin{defn} Let $\Gamma\le\SL_2(\bZ)$ be a finite index subgroup. A modular form of weight $k$ for $\Gamma$ is a holomorphic function $f : \cH\ra\bC$ which satisfies, for any $\gamma = \spmatrix{a}{b}{c}{d}\in\SL_2(\bZ)$, the conditions:
\begin{enumerate}
	\item $f(\gamma \tau) = (cz+d)^k f(\tau)$ for all $\tau\in\cH$, and
	\item $f(\tau)$ is bounded as $\im(\tau)\ra \infty$.
\end{enumerate}
We say that $f$ is a \emph{congruence} modular form if it is a modular form for a congruence subgroup of $\SL_2(\bZ)$. Otherwise, it is \emph{noncongruence}. Since $\Gamma$ is finite index, any modular form for $\Gamma$ is invariant under $\spmatrix{1}{n}{0}{1}$ for some integer $n\ge 1$, and hence we have a \emph{Fourier expansion}
$$f(\tau) = \sum_{n\ge 0} a(n)q^n\quad\text{where $q := e^{2\pi i\tau/n}.$}$$
For a subring $R\subset\bC$, the space of weight $k$ modular forms for $\Gamma$ with Fourier coefficients in $R$ is denoted $M_k(\Gamma,R)$. The subspace of \emph{cusp forms} is defined by the condition that $f(\tau)\ra 0$ as $\im(\tau)\ra\infty$ (i.e., $f$ vanishes at all cusps), and is denoted by $S_k(\Gamma,R)$.
\end{defn}

A crucial element in the theory of congruence modular forms is the effectivity of the Hecke operators, the most important of which are the operators $T_p$, associated to primes $p$ \cite[\S5]{DS06}, which act on the spaces $M_k(\Gamma) := M_k(\Gamma,\bC)$. By contrast, it is known that for a noncongruence subgroup $\Gamma$, the action of $T_p$ on $S_k(\Gamma)$ essentially factors through its action on the subspace $S_k(\Gamma^c)$, where $\Gamma^c$ denotes the congruence closure \cite{Berg94}. For congruence $\Gamma$, properties of the Hecke operators guarantee that $M_k(\Gamma)$ has a basis whose Fourier coefficients are algebraic \emph{integers}. On the other hand, a folklore conjecture states:

\begin{conj}[Unbounded denominators conjecture] Let $f$ be a noncongruence modular form with Fourier coefficients in $\Qbar$. Then it's Fourier coefficients have \emph{unbounded denominators}. In other words, $nf$ does not have (algebraically) integral coefficients for any positive integer $n$.
\end{conj}

The conjecture was recently proved by Calegari, Dimitrov and Tang using techniques from Nevanlinna theory \cite{CDT21}, though their proof does not give any indication for which primes can appear in the unbounded denominators of a noncongruence form $f$. Below, we describe a framework for studying questions of this type.

Let $M := \Gamma\bs\cH$ be a noncongruence modular curve over a number field $K$, and let $\ff : M\ra M(1)_K$ be the forgetful map. Possibly passing to a finite extension, we may assume that $M = \Spec M_0(\Gamma,K)$ \cite[Cor 5.3.4]{Chen18}. Let $A$ be a subring of $\Qbar\ls{q^{1/\infty}} := \varinjlim_n\Qbar\ls{q^{1/n}}$ which contains $\bZ\ls{q}\otimes K$. Consider diagrams of the following type:
\begin{equation}\label{eq_qexp}
\begin{tikzcd}
	& & M\ar[d,"\ff"] \\
	\Spec A\ar[rr,"\Tate(q)_A"']\ar[rru,"c"] & & M(1)_K
\end{tikzcd}
\end{equation}
where $\Tate(q)_A$ denotes the map given by the Tate curve over $A$ \cite[\S5.1]{Chen18}. Writing $M(1)_K\cong \Spec K[j]$, this map sends $j$ to the Fourier expansion $j(q)$, viewed as an element of $\bZ\ls{q}\subset A$. A choice of a map $c$ making the diagram commute amounts to choosing a cusp of $M$, and the induced map on rings $c : M_0(\Gamma)\ra\bC\ls{q^{1/\infty}}$ is exactly Fourier expansion at the corresponding cusp. Of course, such maps $c$ do not always exist; such a map exists if and only if the pullback $M\times_{M(1)_K} A$ admits a section over $A$. If $A = \Qbar\ls{q^{1/\infty}}$, then such maps always exist since $M\times_{M(1)_K} A$ is the spectrum of a finite algebra over the algebraically closed field $\Qbar\ls{q^{1/\infty}}$. A more interesting case is $A = \bZ\ls{q^{1/\infty}}\otimes K$, in which case such a map $c$ exists if and only if all modular functions for $\Gamma$ have \emph{bounded denominators} at the cusp corresponding to $c$.\footnote{Note that elements of $\bZ\ls{q^{1/\infty}}\otimes K$ have coefficients with \emph{bounded denominators}.}

Suppose now that $R\subset K$ is a DVR with mixed characteristic $(0,p)$, and that the cover $\ff : M\ra M(1)_K$ admits a finite model $\ff_R : M_R\ra M(1)_R$ which is \'{e}tale over the image of $\Spec R\ls{q}\ra M(1)_R$ given by the Tate curve (the generic point maps to the generic point of the generic fiber, and the closed points map to the generic point of the special fiber). Then, by Abhyankar's lemma, there exists a finite \'{e}tale extension $R'$ of $R$ and an integer $e\ge 1$ coprime to $p$ such that the map $\Spec R'\ls{q^{1/e}}\ra\Spec R\ls{q}\ra M(1)_R$ admits a lift to $M_R$ via $\ff_R$ \cite[Cor 5.4.3]{Chen18}. If $K' := \Frac(R')$, then tensoring with $K'$ would give a diagram of type \eqref{eq_qexp} for the bounded denominators ring $A = R'\ls{q^{1/e}}\otimes K'$, which is to say that modular forms for $\Gamma$ have \emph{bounded denominators at $p$}. If $M$ is a component of $M(G)$, then taking $R$ to be a localization of $\cO_K[1/|G|]$, we find that

\begin{thm} Modular forms for any component of $M(G)$ \emph{have bounded denominators at all primes not dividing $|G|$}.
\end{thm}
\begin{proof} See \cite[Thm 5.4.1]{Chen18}.	
\end{proof}

In general, this result is far from sharp. In \cite{FF20}, Fiori and Franc showed that modular forms for the degree 7 component of $M(\PSU_2(\bF_7))$ \emph{only have unbounded denominators at $p = 7$}.\footnote{The corresponding subgroup is conjugate to Fiori and Franc's ``$G_1$'' \cite{FF20}.} In particular, even though $|\PSU_2(\bF_7)| = 168$ is divisible by $6$, modular forms for this component do not have unbounded denominators at 2 or 3. They also showed the same result for the unique component $M\subset M(\PSL_3(\bF_4))/\Out(\PSL_3(\bF_4))$ with degree 7 over $M(1)$, even though $|\PSL_3(\bF_4)| = 20160 = 2^6 \cdot 3^2 \cdot 5\cdot 7$ is also divisible by $2,3$, and $5$.\footnote{The corresponding subgroup is conjugate to ``$H_1$'' in the notation of \cite{FF20}.} Interestingly, in both cases the component parametrizes admissible covers with ramification index $e = 7$. One is led to ask:

\begin{question} What is the relationship between the primes of unbounded denominators and the ramification indices of the parametrized objects?
\end{question}

The analysis above links unbounded denominators to the behavior of suitable integral models of the cover $\ff : M\ra M(1)$. If $\ff$ has good reduction at $p$, one easily deduces bounded denominators at $p$. In general, if the reduction of $\ff$ is \'{e}tale above the generic point of a suitable component of the special fiber of an integral model, then Abhyankar's lemma would again imply bounded denominators.

\end{document}